\documentclass{article}
\usepackage[utf8]{inputenc}
\usepackage{amsmath,amssymb,mathtools,amsthm}
\usepackage{geometry}
\usepackage{enumerate}
\usepackage{graphicx}
\usepackage{fancyhdr}
\usepackage{bbm}
\usepackage[mathscr]{euscript}
\usepackage{mathrsfs}
\usepackage{caption}
\usepackage[T1]{fontenc}
\usepackage{xcolor}
\usepackage{authblk}

\definecolor{bleu_sombre}{rgb}{0,0,0.6}  \definecolor{rouge_sombre}{rgb}{0.8,0,0}\definecolor{vert_sombre}{rgb}{0,0.6,0}
\usepackage[plainpages=false,colorlinks,linkcolor=bleu_sombre,
citecolor=rouge_sombre,urlcolor=vert_sombre,breaklinks]{hyperref}

\usepackage{cleveref}

\makeatletter

\@addtoreset{equation}{section}
\makeatother

\theoremstyle{plain}
\newtheorem{definition}{Definition}[section]

\newtheorem{theorem}[definition]{Theorem}
\newtheorem{remark}[definition]{Remark}
\newtheorem{corollary}[definition]{Corollary}
\newtheorem{proposition}[definition]{Proposition}
\newtheorem{lemma}[definition]{Lemma}

\newcommand{\dd}{\mathrm d}
\newcommand{\R}{\mathbb R}
\newcommand{\N}{\mathbb N}
\newcommand{\SE}{\mathbb S_E}
\newcommand{\C}{\mathbb C}
\newcommand{\Z}{\mathbb Z}
\newcommand{\T}{\mathbb T}
\newcommand{\Op}{\mathsf{Op}_h^{B}}

\newcommand{\Opsansh}{\mathsf{Op}_1^{B}}
\newcommand{\Opgamma}{\mathsf{Op}_{\frac{h}{\sqrt{\gamma_h}}}^{B}}

\begin{document}

\author{Julien LECHAUX}

%\affil[$\star$]{Laboratoire de Math\'ematiques Jean Leray, UMR CNRS 6629, Nantes Universit\'e}

\title{Long-time dynamics for the magnetic Schrödinger equation on the torus}
\date{\today}
\maketitle
\begin{abstract}
We investigate time-dependent semiclassical measures associated with the magnetic Schr\"odinger flow on the flat two-dimensional torus in long-time regimes. Under a geometric nonvanishing assumption on an effective magnetic force, we identify two thresholds and prove a dynamical form of quantum unique ergodicity for such systems. If the time scale $\tau_h \gg h^{-1/2}$, then the configuration marginal is the normalized Lebesgue measure on $\mathbb T^2$. In addition, if $\tau_h \gg h^{-1}$, then the momentum marginal is, conditionally on each regular level set, the normalized invariant measure on that curve.
\end{abstract}

\tableofcontents

\section{Introduction}\label{s:introduction}

This paper studies long-time equidistribution phenomena for magnetic Schr\"odinger evolutions on the two-dimensional torus \(\T^2=\R^2/\Z^2\). Throughout this paper, we work within a semiclassical framework. More precisely, we consider a sequence $(h_n)_{n \in\N}$ of positive semiclassical parameters such that $h_n\to 0^+$. A \emph{magnetic field} on \(\T^2\) is a smooth real-valued closed \(2\)-form
\[
B(x)\,\dd x_1\wedge \dd x_2,
\]
which we identify with the function \(B\in \mathcal C^\infty(\T^2,\R)\). We assume that its total flux satisfies
\begin{equation}\label{e:fluxcondition}
    \widehat B_0:= \int_{\T^2}B(x)\,\dd x_1\dd x_2 \in 2\pi\Z\setminus \{0\}.
\end{equation}
We focus on magnetic Schr\"odinger-type equations of the form
\begin{equation}\label{e:intro-general-evolution}
\left\{
\begin{aligned}
ih_n\partial_t u_n(t,x)
&=
\widehat H_{h_n}\,u_n(t,x),\\[0.3em]
u_n(t=0,x)
&=
u_n^{(0)}(x), 
\end{aligned}
\right.
\qquad
(t,x)\in\R\times\T^2,
\end{equation}
where the initial data are assumed to be normalized, i.e., \(\lVert u_n^{(0)} \rVert_{L^2(\T^2,L)}=1\).
The operator \(\widehat H_{h_n}\) is a semiclassical magnetic Hamiltonian of
the form
\begin{equation}\label{e:definition-widehatH}
\widehat H_{h_n}
=
\mathsf{Op}_{h_n}^{B}
\big(
H(\xi)+h_nR_{h_n}(x,\xi)
\big),
\end{equation}
where \(\mathsf{Op}_{h_n}^{B}\) denotes the toral magnetic Weyl quantization recalled in Section \ref{s:Magnetic field and quantization}. Despite the notation, this quantization depends only on the magnetic flux \(\widehat B_0\). The precise assumptions on the Hamiltonian \(H\) and on the perturbation \(R_{h_n}\) are stated in Section~\ref{ss:magnetic-operator}. In particular, \(H\) is a real-valued elliptic Hamiltonian depending only on \(\xi\), and \(\mathscr{R}\) will denote the principal symbol of \(R_{h_n}\) as \(h_n\to0\) and it will depend both on $x$ and $\xi$. These assumptions are precisely those needed to ensure that the operator \(\widehat H_{h_n}\) has the appropriate functional-analytic properties and is self-adjoint; see Proposition~\ref{prop:esa}. Due to the magnetic field, the operator \(\widehat H_{h_n}\) acts on the magnetic Hilbert space \(L^2(\T^2,L)\) of twisted \(L^2\)-functions associated with the flux condition \eqref{e:fluxcondition} (see Section \ref{ss:Magnetic-translations and-magnetically-periodic-functions}). 

The magnetic Laplacian is the main motivating example. If \(A^{\mathrm{per}}\) is a smooth real-valued periodic vector potential and \(V\) is a smooth real-valued potential, then the magnetic Laplacian $\mathcal{L}^B$ associated with
\begin{equation}\label{e:mag-field}
B(x) = \widehat B_0+\dd A^{\mathrm{per}}(x),
\end{equation}
is included in the class of operators \eqref{e:definition-widehatH}, namely, one has
\begin{equation}\label{e:laplacien-mag}
h_n^2(\mathcal{L}^B+V) 
= 
\mathsf{Op}_{h_n}^{B}\bigg(\big|\xi-h_nA^{\mathrm{per}}(x)\big|^2+h_n^2V(x)\bigg),
\end{equation}
where $\lvert \cdot \rvert$ denotes the Euclidean norm on $\R^2$. 

A standing assumption throughout the paper is that there exists a compact
energy window \([E_1,E_2]\subset(0,+\infty)\) containing no critical value of
the principal Hamiltonian \(H\):
\begin{equation}\label{e:fenetre-energie}
\forall \xi\in
\Omega_{E_1,E_2}
:=
\{\xi\in\R^2:\ E_1\le H(\xi)\le E_2\},\qquad \nabla H(\xi)\neq 0.
\end{equation}
For every \(E\in[E_1,E_2]\), we denote by
\begin{equation}\label{e:energy-surface-intro}
\mathbb S_E
:=
\{\xi\in\R^2:\ H(\xi)=E\}
\end{equation}
the corresponding energy level. By the ellipticity of \(H\) and \eqref{e:fenetre-energie}, each \(\mathbb S_E\) is a smooth compact one-dimensional submanifold of \(\R^2\). We further assume that \(\mathbb S_E\) is connected for every \(E\in[E_1,E_2]\) so that each \(\mathbb S_E\) is diffeomorphic to the circle \(\mathbb S^1\). We also assume that the initial data are spectrally localized in this window:
\begin{equation}\label{e:intro-simplified-spectral-window}
u_n^{(0)}
=
\mathbbm{1}_{[E_1,E_2]}
\left(
\widehat H_{h_n}
\right)
u_n^{(0)}
+
o_{L^2(\T^2,L)}(1) \quad \text{as} \quad h_n\rightarrow 0^+.
\end{equation}
Our first result is stated under a simple sufficient condition.
\begin{theorem}\label{t:intro-simplified-equidistribution}
Suppose that the above assumptions hold and that the Hessian of \(H\) is positive definite on \(\Omega_{E_1,E_2}\). Then, there exists an explicit constant (see Corollary \ref{c:simplified-x-equidistribution}) $C_{H,\widehat B_0,E_1,E_2}>0$, such that, if 
\begin{equation}\label{e:intro-smallness-of-R}
    \|\partial_x \mathscr{R}\|_{L^\infty(\T^2\times\Omega_{E_1,E_2})}
\le C_{H,\widehat B_0,E_1,E_2},
\end{equation}
if \((u_n)_{n\in\N}\) is a sequence of normalized solutions to \eqref{e:intro-general-evolution} satisfying \eqref{e:intro-simplified-spectral-window}, and if \((\tau_{h_n})_{n\in\N}\) is a positive sequence satisfying
\[
\tau_{h_n}\gg h_n^{-1/2},
\]
then, for every \(\psi\in \mathcal{C}_c^\infty(\R)\) and every \(a\in \mathcal{C}^\infty(\T^2)\), one has 
\begin{equation}\label{e:intro-simplified-equi}
\lim_{n\to +\infty}
\int_\R
\psi(t)
\int_{\T^2}
a(x)\,
|u_n(t\tau_{h_n},x)|^2
\,\dd x\,\dd t
=
\left(
\int_\R \psi(t)\,\dd t
\right)
\left(
\int_{\T^2} a(x)\,\dd x
\right).
\end{equation}
\end{theorem}

Equivalently, if one introduces the positive Radon measures
\begin{equation}\label{e:distribution-intro-config-space}
\nu_{h_n}(\tau_{h_n}) := |u_n(t\tau_{h_n},x)|^2\,\dd x\,\dd t \qquad\text{on }\T^2\times\R,
\end{equation}
then Theorem~\ref{t:intro-simplified-equidistribution} says that
\begin{equation}\label{e:intro-simplified-measure}
\nu_{h_n}(\tau_{h_n})
\rightharpoonup
\dd x\otimes \dd t
\qquad
\text{in }\mathcal D'(\T^2\times\R).
\end{equation}
In particular, for almost every time \(t\), the configuration-space marginal associated with any accumulation point of \eqref{e:distribution-intro-config-space} is the uniform measure \(\dd x\). The threshold \(h_n^{-1/2}\) is essential. In general, for observation scales
\(\tau_{h_n}\lesssim h_n^{-1/2}\), one cannot expect a statement of the form \eqref{e:intro-simplified-equi} for arbitrary initial data. In those regimes, nontrivial concentration mechanisms may persist near periodic directions of the underlying classical flow. A precise description of these shorter scales is given later in Section \ref{s:periodic-orbits}. The smallness assumption on \(\partial_x\mathscr R\) is a convenient sufficient condition ensuring the validity of the effective non-vanishing condition used in the general theorem. The latter condition is geometric in nature and will be stated precisely in Theorem \ref{t:intro-main-structure}. In the case of the magnetic Laplacian, it takes a simple form and is satisfied, for instance, as soon as the magnetic field is positive, namely \(B>0\) on \(\T^2\).

\begin{theorem}[Magnetic Laplacian]\label{c:intro-nonsemiclassical-magnetic}
Assume that \(B>0\) and let \((v_n^{(0)})_{n\in\N}\) be a normalized sequence of initial data. We consider the non-semiclassical magnetic Schr\"odinger evolution
\begin{equation}\label{e:evolution-non-semiclassique}
    v_n(t):=
e^{-it(\mathcal L^B+V)}\,v_n^{(0)}.
\end{equation}
Assume that the initial data are spectrally localized in the energy window \([E_1,E_2]\) in the sense of \eqref{e:intro-simplified-spectral-window}, with $H(\xi) = \lvert \xi\lvert^2$. Then, for every
\(\psi\in \mathcal{C}_c^\infty(\R)\) and every \(a\in \mathcal{C}^\infty(\T^2)\),
\begin{equation}\label{e:intro-nonsemiclassical-equi}
\lim_{n\to+\infty}
\int_\R
\psi(t)
\int_{\T^2}
a(x)\,
\left|
e^{-it(\mathcal L^B+V)}v_n^{(0)}(x)
\right|^2
\,\dd x\,\dd t
=
\left(
\int_\R\psi(t)\,\dd t
\right)
\left(
\int_{\T^2}a(x)\,\dd x
\right).
\end{equation}
\end{theorem}

This theorem corresponds to the time scale \(\tau_{h_n}=h_n^{-1}\) which is natural in this setting. Indeed, the evolution \eqref{e:evolution-non-semiclassique} solves the non-semiclassical magnetic Schr\"odinger equation
\[
i\partial_t v_n=(\mathcal L^B+V)v_n.
\]
It is related to the semiclassical equation by setting
\[
u_n(t):=v_n(h_nt)
=
e^{-ih_nt(\mathcal L^B+V)}v_n^{(0)}.
\]
Then \(u_n\) satisfies
\[
ih_n\partial_t u_n
=
h_n^2(\mathcal L^B+V)u_n.
\]
By \eqref{e:laplacien-mag}, this equation belongs to the class
\eqref{e:intro-general-evolution} and, taking \(\tau_{h_n}=h_n^{-1}\), one has 
\[
u_n(t \tau_{h_n})=v_n(t).
\]
The time scale \(\tau_{h_n}=h_n^{-1}\) is the Heisenberg scale in the semiclassical normalization used here. It lies far beyond the Ehrenfest regime, which in general corresponds to small logarithmic times of order \(|\log h_n|\) for semiclassical Egorov estimates. We shall discuss other long observation regimes, including
\[
h_n^{-1/2}\ll\tau_{h_n}\ll h_n^{-1}
\qquad\text{and}\qquad
h_n^{-1}\ll\tau_{h_n},
\]
later in the article.

\subsection{The semiclassical context}\label{ss:The-semiclassical-context}

As mentioned above, Theorem~\ref{t:intro-simplified-equidistribution} follows from a more general result, stated in terms of semiclassical measures. We now describe this framework and state the main
structure theorem of the article.

To simplify notation, we shall henceforth write \(h\) instead of \(h_n\), and \(u_h\) instead of \(u_n\). All limiting statements are understood along the fixed sequence \((h_n)_{n\in\N}\), possibly after extraction of a subsequence. Let \((\tau_h)_{h\to0^+}\) be a sequence of positive observation times. In order
to describe the long-time behavior of the solutions, we first consider the
configuration-space distributions
\begin{equation}\label{e:distribution-intro}
\nu_h(\tau_h)
:=
|u_h(t\tau_h,x)|^2\,\dd x\,\dd t
\qquad\text{on }\T^2\times\R.
\end{equation}
The compactness arguments recalled in
Section~\ref{s:semiclassicalmeasure} show that every accumulation point of
\((\nu_h(\tau_h))_{h\to0^+}\) is a positive Radon measure $\nu$ on
\(\T^2\times \R\). Such limits describe the macroscopic distribution of the
position densities along the rescaled evolution \(t\mapsto u_h(t\tau_h)\).

The semiclassical strategy consists in lifting these configuration-space
distributions to phase space. In the present magnetic setting, this is achieved
through time-dependent magnetic Wigner distributions that generalize \eqref{e:distribution-intro} and that are defined by 
\begin{equation}\label{e:Wigner-magnetic-distrib}
    W_h^B(\tau_h): a \in \mathcal{C}_c^\infty(\R \times T^*\T^2) \longmapsto \int_\R \big\langle \Op(a(t))u_h(t\tau_h), u_h(t\tau_h)\big\rangle_{L^2(\T^2,L)} \dd t, 
\end{equation}
where $(u_h)_{h\to 0^+}$ is a family of solutions to \eqref{e:intro-general-evolution}.

More precisely, Section~\ref{s:semiclassicalmeasure} shows that, after extraction of a subsequence, one obtains after taking $h \to 0^+$ a positive Radon measure
\begin{equation}\label{e:definition-WB}
W^B(\dd t,\dd x,\dd\xi)
=
\nu_t^B(\dd x,\dd\xi)\otimes \dd t
\qquad\text{on } T^*\T^2\times \R,
\end{equation}
where \((\nu_t^B)_{t\in\R}\) is a measurable family of probability measures on \(T^*\T^2\), defined for almost every \(t\). The family \((\nu_t^B)_{t\in\R}\) will be referred to as a time-dependent magnetic semiclassical measure. Moreover, under the spectral localization assumption \eqref{e:intro-simplified-spectral-window}, one has
\begin{equation}\label{e:intro-propriete-support}
\operatorname{supp}\nu_t^B
\subset
\T^2\times\Omega_{E_1,E_2}
\qquad\text{for a.e. }t\in\mathbb{R}.
\end{equation}
The corresponding configuration-space limit is recovered by projection:
\[
\nu_t:=(\pi_x)_*\nu_t^B,
\qquad
\pi_x:(x,\xi) \in T^*\T^2\mapsto x.
\]
In particular, the corresponding configuration-space distributions satisfy
\[
\nu_h(\tau_h)
\rightharpoonup
\nu_t(\dd x)\otimes\dd t
\qquad
\text{in }\mathcal D'(\T^2\times \R).
\]

When \(\tau_h\to+\infty\), the phase-space measures \(\nu_t^B\) are, for almost every \(t\), invariant under the Hamiltonian flow generated by the principal symbol \(H(\xi)\),
\begin{equation}\label{e:intro-def-flot}
\varphi_H^s(x,\xi)=(x+s\nabla H(\xi),\xi).
\end{equation}
This invariance alone does not exclude concentration near the periodic trajectories of \(\varphi_H^s\). To rule out these remaining concentration phenomena, we introduce a non-vanishing condition
along the periodic directions of the classical flow. Periodic directions are
indexed by primitive rank-one submodules
\(\Lambda\subset\Z^2\). We denote by \(\mathcal L_1\) the set of such
submodules. For \(\Lambda\in\mathcal L_1\), let
\(\mathfrak e_\Lambda\) be a primitive generator of \(\Lambda\) and set
\[
L_\Lambda:=|\mathfrak e_\Lambda|.
\]
The resonant set associated with \(\Lambda\) is
\[
E_{\Lambda^\perp\setminus\{0\}}
:=
\left\{
\xi\in\R^2:
\nabla H(\xi)\in\Lambda^\perp\setminus\{0\}
\right\}.
\]
Along these directions, the flow $\varphi_H^s$ is periodic. The two-microlocal analysis developed in Section~\ref{ss:Effective-dynamics-for-the-two-microlocal-lift} gives rise to the effective coefficient
\begin{equation}\label{e:intro-definition-G}
G_{\widehat B_0,\Lambda}(x,\xi)
:=
\operatorname{Hess}(H)(\xi)
\frac{\mathfrak e_\Lambda}{L_\Lambda}
\cdot
\left(
\widehat B_0\,\nabla H(\xi)^\perp
+
\partial_x\mathcal I_\Lambda(\mathscr{R})(x,\xi)
\right),
\end{equation}
where we use the convention $(v_1,v_2)^\perp = (-v_2,v_1)$, and where \(\mathcal I_\Lambda(\mathscr{R})\) denotes the \(\Lambda\)-average of \(\mathscr{R}\), defined in Section~\ref{s:decomposition}. Roughly speaking, it consists in keeping only the Fourier coefficients of $\mathscr{R}$ lying in $\Lambda$. For instance, in the case of the magnetic Laplacian, one has
\[
H(\xi)=|\xi|^2,
\qquad
\mathscr{R}(x,\xi)=-2\xi\cdot A^{\mathrm{per}}(x),
\]
and, on the resonant set \( E_{\Lambda^\perp\setminus\{0\}}\), the effective coefficient reduces to
\[
G_{\widehat B_0,\Lambda}(x,\xi)
=
-4
\left(
\xi\cdot\frac{\mathfrak e_\Lambda^\perp}{L_\Lambda}
\right)
\mathcal I_\Lambda(B)(x).
\]
Thus, in this particular case, the non-vanishing of
\(G_{\widehat B_0,\Lambda}\) is directly related to the non-vanishing
of the magnetic field averaged along the corresponding periodic direction. The main result of this article shows that the magnetic subprincipal structure removes any concentration phenomena under a suitable non-vanishing condition. It also yields a complete description of the limiting measures on the long time scales considered here.

\begin{theorem}\label{t:intro-main-structure}
Assume that the standing assumptions on \(H\) and \(R_h\) stated above hold. More precisely, assume that the ellipticity of \(H\), the regular energy condition  \eqref{e:fenetre-energie} and the spectral localization condition \eqref{e:intro-simplified-spectral-window} hold. Suppose also that for every $\Lambda\in\mathcal L_1$ and every \((x,\xi)\in \T^2\times
\big(
E_{\Lambda^\perp\setminus\{0\}}
\cap
\Omega_{E_1,E_2}
\big),
\)
\begin{equation}\label{e:intro-uniform-escape}
G_{\widehat B_0,\Lambda}(x,\xi)
\neq0.
\end{equation}

Let $W^B(\dd t,\dd x,\dd\xi)=\nu_t^B(\dd x,\dd\xi)\otimes\dd t$ be any time-dependent magnetic semiclassical measure associated with the rescaled evolution $(u_h(t\tau_h))_{h\to 0^+}$ as in \eqref{e:Wigner-magnetic-distrib}, and assume that $\tau_h\gg h^{-1/2}$. Then, there exists a measurable family of probability measures
\((\lambda_t)_{t\in\R}\) on \(\Omega_{E_1,E_2}\) such that
\begin{equation}\label{e:intro-product-form}
\nu_t^B(\dd x,\dd\xi)
=
\dd x\otimes\lambda_t(\dd\xi)
\qquad\text{for a.e. }t.
\end{equation}
Moreover, up to another extraction, there exists a probability measure $\lambda_0$ on $\Omega_{E_1,E_2}$ associated to the initial data in the sense that 
\begin{equation}\label{e:intro-mesure-lambda0}
    \forall a \in \mathcal{C}_c^\infty(\R^2), \qquad \left\langle \Op(a)u_h^{(0)}, u_h^{(0)}\right\rangle_{L^2(\T^2,L)} \xrightarrow[h \to 0^+]{} \int_{\R^2} a(\xi) \lambda_0(\dd \xi)
\end{equation}
and such that the following finer description of $\lambda_t$ holds.
\begin{enumerate}
\item[(i)]
\textbf{Intermediate long times.}
If $h^{-1/2}\ll\tau_h\ll h^{-1}$, then one has \(\lambda_t=\lambda_0\). More precisely, $\nu_t^B$ is of the form
\[
\nu_t^B(\dd x,\dd\xi)
=
\dd x\otimes\lambda_0(\dd\xi)
\qquad\text{for a.e. }t.
\]

\item[(ii)]
\textbf{The Heisenberg scale.}
Assume that $\tau_h = h^{-1}$. Then one has
\[
\nu_t^B(\dd x,\dd\xi)
=
\dd x\otimes (\Upsilon_H^t)_*\lambda_0(\dd \xi),
\]
where \(\Upsilon_H^t\) is the flow on momentum space generated by
\begin{equation}\label{e:intro-momentum-vector-field}
\widetilde X_H 
:=
-\widehat B_0
\nabla H(\xi)^{\perp}\cdot\partial_\xi.
\end{equation}

\item[(iii)]
\textbf{Beyond the Heisenberg scale.}
If $\tau_h\gg h^{-1}$, then letting $\sigma^B:=H_*\lambda_0$, one has 
\begin{equation}\label{e:intro-energy-disintegration}
\nu_t^B(\dd x,\dd\xi)
=
\dd x\otimes
\int_{E_1}^{E_2}
\mu_{E}(\dd\xi)\,\sigma^B(\dd E)
\qquad\text{for a.e. }t,
\end{equation}
where \(\mu_E\) denotes the unique \(\Upsilon_H^t\)-invariant probability measure on $\mathbb S_E$. Moreover, one has for every $f \in \mathcal{C}^\infty_c(\R)$ 
\[
\int_{E_1}^{E_2}f(E)\sigma^B(\dd E) = \int_{\R^2} f(H(\xi)) \lambda_0(\dd \xi) = \lim_{h \to 0^+}\left\langle f(\widehat H_h)u_h^{(0)}, u_h^{(0)}\right\rangle_{L^2(\T^2,L)}.
\]
\end{enumerate}
\end{theorem}

The spectral localization condition \eqref{e:intro-simplified-spectral-window} can be read at the level of the initial momentum measure. Namely, it implies that 
\[
\lambda_0(\Omega_{E_1,E_2})=1, \qquad \operatorname{supp}(\lambda_0)\subset \Omega_{E_1,E_2}, 
\]
and therefore the initial energy distribution $\sigma^B = H_*\lambda_0$ is a probability measure supported in $[E_1,E_2]$. Theorem~\ref{t:intro-main-structure} reveals two distinct long-time thresholds. The first threshold occurs at the scale $h^{-1/2}$: once $\tau_h\gg h^{-1/2}$, the configuration variable becomes equidistributed. At scales $\tau_h \le h^{-1/2}$, such an equidistribution statement fails. Concentration effects may persist near periodic directions of the classical flow, and the limiting measures can be described in terms of the measure arising from the initial data. This shorter-time analysis is carried out in Section \ref{s:periodic-orbits}. The second threshold occurs at the scale $h^{-1}$: once $\tau_h\gg h^{-1}$, the remaining momentum dynamics equidistributes along the regular energy curves of $H$. The non-vanishing condition \eqref{e:intro-uniform-escape} expresses an escape mechanism for the effective dynamics near each periodic direction. For each fixed \(\Lambda\), this condition yields a positive lower bound on \(|G_{\widehat B_0,\Lambda}|\) on the relevant compact resonant set. It can be viewed as a directional curvature condition: on the resonant set associated with \(\Lambda\), the leading term of \(G_{\widehat B_0,\Lambda}\) involves the Hessian of \(H\) in the direction \(\mathfrak e_\Lambda\). This is why the positive definiteness assumption on \(\operatorname{Hess}(H)\), together with the smallness assumption \eqref{e:intro-smallness-of-R} on \(\partial_x\mathscr R\), implies \eqref{e:intro-uniform-escape}: the term proportional to \(\widehat B_0\) in \eqref{e:intro-definition-G} dominates the contribution involving \(\partial_x\mathcal I_\Lambda(\mathscr R)\). This mechanism is explained more precisely in Section~\ref{s:x-regularity}. Therefore, Theorem~\ref{t:intro-simplified-equidistribution} follows from Theorem~\ref{t:intro-main-structure}. Likewise, Theorem~\ref{c:intro-nonsemiclassical-magnetic} is obtained from Theorem~\ref{t:intro-main-structure} in the magnetic Laplacian case.

\begin{corollary}[Eigenfunctions]
\label{c:intro-eigenfunctions}
Assume the standing assumptions on \(H\) and \(R_h\). Let \(E\in(0,+\infty)\) be a regular value of \(H\), and assume that, for every \(\Lambda\in\mathcal L_1\) and every
\( (x,\xi)\in \T^2\times \bigl( E_{\Lambda^\perp\setminus\{0\}} \cap \SE
\bigr),
\)
one has
\[
G_{\widehat B_0,\Lambda}(x,\xi)\neq0.
\] Let \((u_h)_{h\to0^+}\) be a normalized family of eigenfunctions of \(\widehat H_h\),
\[
\widehat H_hu_h=E_hu_h,
\qquad
E_h\longrightarrow E.
\]
Then, for every \(a\in \mathcal{C}_c^\infty(T^*\T^2)\), one has
\[
\lim_{h\to0^+}\left\langle
\mathsf{Op}_h^B(a)u_h,u_h
\right\rangle_{L^2(\T^2,L)}
=
\int_{\T^2\times \SE}
a(x,\xi)\,\dd x\,\mu_E(\dd\xi),
\]
where \(\mu_E\) is uniformly equidistributed on \(\SE\).
In particular, for every \(a\in \mathcal{C}^\infty(\T^2)\), one has
\[
\lim_{h\to0^+}
\int_{\T^2}
a(x)|u_h(x)|^2\,\dd x
=
\int_{\T^2}a(x)\,\dd x.
\]
\end{corollary}
This stationary consequence may be viewed as a quantum unique ergodicity-type statement in a magnetic and completely integrable setting.

\subsection{Relation to previous works}

In summary, the above results describe regularization and equidistribution phenomena for quantum systems associated with completely integrable classical dynamics, where the relevant rigidity mechanism is created by a magnetic subprincipal effect. Such results pertain to the broader field where one aims at describing the long time dynamics of Schrödinger equations depending on the underlying classical Hamiltonian dynamics. We will now briefly describe other results in that direction.

\subsubsection*{Semiclassical measures and quantum ergodicity}

The Quantum Ergodicity Theorem of \v{S}hnirelman~\cite{Shnirelman1974}, Zelditch~\cite{Zelditch1987} and Colin de Verdi\`ere~\cite{ColindeVerdiere1985} asserts that, on a compact Riemannian manifold whose geodesic flow is ergodic, a density-one subsequence of Laplace eigenfunctions equidistributes in phase space. The stronger Quantum Unique Ergodicity conjecture of Rudnick--Sarnak~\cite{RudnickSarnak1994} predicts that, in negatively curved settings, no subsequence extraction should be necessary.

Several major developments have clarified the possible structure of semiclassical measures in chaotic geometries. Lindenstrauss proved quantum unique ergodicity for Hecke eigenfunctions on arithmetic hyperbolic surfaces~\cite{Lindenstrauss2006}. In the Anosov setting, Anantharaman~\cite{Anantharaman-2008} and Anantharaman--Nonnenmacher~\cite{Anantharaman-Nonnenmacher-2007} established positive entropy bounds for semiclassical measures, excluding in particular concentration on a single closed geodesic. Hassell showed, by constructing ergodic billiards that are not quantum uniquely ergodic, that ergodicity of the classical dynamics does not imply quantum unique ergodicity in full generality~\cite{Hassell-2010}. More recently, Dyatlov and Jin proved that every semiclassical measure on a compact hyperbolic surface has full support \cite{Dyatlov-Jin-2018}, a result extended to negatively curved surfaces of variable curvature by Dyatlov, Jin and Nonnenmacher~\cite{Dyatlov-Jin-Nonnenmacher-2022}. Related higher-dimensional support results were recently obtained by Kim and Miller~\cite{Kim-Miller-2025}.

The situation considered in the present paper is of very different nature. The principal Hamiltonian depends only on the momentum variable,
\[
H=H(\xi),
\]
and therefore generates a completely integrable flow on \(T^*\T^2\). In such a setting, one cannot rely on chaoticity of the classical flow to force equidistribution. Nevertheless, even in the non-magnetic case, the flat torus exhibits remarkable regularity properties for quantum limits. Indeed, the classical estimates of Cooke~\cite{Cooke-1971} and Zygmund~\cite{Zygmund-1974} imply uniform \(L^4\)-bounds for normalized eigenfunctions of the Laplacian on \(\T^2\). As a consequence, if \((u_{\lambda_n})_{n\in\N}\) is a sequence of normalized Laplace eigenfunctions on \(\T^2\), then every accumulation point of the probability densities
\[
|u_{\lambda_n}(x)|^2\,\dd x
\]
is absolutely continuous with respect to the Lebesgue measure and has a density in \(L^2(\T^2)\). Jakobson~\cite{Jakobson-1997} refined this result by classifying the configuration-space quantum limits on \(\T^2\): their densities are trigonometric polynomials whose Fourier supports satisfy additional geometric constraints. He also obtained absolute continuity results for quantum limits on flat tori in arbitrary dimension. In a related pseudo-integrable setting, Marklof and Rudnick~\cite{MarklofRudnick} proved that, for rational polygons, a density-one subsequence of eigenfunctions becomes uniformly distributed in configuration space. These stationary results show that integrable or pseudo-integrable geometries may
already impose nontrivial regularity on configuration-space limits. The present article is concerned with a dynamical and magnetic approach to these phenomena.

\subsubsection*{Long-time semiclassical dynamics, observability and two-microlocal analysis}

A second line of work, much closer in spirit to the present article, concerns time-dependent semiclassical measures for Schr\"odinger evolutions observed on long time scales. Maci\`a initiated a systematic study of such limits on compact manifolds and on flat tori \cite{Macia-2009,Macia-2010}. In the torus case, he showed that the limiting dynamics is not only determined by the initial semiclassical measure: one must keep track of additional information concentrated near resonant frequencies. He also proved that limits of configuration-space densities for the Schr\"odinger flow are absolutely continuous with respect to Lebesgue measure. This program was further developed by Anantharaman and Maci\`a \cite{AnantharamanMacia11,Anantharaman-Macia-2014}. They obtained a detailed description of semiclassical measures for the Schr\"odinger equation on flat tori and established strong regularity properties of the associated configuration-space limits. As a by-product, their analysis also yielded observability consequences for the Schr\"odinger evolution.

These questions are also closely related to observability and control problems. The seminal work of Bardos, Lebeau and Rauch~\cite{Bardos-Lebeau-Rauch-1992} concerned the wave equation. For Schr\"odinger equations, Lebeau~\cite{Lebeau-1992} established exact controllability results in a general geometric framework while, in periodic geometries, observability phenomena already appear in the work of Jaffard~\cite{Jaffard-1990}. In a non-semiclassical formulation, Burq and Zworski~\cite{BurqZworski} studied the equation
\[
i\partial_tu=(-\Delta+V)u
\qquad\text{on }\T^2,
\]
with \(V\in \mathcal{C}^\infty(\T^2)\), and proved that any nonempty open subset controls, in the \(L^2\)-sense, both stationary and dynamical Schr\"odinger solutions. Bourgain, Burq and Zworski~\cite{Bourgain-Burq-Zworski-2013} later extended this analysis to potentials \(V\in L^2(\T^2)\) on two-dimensional flat tori, including both rational and irrational tori.

A systematic framework for long-time semiclassical dynamics in completely integrable systems on flat tori \(\T^d\) was developed by Anantharaman, Fermanian-Kammerer and Maci\`a~\cite{Anantharaman-Fermanian-Kammerer-Macia-2015}. They consider semiclassical evolutions with principal Hamiltonian \(H(\xi)\) and study time-dependent semiclassical measures associated with the rescaled evolution \(t\mapsto u_h(t\tau_h)\). Their analysis reveals a critical time scale
\[
\tau_h= h^{-1},
\]
at which two-microlocal effects become decisive. Below this threshold, concentration phenomena near resonant directions may persist while at and beyond it, one obtains stronger regularity properties in the configuration variable under suitable nondegeneracy assumptions on \(\dd^2H\). Their results also provide observability consequences in this integrable framework. The two-microlocal viewpoint underlying these works has a longer history in semiclassical analysis. Early constructions of two-microlocal semiclassical measures appear in the thesis of Fermanian-Kammerer~\cite{Fermanian-Kammerer-thesis-1995}, in the work of Nier~\cite{Nier-1996} and in Fermanian-Kammerer~\cite{Fermanian-Kammerer-2000}. Related two-microlocal methods were used by Anantharaman, L\'eautaud and Maci\`a \cite{Anantharaman-Leautaud-Macia-2016} for the Schr\"odinger equation on the disk which is another example of an integrable system but with a singularity at the boundary. Again, they obtained a detailed structural description of Wigner measures associated with the billiard dynamics and used it to derive observability results. Related long-time concentration, quantum-limit, and observability phenomena on manifolds with periodic geodesic flow were investigated by Macià and
Rivière on Zoll manifolds~\cite{Macia-Riviere-2016} and on spheres~\cite{Macia-Riviere-2019}. The work of Maci\`a and Rivi\`ere \cite{Macia-Riviere-2018} is especially close to the present paper. They study stationary and time-dependent solutions to strong semiclassical perturbations of the free Schr\"odinger equation on the two-dimensional flat torus. By performing a second microlocalization adapted to the perturbative scale, they describe refined concentration and nonconcentration mechanisms. In particular, they prove that sufficiently accurate quasimodes may concentrate only on the critical set of the geodesic average of the perturbation.

The present article belongs to this family of results on long-time dynamics in integrable geometries, but the magnetic structure produces a genuinely different effective behavior. The key new feature is the appearance of an earlier critical time scale,
\[
\tau_h= h^{-1/2},
\]
at which the magnetic contribution already enforces spatial equidistribution. At longer times, and in particular beyond the Heisenberg scale
\[
\tau_h= h^{-1},
\]
the remaining momentum distribution also becomes rigid along the regular energy curves of \(H\).

\subsubsection*{Magnetic semiclassical problems}

The magnetic effect at work in the present work is most directly related to the recent article of Morin and Rivi\`ere \cite{MorinRiviere2025}. They study high-energy eigenfunctions of magnetic Schr\"odinger operators on \(\T^2\), acting on sections of a magnetic line bundle. Under the assumption $B>0$, they prove a quantum unique ergodicity result: every sequence of high-energy eigenfunctions equidistributes in phase space, despite the complete integrability of the underlying principal dynamics.

The present article may be viewed as a dynamical counterpart to this stationary result. We study long-time Schr\"odinger evolutions on the same magnetic geometric background, but for the much broader class of semiclassical Hamiltonians
\begin{equation}\label{e:equation-schrodinger-magnetique}
    ih\partial_tu_h
=
\Op\big(H(\xi)+hR_h(x,\xi)\big)u_h.
\end{equation}

Here the operator acts on sections of a magnetic line bundle over \(\T^2\) (see Section~\ref{s:Magnetic field and quantization}), the principal Hamiltonian \(H(\xi)\) remains completely integrable, and the subprincipal perturbation \(R_h\) is not necessarily assumed to be of magnetic origin. However, in the magnetic Laplacian case introduced in \eqref{e:laplacien-mag}, this perturbation is precisely generated by the periodic part of the magnetic potential \(A^{\mathrm{per}}\), and therefore encodes the nonconstant component \(\dd A^{\mathrm{per}}\) of the magnetic field. The twisted structure of our torus nevertheless enters the effective long-time dynamics and is responsible for the equidistribution mechanisms proved below.

The non-vanishing condition appearing in Theorem~\ref{t:intro-main-structure},
\[
G_{\widehat B_0,\Lambda}(x,\xi)\neq0,
\]
is the analogue, in this more general dynamical setting, of the geometric positivity condition used by Morin and Rivi\`ere which itself already appeared in the works of Glass and Han-Kwan on Vlasov equations \cite{Glass-Han-Kwan-2012}. In the particular case of the magnetic Laplacian, where \(H(\xi)=|\xi|^2\) and the subprincipal contribution comes from the magnetic potential, this condition implies the positivity of the magnetic field averaged along periodic geodesics. Thus the mechanism used in the present work extends to general integrable Hamiltonians and general subprincipal perturbations the magnetic geometric condition underlying the stationary result of \cite{MorinRiviere2025}. See also the recent work of Le~Balc'h, Niu and Sun~\cite{LeBalc'h} for observability results on \eqref{e:equation-schrodinger-magnetique} under related geometric conditions.

Other magnetic semiclassical regimes have been studied in different geometric and asymptotic settings. In contrast with the regime considered in the present article, these works concern a strong magnetic field which, in the normalization used here, corresponds to taking \(\widehat B_0\) of order \(h^{-1}\). The magnetic effect therefore enters the principal classical dynamics rather than appearing as a subprincipal correction. For instance, magnetic Laplacians on hyperbolic surfaces with constant magnetic field are related to earlier work of Zelditch~\cite{Zelditch-1992}; more recently, Charles and Lefeuvre~\cite{Charles-Lefeuvre-2026} obtained a finer description of the associated semiclassical defect measures. Their analysis separates three regimes: a flexible low-energy regime, a quantum unique ergodicity result at the critical energy level, and a Shnirelman-type equidistribution theorem in the high-energy regime.

In the Euclidean setting, Boil and V{\~u} Ng\d{o}c~\cite{Boil-VuNgoc-2021} studied the long-time propagation of generalized coherent states for the magnetic Laplacian on \(\R^2\), again in the strong magnetic field regime. For low-energy initial states, they obtain a precise description of the Schrödinger evolution up to times of order \(h^{-1}\), without performing any time averaging. At this time scale, an initially localized coherent state may split into several coherent states, each following the averaged guiding-center dynamics with a different effective speed.

\subsection{Strategy of the proof}

We now outline the main ideas of the proof. We work with the time-dependent magnetic Wigner distributions and semiclassical measures introduced in \eqref{e:Wigner-magnetic-distrib}. For a given observation time scale $\tau_h$, we consider a subsequential limit
\[
W^B(\dd t,\dd x,\dd\xi)=\nu_t^B(\dd x,\dd\xi)\otimes \dd t,
\]
where $(\nu_t^B)_{t\in\R}$ is a measurable family of probability measures on $T^*\T^2$ defined for a.e.\ $t$.

\medskip

\noindent\textbf{Propagation constraints and dependence on the time scale.}
A first step is to derive the dynamical constraints satisfied by $\nu_t^B$ and to understand how they depend on $\tau_h$. This is obtained from a commutator identity. For any time-independent test symbol $a(x,\xi)$, one has
\begin{equation}\label{e:derivation-intro}
    \frac{\dd}{\dd t}\,\big\langle \Op(a)\,u_h(t\tau_h),u_h(t\tau_h)\big\rangle
=\tau_h\,\frac{i}{h}\,\big\langle [\widehat H_h,\Op(a)]\,u_h(t\tau_h),u_h(t\tau_h)\big\rangle.
\end{equation}
Using the symbolic commutator expansion in the magnetic calculus, one has
\begin{equation}\label{e:commutateur-intro}
    \frac{i}{h}\,[\widehat H_h,\Op(a)] =\Op(\{H,a\}) + h\,\Op(\{R_h,a\}) + \mathcal{O}(h).
\end{equation}
Hence, combining \eqref{e:derivation-intro} and \eqref{e:commutateur-intro}, one obtains
\begin{align*}
\frac{1}{\tau_h}\frac{\dd}{\dd t}\,\big\langle \Op(a)\,u_h(t\tau_h),u_h(t\tau_h)\big\rangle
=&\,\big\langle\Op(\{H,a\}) u_h(t\tau_h),u_h(t\tau_h)\big\rangle \\
&+ h\big\langle\Op(\{R_h,a\}) u_h(t\tau_h),u_h(t\tau_h)\big\rangle + \mathcal{O}(h). 
\end{align*}
The relative size of \(\tau_h\) determines which terms survive in the limit. For instance, in the long-time regime \(\tau_h\to+\infty\), the left-hand side formally vanishes at the level of semiclassical limits, while the terms of size \(h\) disappear, leading to invariance under the Hamiltonian flow $\varphi_H^s$ defined in \eqref{e:intro-def-flot}.

\medskip

\noindent\textbf{Resonant versus non-resonant directions.}
The principal symbol \(H(\xi)\) depends only on \(\xi\) and generates a completely integrable Hamiltonian flow on \(T^*\T^2\), which was defined in \eqref{e:intro-def-flot}. For fixed \(\xi\), the corresponding motion on \(\T^2\) is the straight-line flow with velocity \(\nabla H(\xi)\). Recall that we are working on regular energy levels, so that $\nabla H\neq0$. Depending on the arithmetic properties of \(\nabla H(\xi)\), this trajectory is either dense or periodic. This dichotomy is reflected at the level of semiclassical measures and is at the core of our rigidity argument.

In the non-resonant (dense) case, invariance under the Hamiltonian flow forces the measure $\nu_t^B(\dd x,\dd\xi)$ to be of the form $\dd x \otimes \lambda_t(\dd\xi)$ along these directions.

Periodic directions are parametrized by primitive rank-one submodules \(\Lambda\in\mathcal L_1\). For the long-time regimes \(\tau_h\gg h^{-1/2}\) relevant to our main theorem, the critical concentration phenomenon occurs in neighborhoods of size \(h^{1/2}\) of the corresponding resonant sets
\[
E_{\Lambda^\perp\setminus\{0\}}
=
\left\{
\xi\in\R^2:
\nabla H(\xi)\in\Lambda^\perp\setminus\{0\}
\right\}.
\]
To separate the two contributions, we introduce a cutoff \(\chi_{h,\Lambda}(\xi)\) supported in such a neighborhood and decompose an observable as
\[
a
=
a\,\chi_{h,\Lambda}
+
a\,(1-\chi_{h,\Lambda}).
\]
The second term corresponds to the region that stays away from the resonant set at the relevant \(h^{1/2}\)-scale. In this region, the propagation identities coming from the invariance under \(\varphi_H^s\) force again the corresponding contribution to be the Lebesgue measure in the $x$-variable. Hence this part does not contribute to the non-uniform component of the limiting measure in the configuration variable. In other words, only the first term can capture the possible concentration near periodic directions and requires a finer analysis.

\medskip

\noindent\textbf{Two-microlocal analysis near periodic directions and the first threshold.}
To analyze this remaining contribution, we perform a two-microlocal localization near each resonant set, in the spirit of \cite{Fermanian-Kammerer-thesis-1995,
Fermanian-Kammerer-2000,
Anantharaman-Fermanian-Kammerer-Macia-2015,
Anantharaman-Leautaud-Macia-2016,
Macia-Riviere-2018,
MorinRiviere2025}. The idea is to enrich the usual phase-space description by an additional variable \(\eta\), which measures the transverse distance to the resonant set at the critical \(h^{1/2}\)-scale. This lift separates genuine concentration along periodic directions from the part already controlled by the ordinary semiclassical measure. This new contribution is referred to as the two-microlocal measure along $\Lambda$.

The resulting two-microlocal measures satisfy different dynamical properties according to the size of the observation time \(\tau_h\). For \(\tau_h\ll h^{-1/2}\), the magnetic contribution is not yet visible at the two-microlocal scale, and nontrivial periodic concentration may persist, see Lemma~\ref{l:three-regimes-2micro}. At the critical scale
\[
\tau_h= h^{-1/2},
\]
the lifted measures are transported by an effective vector field in the extended variables \((x,\xi,\eta)\),
\[
X_\Lambda
=
\eta\,\frac{\mathfrak e_\Lambda}{L_\Lambda}\cdot\partial_x
-
G_{\widehat B_0,\Lambda}(x,\xi)\,\partial_\eta,
\]
where \(G_{\widehat B_0,\Lambda}\) is the coefficient introduced in
\eqref{e:intro-definition-G}.

In the supercritical regime
\[
\tau_h\gg h^{-1/2},
\]
this transport relation becomes an invariance property under the flow generated by \(X_\Lambda\). The non-vanishing condition imposed in Theorem~\ref{t:intro-main-structure} forces the trajectories of this effective flow to escape in the \(\eta\)-direction. Since this two-microlocal component has finite mass, such an invariance is possible only if this component vanishes. Consequently, concentration in \(h^{1/2}\)-neighborhoods of \(E_{\Lambda^\perp\setminus\{0\}}\) cannot contribute to a non-uniform configuration-space limit. Gathering all these observations, this yields
\[
(\pi_x)_*\nu_t^B=\dd x
\qquad\text{for a.e. }t,
\]
and therefore
\[
\nu_t^B(\dd x,\dd\xi)
=
\dd x\otimes\lambda_t(\dd\xi)
\qquad\text{for a.e. }t,
\]
for some measurable family \((\lambda_t)_{t\in\R}\) of probability measures on \(\R^2\).

\medskip

\noindent\textbf{Momentum dynamics and the Heisenberg threshold.}
Once the configuration marginal is determined, the remaining long-time behavior is entirely encoded in the momentum marginal
\[
\lambda_t=(\pi_\xi)_*\nu_t^B,
\qquad
\pi_\xi:(x,\xi)\in T^*\T^2\longmapsto \xi.
\]
To study this marginal, it is enough to test the commutator identity \eqref{e:commutateur-intro} against symbols depending only on the momentum variable, namely \(b=b(\xi)\). Since
\[
\{H,b\}=0,
\]
the leading classical transport disappears in \eqref{e:derivation-intro}. Expanding the commutator identity \eqref{e:commutateur-intro} to the next order reveals a magnetic contribution proportional to the flux \(\widehat B_0\). Moreover, using the spatial equidistribution obtained above allows us to describe the propagation/invariance of the momentum marginal and to see that it comes from the magnetic correction in the commutator expansion.

For intermediate time scales
\[
h^{-1/2}\ll\tau_h\ll h^{-1},
\]
this correction is still negligible in the limit. One therefore obtains that
\[
t\longmapsto
\int_{\R^2} b(\xi)\,\lambda_t(\dd\xi)
\]
is constant for every test function \(b\), which shows that \(\lambda_t\) is independent of \(t\). At the critical scale
\[
\tau_h= h^{-1},
\]
the magnetic correction contributes at leading order and induces a genuine transport dynamics in the momentum variable. More precisely, the family \((\lambda_t)_{t\in\R}\) is transported by the flow \(\Upsilon_H^t\) generated by the vector field defined in \eqref{e:intro-momentum-vector-field}
\[
\widetilde X_H
=
-\widehat B_0
\left(
\partial_{\xi_1}H\,\partial_{\xi_2}
-
\partial_{\xi_2}H\,\partial_{\xi_1}
\right).
\]
Equivalently, one has
\[
\lambda_t=(\Upsilon_H^t)_*\lambda_0.
\]

For longer time scales
\[
\tau_h\gg h^{-1},
\]
the same commutator argument yields invariance of \(\lambda_t\) under the flow \(\Upsilon_H^s\). Moreover, conservation of the energy distribution along the quantum evolution implies that the energy marginal
\[
\sigma_t^B:=H_*\lambda_t
\]
is independent of \(t\). More precisely, one has
\[
\sigma_t^B
=
H_*\lambda_t
=
H_*\lambda_0
=:\sigma^B
\qquad\text{for a.e. }t.
\]
Hence, one may disintegrate \(\lambda_t\) with respect to the energy variable
\[
\lambda_t(\dd\xi)
=
\int_{E_1}^{E_2}
\lambda_{t,E}(\dd\xi)\,\sigma^B(\dd E),
\]
where \(\lambda_{t,E}\) is a probability measure on
\[
\mathbb S_E=\{\xi\in\R^2:H(\xi)=E\}.
\]
The invariance of \(\lambda_t\) under \(\Upsilon_H^s\) implies that, for almost every \(E\), the conditional measure \(\lambda_{t,E}\) is invariant under the induced flow on \(\mathbb S_E\). 
Since each \(\mathbb S_E\) is a smooth connected closed curve and since $\nabla H\neq 0$, this invariant probability measure is unique. Denoting it by \(\mu_E\), one obtains
\[
\lambda_t(\dd\xi)
=
\int_{E_1}^{E_2}
\mu_E(\dd\xi)\,\sigma^B(\dd E).
\]
In the particular case \(H(\xi)=|\xi|^2\), the measures \(\mu_E\) coincide with the normalized arc-length measures on the circles \(\mathbb S_E\).
\medskip

\noindent\textbf{Summary.}
The proof therefore relies on two complementary mechanisms. The first one, which becomes effective at the scale
\[
\tau_h= h^{-1/2},
\]
eliminates concentration near periodic directions and enforces equidistribution in the configuration variable. The second one, which appears at the scale
\[
\tau_h= h^{-1},
\]
governs the remaining momentum dynamics and yields rigidity along the regular energy curves of \(H\). Together, these two mechanisms give the structure of the time-dependent semiclassical measures stated in the main theorem.

\subsection{Structure of the paper}

Section~\ref{s:Magnetic field and quantization} recalls the magnetic line bundle framework on \(\T^2\), the torus-adapted magnetic Weyl quantization and the class of semiclassical operators considered in the paper. Section~\ref{s:semiclassicalmeasure} constructs the time-dependent magnetic Wigner distributions and proves the first propagation properties of the associated semiclassical measures, following the work of Macià~\cite{Macia-2009,Macia-2010} that we adapt in a magnetic setting. Section 4 develops the resonant decomposition and the two-microlocal analysis near periodic directions. Section~\ref{s:x-regularity} uses this analysis to prove the configuration-space equidistribution statement \eqref{e:intro-product-form}, hence the first part of Theorem~\ref{t:intro-main-structure} and Theorem~\ref{t:intro-simplified-equidistribution}. Section~\ref{s:long-time-propagation} analyzes the momentum marginal at and beyond the time scale \(h^{-1}\), proving Theorem~\ref{t:intro-main-structure} (i)--(iii). Finally, Appendix~\ref{ss:functional-calculus} gathers the functional-calculus and self-adjointness results needed for the magnetic pseudodifferential operators, while Appendix~\ref{ss:disintegration} recalls the disintegration theorem for Radon measures used throughout the paper.

\subsection*{Acknowledgements}
The author thanks Gabriel Rivi\`ere for many helpful comments throughout this project. The author acknowledges the support of the PRC grant ADYCT (ANR-20-CE40-0017) and of the ANR project La Gabare (ANR-25-CE40-7296).
\section{Magnetic field and quantization}\label{s:Magnetic field and quantization}

This section gathers the basic geometric and pseudodifferential tools needed throughout the paper. It is essentially a review of the torus-adapted framework introduced by Morin--Rivi\`ere~\cite{MorinRiviere2025} (itself rooted in the magnetic pseudodifferential calculus of M\u{a}ntoiu--Purice \cite{Mantoiu-Purice-2004} and the geometric quantization viewpoint of Charles \cite{Charles-2016}). For the reader's convenience, we recall: (i) the flux quantization condition ensuring that magnetic translations define a consistent notion of magnetic periodicity on $\T^2$; (ii) the associated Hilbert space $L^2(\T^2,L)$; and (iii) the corresponding torus-adapted magnetic Weyl quantization $\Op$ acting on large symbol classes $\mathscr{S}^m(T^*\T^2)$. Our main results rely only on the structural properties of this calculus (boundedness, adjointness, composition and commutator expansions) proved in \cite{MorinRiviere2025} and which we will use as a black box in the analysis of time-dependent semiclassical measures in the next sections.

\subsection{Magnetic fields on $\T^2$}

A magnetic field on a configuration space is a real closed $2$-form (the curvature of a unitary connection on a complex line bundle). On the flat two-torus $\T^2=\R^2/\Z^2$, once the orientation (and volume form) $\dd x_1\wedge \dd x_2$ is fixed, any magnetic field can be identified with a smooth real-valued function
\[
B\in \mathcal{C}^\infty(\T^2;\R),\qquad \mathbf{B}:=B(x)\dd x_1\wedge \dd x_2.
\]
We denote by
\[
\widehat B_0:=\int_{\T^2} B(x)\dd x_1\wedge \dd x_2
\]
the total magnetic flux, and we impose the condition
\begin{equation}\label{e:flux}
\widehat B_0\in 2\pi\Z \setminus \{0\}.
\end{equation}
This is the usual flux quantization condition ensuring the existence of a compatible line bundle structure (equivalently, a global magnetic Schr\"odinger operator acting on twisted periodic functions); see the discussion in \cite[Section~2]{MorinRiviere2025}. On $\T^2$, the zero-mean part $B-\widehat B_0$ is exact. Thus one may choose a periodic vector potential $A^{\mathrm{per}}$ such that 
\begin{equation}\label{e:B=dA}
B(x)-\widehat B_0=  \dd A^{\mathrm{per}}(x).
\end{equation}
Following Morin--Rivi\`ere, we fix once and for all the decomposition
\begin{equation}\label{e:A-decomp}
A=A^0+A^{\mathrm{per}},
\qquad
A^0(x_1,x_2)=\frac{\widehat B_0}{2}(-x_2,x_1),
\qquad
A^{\mathrm{per}}\ \text{is $\Z^2$--periodic}. 
\end{equation}
Thus, one has
\[
B(x) = \dd A(x).
\]
The periodic part $A^{\mathrm{per}}$ is not unique and may be changed by a periodic gauge $A^{\mathrm{per}}\mapsto A^{\mathrm{per}}+\nabla\varphi$, $\varphi\in \mathcal{C}^\infty(\T^2;\R)$, yielding unitarily equivalent operators. Let us finally stress that, in the abstract semiclassical model studied in this article, only the flux $\widehat B_0$ enters the torus-adapted magnetic Weyl quantization $\Op$ and the operators $\widehat H_h$ introduced in \eqref{e:definition-widehatH}. The full magnetic field $B$ appears only when this abstract framework is specialized to the magnetic Laplacian example \eqref{e:laplacien-mag}.

\subsection{Magnetic translations and magnetically periodic functions}\label{ss:Magnetic-translations and-magnetically-periodic-functions}

A key feature of magnetic-type operators is that they commute with magnetic (rather than usual) translations defined, for $m\in\Z^2$, by
\begin{equation}\label{e:mag-trans}
(T_m^{B}u)(x) := \exp\!\left( i\frac{\widehat B_0}{2} \big(m\wedge x+m_1m_2\big) \right)u(x-m), \qquad m\wedge x:=m_1x_2-m_2x_1,
\end{equation}
as operators on $L^2_{\mathrm{loc}}(\R^2)$. They form a \emph{projective representation} of $\Z^2$ on $L^2_{\mathrm{loc}}(\R^2)$, meaning that
\[
T_m^B T_{m'}^B = \sigma(m,m')\,T_{m+m'}^B, \qquad m,m'\in\Z^2,
\]
for a phase factor $\sigma(m,m')\in\mathbb S^1$. The map $\sigma:\Z^2\times\Z^2\to\mathbb S^1$ is called the associated cocycle; see \cite{MorinRiviere2025} for details. In the present case, a direct computation gives
\[
\sigma(m,m') = \exp\left(-i\widehat B_0\,m_1'm_2\right),
\]
and therefore
\[
T_m^B\,T_{m'}^B = \exp\left(-i\widehat B_0\,m_1'm_2\right)\,T_{m+m'}^B, \qquad m,m'\in\Z^2.
\]
The corresponding commutator cocycle is given by
\[
\frac{\sigma(m,m')}{\sigma(m',m)} = \exp\left(i\widehat B_0\,m\wedge m'\right).
\]
It is trivial, equivalently the magnetic translations commute, if and only if \eqref{e:flux} holds. In particular, under \eqref{e:flux}, the notion of magnetic periodicity is consistent, meaning that the magnetic translations commute with each other (see \cite{MorinRiviere2025}). We then define the space of \emph{twisted periodic} smooth functions by
\begin{equation}\label{e:CinfTL}
\mathcal{C}^\infty(\T^2,L) :=\{u\in \mathcal{C}^\infty(\R^2,\C):\ T_m^{B}u=u\ \text{for all }m\in\Z^2\}.
\end{equation}
Although we keep using $B$ in the notation $T_m^B$, the definition of these magnetic translations, and therefore of $\mathcal C^\infty(\T^2,L)$, depends only on the flux $\widehat B_0$, not on the full magnetic field $B$. Similarly, we define $L^2(\T^2,L)$ and Sobolev spaces $\mathcal H^s(\T^2,L)$ as the subspaces of $L^2_{\mathrm{loc}}(\R^2)$ and $\mathcal H^s_{\mathrm{loc}}(\R^2)$ invariant under all $T_m^B$, namely 
\begin{equation}\label{e:espace L^2}
    L^2(\T^2, L) = \{ u \in L^2_{\mathrm{loc}}(\R^2) ~:~ T_m^Bu=u \quad \forall m \in \Z^2 \},
\end{equation}
\begin{equation}\label{e:espace H^s}
    \mathcal H^s(\T^2, L) = \{ u \in H^s_{\mathrm{loc}}(\R^2) ~:~ T_m^Bu=u \quad \forall m \in \Z^2 \}.
\end{equation}
These are the natural ``twisted'' analogs of the usual Sobolev spaces on $\T^2$. They provide the natural functional setting for the operators appearing in the introduction.

For any integer $s\in\N$, one may equivalently define $\mathcal H^s(\T^2,L)$ as the completion of $\mathcal{C}^\infty(\T^2,L)$ for the norm built out of the magnetic derivative. More precisely, with $A^0$ as in \eqref{e:A-decomp}, set $\partial^A = (\partial_{x_1} - iA_1^0, \partial_{x_2} - i A^0_2)$, and 
\begin{equation}\label{e:norme}
    \displaystyle \lVert u \rVert_{H^s} := \left(\sum_{\lvert \alpha \lvert \le s } \int_{\T^2} \lvert (\partial^A)^\alpha u(x) \lvert^2\dd x\right)^\frac{1}{2}, \quad s \in \N.
\end{equation}
For general, possibly non-integer, \(s\geq0\), we refer to Appendix~\ref{ss:functional-calculus} for another definition of \(\mathcal H^s(\T^2;L)\).

\begin{remark}
    Geometrically, $\mathcal{C}^\infty(\T^2,L)$ can be identified with smooth sections of a complex line bundle $L\to\T^2$ with curvature $\widehat B_0 \dd x_1 \wedge \dd x_2$. Yet, the reader may only keep in mind the concrete realization in \eqref{e:CinfTL} which is the point of view adopted all along the article.
\end{remark}

\subsection{Torus-adapted magnetic Weyl quantization}

In order to quantize observables while preserving the magnetic periodicity, we use the torus-adapted magnetic Weyl calculus of \cite{MorinRiviere2025}, inspired by the magnetic calculus on $\R^d$ of M\u{a}ntoiu--Purice and collaborators \cite{Mantoiu-Purice-2004,Iftimie-Mantoiu-Purice-2007,Helffer-Purice-2010,Iftimie-Mantoiu-Purice-2019}. The construction relies on the magnetic Weyl operators $\mathcal{W}_h^{B}(\eta,\zeta)$, defined for
$\eta\in 2\pi\Z^2$ and $\zeta\in\R^2$ by
\begin{equation}\label{e:WB-ops}
\mathcal{W}_h^{B}(\eta,\zeta)u(x)
:=\exp\!\Big(ih\frac{\widehat B_0}{2}\,\zeta\wedge x\Big)\,
\exp\!\Big(i\eta\cdot(x+\tfrac{h\zeta}{2})\Big)\,
u(x+h\zeta),
\qquad u\in L^2(\T^2,L),
\end{equation}
which commute with all magnetic translations $T_m^B$ and therefore act on $L^2(\T^2,L)$. They satisfy the adjoint relation $\mathcal{W}_h^{B}(\eta,\zeta)^*=\mathcal{W}_h^{B}(-\eta,-\zeta)$ and an explicit composition law (a twisted Heisenberg relation); see \cite[\S3.1]{MorinRiviere2025}.

Given a smooth symbol $a=a(x,\xi)$, periodic in $x$, we define its Fourier transform
\[
\mathcal F(a)(\eta,\zeta)
:=\int_{\T^2\times\R^2} a(x,\xi)\,e^{-i\eta\cdot x}\,e^{-i\zeta\cdot\xi}\dd x\dd\xi,
\]
and we set
\begin{equation}\label{e:OpB-def}
\Op(a)
:=\frac{1}{(2\pi)^2}\sum_{\eta\in 2\pi\Z^2}\int_{\R^2}\mathcal F(a)(\eta,\zeta)\,\mathcal{W}_h^{B}(\eta,\zeta)\dd\zeta,
\end{equation}
initially for $a\in \mathcal{C}_c^\infty(\T^2\times\R^2)$ and then extended to the symbol classes used below. 

We work with the large symbol class
\[
\mathscr{S}^m(T^*\T^2)
:=\big\{a\in \mathcal{C}^\infty(\T^2\times\R^2):\
|\partial_x^\alpha\partial_\xi^\beta a(x,\xi)|\le C_{\alpha\beta}\langle\xi\rangle^{m}\big\}, 
\]
which is stable under the operations required in this article. This is the toral analogue of the class $S(\langle \xi\rangle^m)$ from \cite[Chapter~4]{Zworski2012}. In the sequel, we shall also use \(h\)-dependent symbols \(a_h\in \mathscr{S}^m(T^*\T^2)\). Unless otherwise stated, this always means that the estimates defining \(\mathscr{S}^m(T^*\T^2)\) hold uniformly with respect to \(h\in(0,1]\). The associated magnetic calculus enjoys the usual structural properties (boundedness, adjoint, composition and commutator expansions). In particular, for $a\in \mathscr{S}^m$, the formal adjoint is given by \cite[Proposition~3.4]{MorinRiviere2025}
\begin{equation}\label{e:formal-adjoint}
    \Op(a)^*=\Op(\overline a) \quad \text{on }L^2(\T^2,L).
\end{equation}
We shall also use the following composition formula for the magnetic Weyl quantization, established in \cite[Theorem~3.5]{MorinRiviere2025}: for every $a\in \mathscr{S}^{m_1}$ and $b\in \mathscr{S}^{m_2}$,
\begin{equation}\label{e:composition-rule}
    \Op(a)\Op(b)=\Op(a\star_h b),
\end{equation}
where $a\star_h b \in \mathscr{S}^{m_1+m_2}$. Moreover, \cite[Theorem~3.5]{MorinRiviere2025} provides an explicit formula for the expansion of $a\star_h b$ at any order $N \ge 1$: 
\begin{equation}\label{e:a-star-b}
    \displaystyle a\star_h b(x,\xi) = \sum_{n=0}^{N-1}\frac{1}{n!}\left(\frac{ih}{2}\right)^n \Omega_{hB}(D_x,D_\xi,D_{x'},D_{\xi'})^na(x,\xi)b(x',\xi')\lvert_{(x,\xi)=(x',\xi')} +\mathcal{O}_{\mathscr{S}^{m_1+m_2}}(h^N),
\end{equation}
where $\Omega_{hB}(\eta,\zeta,\eta',\zeta') = \eta'\cdot \zeta - \eta \cdot \zeta' +h\widehat B_0 \zeta' \wedge \zeta$ and where we use the notation $D := -i\partial$. Applying \eqref{e:a-star-b} with $N=1$, one obtains 
\begin{equation}\label{e:premier-ordre-composition}
    a \star_h b = ab +\mathcal{O}_{\mathscr{S}^{m_1+m_2}}(h).
\end{equation}
In order to study magnetic semiclassical measures, we will use the following commutator identity:
\begin{equation}\label{e:commutator}
a \star_h b - b\star_h a = \frac{h}{i }\{a,b\} + \frac{\widehat{B}_0h^2}{i}(\partial_{\xi_2} a\partial_{\xi_1}b - \partial_{\xi_1}a\partial_{\xi_2}b) + O_{\mathscr{S}^{m_1+m_2}}(h^3).
\end{equation}
This identity is the basic input for the propagation and invariance properties of time-dependent semiclassical measures in the next section. In what follows, we will mainly use the expansion \eqref{e:commutator} at order 3, and we will use the notation $\mathcal{O}_{\mathscr{S}^k}(h^n)$ to say that the remaining term can be expressed as $h^n r_h$, where $r_h$ belongs uniformly in $h$ to the symbol class $\mathscr{S}^k(T^*\T^2)$. Throughout the paper, we shall repeatedly use the following rescaling formula. As observed in~\cite[Remark~3.8]{MorinRiviere2025}, for every $\delta>0$, one has
\begin{equation}\label{e:rescaling-semiclassique}
    \Op(a(x,\xi))= \mathsf{Op}_{\delta h}^{B}(a(x,\delta^{-1}\xi)) .
\end{equation}
In order to study the limit as the semiclassical parameter $h\to 0^+$, we shall also repeatedly use the Calderón--Vaillancourt theorem
\cite[Theorem~3.7]{MorinRiviere2025}, which ensures that the quantization of a symbol in $\mathscr{S}^0(\T^2\times\R^2)$ defines a bounded operator on $L^2(\T^2,L)$.
\begin{theorem}[Calderón--Vaillancourt]\label{t:Calderon-Vaillancourt}
There exist constants $C_0>0$ and $N_0\in \N$ such that, for every $a\in \mathscr{S}^0(\T^2\times\R^2)$, the operator $\Op(a)$ is bounded on $L^2(\T^2,L)$ and satisfies 
\begin{equation}\label{e:Calderon-Vaillancourt}
    \forall h\in (0,1], \qquad
    \|\Op(a)\|_{L^2\to L^2}
    \le
    C_0 \sum_{|\alpha|+|\beta|\le N_0}
    h^{|\beta|}
    \|\partial_x^\alpha \partial_\xi^\beta a\|_{\infty}.
\end{equation}
\end{theorem}

\subsection{Magnetic operators}\label{ss:magnetic-operator}

As already described in the introduction, the rest of this article will be dedicated to the study of solutions to Schrödinger-type equations of the form 
\begin{equation}\label{e:eqS}
ih\partial_t(u_h) = \widehat H_hu_h, 
\end{equation}
where the operator $\widehat H_h$ is the magnetic quantization of specific symbols. Let $H\in \mathcal C^\infty(\R^2,\R)$ be a general Hamiltonian depending only on the $\xi$-variable, belonging to the symbol class $\mathscr{S}^m(\R^2)$ and elliptic of order \(m\). This means that there exist $M>0$ and $c_0>0$ such that
\begin{equation}\label{e:definition-elliptic}
\forall\,|\xi|\ge M,\qquad c_0\langle\xi\rangle^{m}\le H(\xi).
\end{equation}
Throughout the paper, we fix an energy window $[E_1,E_2]$, with $0<E_1<E_2$, containing no critical values of $H$. Equivalently, setting $\Omega_{E_1,E_2}:=\{\xi\in\R^2:\ E_1\leq H(\xi)\leq E_2\}$, we assume that
\begin{equation}\label{e:energy-window}
\nabla H(\xi)\neq 0,
\qquad \forall \xi\in\Omega_{E_1,E_2}.
\end{equation}
For $E\in [E_1,E_2]$, we denote by $\SE:=\{\xi\in\R^2:\ H(\xi)=E\}$ the corresponding energy layer. We also assume that each $\SE$ is compact and connected. Therefore, $\SE$ is a smooth compact connected one-dimensional submanifold of $\R^2$, hence
\begin{equation}\label{e:SE-diffeo-S1}
\SE\simeq \mathbb S^1.
\end{equation}
Let also $R_h$ belonging to $\mathscr{S}^{m}(\T^2\times \R^2)$ uniformly in $h\in(0,1]$ and real valued. More precisely, we suppose that $R_h$ is of the form 
\begin{equation}\label{e:definition-Rh}
R_h(x,\xi) = \mathscr{R}(x,\xi) + \mathcal{O}_{\mathscr{S}^{m}}(h).
\end{equation}
We will consider magnetic operators denoted by $\widehat H_h$ which are of the form
\begin{equation}\label{e:magnetic-operators}
    \widehat H_h = \Op(H(\xi)+hR_h(x,\xi)).
\end{equation}
Since it is the quantization of a real-valued symbol and its principal symbol $H$ is elliptic, the operator $\widehat H_h$ enjoys the strong properties recalled in the following proposition. 
\begin{proposition}[Self-adjointness]\label{prop:esa} Let $m>1$, assume that $H\in \mathscr{S}^m(\R^2)$ is real-valued and elliptic of order m, and that
$R_h\in \mathscr{S}^{m}(T^*\T^2)$ is real-valued and of the form \eqref{e:definition-Rh}.
Then, for $h>0$ small enough, the operator
\[
\widehat H_h:=\Op\big(H(\xi)+hR_h(x,\xi)\big)
\]
is essentially self-adjoint on $\mathcal{C}^\infty(\T^2,L)$. Its closure (still denoted
$\widehat H_h$) is self-adjoint on $L^2(\T^2,L)$ with domain $\mathcal H^m(\T^2,L)$ and has
compact resolvent.
\end{proposition}
We refer to Appendix~\ref{s:appendix} for a reminder of the standard arguments leading to this result in the torus magnetic setting. By Stone's theorem~\cite[Theorem VIII.7]{ReedSimonI}, since $\widehat H_h$ is
self-adjoint, it generates a strongly continuous one-parameter unitary
group
\begin{equation}\label{e:propagator}
    U_h(t):=\exp\!\Big(-\frac{i}{h}t\,\widehat H_h\Big),\qquad t\in\R,
\end{equation}
so that the Cauchy problem $ih\partial_t u_h=\widehat H_h u_h$ is globally
well-posed in $L^2(\T^2,L)$ and satisfies $\|u_h(t)\|_{L^2(\T^2,L)}=\|u_h(0)\|_{L^2(\T^2,L)}$.
The operator $\widehat H_h $ can be viewed as a twisted counterpart of the classical operator
$\mathsf{Op}_h^{\rm W}(H(\xi)+hR_h(x,\xi))$ on the (non-twisted) torus, where $\mathsf{Op}_h^{\rm W}$ denotes the usual Weyl quantization. 

Under \eqref{e:flux}, Morin--Rivi\`ere considered the magnetic Laplacian,
\begin{equation}\label{e:mag-laplacian}
\mathcal L^{B}u
:=\big(-i\partial_{x_1}-A_1(x)\big)^2u +\big(-i\partial_{x_2}-A_2(x)\big)^2u,
\qquad u\in H^2(\T^2,L),
\end{equation}
where $A(x)= (A_1(x),A_2(x))$ is defined by \eqref{e:A-decomp}. This is exactly the magnetic Laplacian for the magnetic field $B = \widehat B_0 + \dd A^{\mathrm{per}}$ and in that case, the magnetic Weyl quantization produces 
\begin{equation}\label{e:quantization-of-magnetic-laplacian}
h^2\mathcal L^{B} = \Op(\lvert \xi \lvert^2 + h \xi \cdot c_1(x) +h^2 c_0(x)),
\end{equation}
for some smooth real-valued coefficients $c_1:\T^2\to\R^2$ and $c_0:\T^2\to\R$; see \cite[\S3.3]{MorinRiviere2025}. 
\begin{remark}\label{r:laplacien-magnetique-champs-constant}
    For a constant magnetic field, i.e. $B(x) = \widehat B_0$, the magnetic quantization produces 
    \begin{equation}\label{e:def-laplacien-magnetique-champs-constant}
        h^2\mathcal{L}^{\widehat B_0} = \Op(\lVert \xi \lVert^2). 
    \end{equation}
\end{remark}
\section{Time-dependent magnetic semiclassical measures}\label{s:semiclassicalmeasure}

This section introduces the space--time objects used throughout the paper to describe the high-frequency behavior of solutions to \eqref{e:eqS} observed on a macroscopic time scale $\tau_h$. Our presentation follows the standard semiclassical framework for time-dependent Wigner distributions and their disintegration properties \cite{Macia-2009,Anantharaman-Macia-2014,Anantharaman-Fermanian-Kammerer-Macia-2015}, but all pseudodifferential objects are defined here through the torus-adapted magnetic Weyl quantization recalled in Section~\ref{s:Magnetic field and quantization}. We also record the basic propagation and invariance properties satisfied by the limiting measures, depending on the asymptotic behavior of the time scale $\tau_h$.

\subsection{Time-dependent quantum limits on configuration space}
Let $h\to0^+$ and let $(u_{h})_{h\to 0^+}\in \mathcal{C}(\R;\mathcal H^m(\T^2,L))\cap \mathcal{C}^1(\R;L^2(\T^2,L))$ solve \begin{equation}\label{e:Schrodinger-PDE} 
\left\{
\begin{aligned}
   i h \partial_t u_{h}(t,x) &= \widehat{H}_{h}\, u_{h}(t,x), \\[0.3em]
   u_{h}(t=0,x) &= u_{h}^{(0)}(x),
\end{aligned}
\right.
\end{equation} 
with the normalization
\[
\|u_{h}^{(0)}\|_{L^2(\T^2,L)}=1.
\]
By unitarity, for every $t \in \R$, one has $\lVert u_{h}(t) \lVert_{L^2(\T^2,L)}=1$. Fix a time scale $(\tau_{h})_{h\to 0^+}>0$. We say that a (positive) Radon measure $\nu$ on $\R\times\T^2$ is a \emph{time-dependent quantum limit} associated with the rescaled family $t\mapsto u_{h}(t\tau_{h})$ if for every test function $a\in \mathcal{C}_c^\infty(\R\times\T^2,\C)$,
\begin{equation}\label{e:DefQL}
\int_{\R\times\T^2} a(t,x)\,|u_{h}(t\tau_{h},x)|^2\,\dd x\,\dd t
\xrightarrow[h \to 0^+]{}
\int_{\R\times\T^2} a(t,x)\,\nu(\dd t,\dd x).
\end{equation}

Using the disintegration theorem \ref{t:desintegration} for Radon measures and the fact that the sequence $(u_h(t\tau_h))_{h \to 0^+}$ is normalized in $L^2$, there exist a positive Radon measure $\omega$ on $\R$ and a $\omega$-measurable family of probability measures $(\nu_t)_{t\in\R}$ on $\T^2$ such that
\begin{equation}\label{e:DefQL-disintegration}
\nu(\dd t,\dd x)=\nu_t(\dd x) \otimes \omega(\dd t).
\end{equation}
In particular, for all $\psi\in \mathcal{C}_c^\infty(\R)$ and $b\in \mathcal{C}^\infty(\T^2)$,
\begin{equation}\label{e:DefQL-projected}
\int_{\R\times\T^2}\psi(t)b(x)\,\nu(\dd t,\dd x)
=
\int_\R \psi(t)\left(\int_{\T^2} b(x)\,\nu_t(\dd x)\right)\omega(\dd t).
\end{equation}
Moreover, choosing $b\equiv 1$ in \eqref{e:DefQL-projected} and using the normalization $\|u_h(t\tau_h)\|_{L^2(\T^2,L)}=1$ for all $t$, we obtain for every
$\psi\in \mathcal{C}_c^\infty(\R)$
\[
\int_\R \psi(t)\,\omega(\dd t)
=
\int_{\R\times\T^2}\psi(t)\,\nu(\dd t,\dd x)
=
\lim_{h\to0^+}\int_\R \psi(t)\,\dd t,
\]
and therefore $\omega(\dd t)=\dd t$. In particular, \eqref{e:DefQL-disintegration} becomes
\begin{equation}\label{e:def-quantum-config-space-measure}
    \nu(\dd t,\dd x)=\nu_t(\dd x)\otimes \dd t.
\end{equation}

We denote by $ \mathcal{N}(\widehat H_h,\tau_h)$ the set of all such configuration-space limits arising from solutions to \eqref{e:Schrodinger-PDE}. 
\begin{remark}
    Since $\|u_h(t\tau_h)\|_{L^2(\T^2,L)}=1$ for all $t$ and since test functions are compactly supported in time, the family of measures
\[
|u_h(t\tau_h,x)|^2\,\dd x\,\dd t
\]
is bounded in the space of Radon measures on $\R\times\T^2$. In particular, $\mathcal N(\widehat H_h, \tau_h)$ is nonempty. Along any sequence $h\to0^+$, one can extract a subsequence such that \eqref{e:DefQL} and \eqref{e:def-quantum-config-space-measure} hold for some $\nu$.
\end{remark}

For later use, we also record the semiclassical objects associated with the initial data. Since $\|u_h^{(0)}\|_{L^2(\T^2,L)}=1$, one may extract a subsequence such that
\begin{equation}\label{e:initial-config-measure}
\int_{\T^2} b(x)\,|u_h^{(0)}(x)|^2\,\dd x
\longrightarrow
\int_{\T^2} b(x)\,\nu_0(\dd x),
\qquad \forall b\in \mathcal{C}^\infty(\T^2),
\end{equation}
for some probability measure $\nu_0$ on $\T^2$.

The goal of this article is to describe elements of the set $\mathcal{N}(\widehat H_h,\tau_h)$ depending on the time scale $\tau_h$ and the initial data used to generate it. In order to do this, we will now lift the measure to the phase space using the Weyl magnetic quantization introduced in \cite{MorinRiviere2025} and recalled in Section \ref{s:Magnetic field and quantization}. 

\subsection{Magnetic Wigner distributions and phase-space lifts}\label{s:lift}

To lift the previous limits to phase space, we use the magnetic Weyl quantization. For $a\in \mathcal{C}_c^\infty(\R\times T^*\T^2)$, we define the \emph{time-dependent magnetic Wigner distribution} by
\begin{equation}\label{e:WB-def}
\langle W_h^B(\tau_h),a\rangle
:=
\int_\R \big\langle \Op(a(t))\,u_h(t \tau_h),\,u_h(t \tau_h)\big\rangle_{L^2(\T^2,L)}\,\dd t.
\end{equation}
The next lemma collects basic compactness and positivity properties.

\begin{lemma}\label{l:magnetic-wigner}
Let $(u_h)_{h\to0^+}$ solve \eqref{e:Schrodinger-PDE}. Then, the following holds: 
\begin{enumerate}
\item the family $(W_h^B(\tau_h))_{h \to 0^+}$ is bounded in $\mathcal D'(\R\times T^*\T^2)$;
\item any accumulation point $W^B$ of $(W_h^B(\tau_h))_{h \to 0^+}$ is a nonnegative Radon measure on
$\R\times T^*\T^2$.
\end{enumerate}
\end{lemma}
\begin{proof}
Boundedness follows from the Calder\'on--Vaillancourt Theorem \ref{t:Calderon-Vaillancourt} in the torus-adapted magnetic calculus; see \cite[\S3.4]{MorinRiviere2025}. Indeed, for $a$ supported in $t$ on a compact set $[t_1, t_2 ] \subset\R$, this theorem implies that
\[
|\langle W_h^B(\tau_h),a\rangle|
\le C_0\,\lvert t_2-t_1 \lvert\,\sum_{|\alpha|+|\beta|\le N_0}\sup_{t\in [t_1, t_2 ]}\|\partial_x^\alpha\partial_\xi^\beta a(t)\|_{\infty},
\]
uniformly for $h$ small and where $C_0$ and $N_0$ are constants independent of the symbol $a$. Positivity is obtained from the easy Gårding's inequality: let $a\in \mathcal{C}_c^\infty(\R\times T^*\T^2)$ be nonnegative and let $[t_1,t_2]\subset \R$ be a compact interval containing the support of $a$ in the time variable. For $\varepsilon>0$, define for each $t\in [t_1,t_2]$
\[
a_\varepsilon(t,x,\xi):=\sqrt{a(t,x,\xi)+\varepsilon}.
\]
Then $(a_\varepsilon(t))_{t\in [t_1,t_2]}$ is a bounded family in $\mathscr{S}^0$. Using the adjointness \eqref{e:formal-adjoint} and composition \eqref{e:composition-rule} rules in the magnetic calculus,
\[
\displaystyle \|\Op(a_\varepsilon)u_h(t \tau_h)\|_{L^2}^2
= \langle \Op(a_\varepsilon)^*\Op(a_\varepsilon)u_h(t \tau_h),u_h(t \tau_h) \rangle = \langle \Op(a+\varepsilon)u_h(t \tau_h),u_h(t \tau_h)\rangle+\mathcal O_\varepsilon(h),
\]
uniformly for $t\in [t_1,t_2]$. Hence, for every nonnegative $a \in \mathcal{C}_c^\infty(\R \times T^*\T^2)$ and for every $\varepsilon>0$, $\langle W^B,a\rangle\ge -\lvert t_2 - t_1 \lvert \varepsilon$, which allows us to deduce that the limit distribution $W^B$ is nonnegative, hence a measure.
\end{proof}
The previous lemma only yields the existence of a nonnegative phase-space measure. Under the spectral localization assumption on the initial data, one can further localize its support in the energy variable.
\begin{lemma}[Support of the limit measure]\label{l:support-de-W^B}
    Suppose that the initial data $(u_h^{(0)})_{h\to 0^+}$ satisfy the following spectral localization assumption: 
\begin{equation}\label{e:spectral-projection}
    \displaystyle \lim_{h \to 0^+} \lVert \mathbbm{1}_{[E_1,E_2]}(\widehat H_h)u_h^{(0)} - u_h^{(0)} \lVert_{L^2} = 0. 
\end{equation} 
Then, the measure $W^B$ is supported in $ \R \times \T^2 \times \Omega_{E_1,E_2}$. 
\end{lemma}
\begin{proof}
Let $(u_h)_{h \to 0^+}$ be a solution to \eqref{e:Schrodinger-PDE}, i.e.
\[
u_h(t)=e^{-it\widehat H_h/h}u_h^{(0)}.
\]
Since the pseudodifferential calculus developed in Section~\ref{s:Magnetic field and quantization} is only available for smooth functions, we first need to show that the spectral projection assumption~\eqref{e:spectral-projection} can be written in terms of smooth functions. This will then allow us to apply the functional calculus developed in the Appendix~\ref{s:appendix}, in particular Corollary~\ref{c:functional-calculus}. We first show that for $\kappa \in \mathcal{C}_c^\infty(\R)$ satisfying $0\le\kappa \le 1$ and $\kappa\equiv 1$ on $[E_1,E_2]$ assumption \eqref{e:spectral-projection} implies that
\begin{equation}\label{e:smooth-spectral-projection}
    \big\|(\mathrm{Id}-\kappa(\widehat H_h))u_h^{(0)}\big\|_{L^2}\xrightarrow[h\to0^+]{}0.
\end{equation}
Indeed, since $\kappa\equiv 1$ on $[E_1,E_2]$, one has $(1-\kappa)\,\mathbbm{1}_{[E_1,E_2]}=0$, hence, by the spectral theorem applied to the self-adjoint operator \(\widehat H_h\), one has
\begin{equation*}
    (\mathrm{Id}-\kappa(\widehat H_h))\mathbbm{1}_{[E_1,E_2]}(\widehat H_h)=0.
\end{equation*}
Therefore, for any $u$, one gets 
\[
(\mathrm{Id}-\kappa(\widehat H_h))u = (\mathrm{Id}-\kappa(\widehat H_h))(\mathrm{Id}-\mathbbm{1}_{[E_1,E_2]}(\widehat H_h))u.
\]
Moreover, $0\le \kappa\le 1$ implies $\|\mathrm{Id}-\kappa(\widehat H_h)\|_{L^2\to L^2}\le 1$. Hence using \eqref{e:spectral-projection} leads to \eqref{e:smooth-spectral-projection}. By unitarity of the propagator, the following also holds
\begin{equation}\label{e:unitarité-de-la-solution}
    \big\|(\mathrm{Id}-\kappa(\widehat H_h))u_h(t)\big\|_{L^2}
=\big\|(\mathrm{Id}-\kappa(\widehat H_h))u_h^{(0)}\big\|_{L^2}.
\end{equation}
Let $a \in \mathcal{C}_c^\infty(\T^2\times \R^2)$ such that $\operatorname{supp}(a) \subset \T^2 \times \{\xi \in \R^2 ~\lvert~H(\xi) \notin [E_1,E_2] ~\}$ and set $\mathcal{K}:=\big\{\,H(\xi)\,:\  (x,\xi)\in \operatorname{supp} (a)\,\big\}.$ To prove the lemma, it is enough to show that
\begin{equation}\label{e:preuve-support-mesure}
\big\langle \Op(a)\,u_h(t),u_h(t)\big\rangle \xrightarrow[h \to 0^+]{} 0
\end{equation}
uniformly in $t\in \R$. Indeed, for every $\psi\in \mathcal{C}_c^\infty(\R)$, one then has
\[
\langle W_h^B(\tau_h),\psi(t)a(x,\xi)\rangle
=
\int_\R \psi(t)\big\langle \Op(a)\,u_h(t\tau_h),u_h(t\tau_h)\big\rangle\dd t
\xrightarrow[h \to 0^+]{} 0.
\]
As this is valid for any $a$ with the above support properties, it implies that $W^B$ does not charge the region
\[
\R\times \T^2\times \{\xi\in\R^2:\ H(\xi)\notin [E_1,E_2]\}.
\]
We now prove \eqref{e:preuve-support-mesure}. Since $\operatorname{supp}(a)$ is compact and $H$ is continuous, $\mathcal{K}$ is a compact subset of $\R\setminus [E_1,E_2]$.
Moreover, by the support assumption on $a$, one has $\mathcal{K}\cap [E_1,E_2]=\emptyset$. Hence, the distance is positive:
\[
d \ :=\ \text{dist}\big(\mathcal{K},[E_1,E_2]\big)\ >\ 0.
\]
Choose $\kappa\in \mathcal{C}_c^\infty(\R)$ such that
\begin{equation*}\label{eq:chi-choice}
0\le \kappa\le 1,\qquad \kappa\equiv 1 \text{ on }[E_1,E_2],
\qquad \operatorname{supp}(\kappa) \subset (E_1-\tfrac d3,\ E_2+\tfrac d3).
\end{equation*}
Next choose $\varphi\in \mathcal{C}_c^\infty(\R)$ such that
\begin{equation*}\label{eq:phi-choice}
\varphi\equiv 1 \text{ on a neighbourhood of }\mathcal{K},
\qquad
\operatorname{supp}(\varphi)\subset (-\infty,E_1-\tfrac d2]\ \cup\ [E_2+\tfrac d2,\infty).
\end{equation*}
Since $\operatorname{supp}(\varphi) \cap \operatorname{supp}(\kappa)=\varnothing$, one has $\varphi\kappa\equiv 0$. Therefore, by the spectral theorem,
\[
\varphi(\widehat H_h)\kappa(\widehat H_h)=0.
\]
Moreover, since $\varphi\equiv 1$ near $\mathcal{K}$, we have the pointwise identity
\begin{equation}\label{eq:a-factor}
a(x,\xi)=a(x,\xi)\,\varphi(H(\xi))\qquad\text{for all }\xi \in \R^2.
\end{equation}
We also have
\[
\varphi(\widehat H_h)u_h(t)
=\varphi(\widehat H_h)\bigl(\mathrm{Id}-\kappa(\widehat H_h)\bigr)u_h(t).
\]
Since \eqref{e:unitarité-de-la-solution} holds and $\varphi(\widehat H_h)$ is bounded on $L^2$, one has
\begin{equation}\label{eq:phiPh-small}
\|\varphi(\widehat H_h)u_h(t)\|_{L^2}
\ \le\ \|\varphi(\widehat H_h)\|_{\mathcal L(L^2)}\,\|(\mathrm{Id}-\kappa(\widehat H_h))u_h(t)\|_{L^2}
\ =\mathcal{O}\!\left(\|(\mathrm{Id}-\kappa(\widehat H_h))u_h^{(0)}\|_{L^2}\right),
\end{equation}
uniformly in $t\in \R$. By \eqref{e:smooth-spectral-projection}, the right-hand side tends to $0$ as $h\to0^+$. One gets
\begin{equation}\label{eq:phiPh-small-limit}
\|\varphi(\widehat H_h)u_h(t)\|_{L^2}\xrightarrow[h\to0]{}0.
\end{equation}
Using \eqref{eq:a-factor}, the symbolic composition in the magnetic calculus \eqref{e:premier-ordre-composition} and the Calderón--Vaillancourt Theorem \ref{t:Calderon-Vaillancourt}, one has
\begin{equation}\label{eq:composition-a-phi}
\Op(a)=\Op(a)\,\Op(\varphi\circ H)+\mathcal{O}_{L^2 \to L^2}(h).
\end{equation}
On the other hand, applying Corollary \ref{c:functional-calculus}, one gets 
\begin{equation}\label{eq:replace-phi}
\Op(\varphi\circ H)=\varphi(\widehat H_h)+\mathcal{O}_{L^2 \to L^2}(h).
\end{equation}
Combining \eqref{eq:composition-a-phi}--\eqref{eq:replace-phi} gives
\[
\big\langle \Op(a)u_h(t),u_h(t)\big\rangle
=
\big\langle \Op(a)\,\varphi(\widehat H_h)u_h(t),u_h(t)\big\rangle + \mathcal{O}(h).
\]
By Calderón--Vaillancourt, $\Op(a)$ is uniformly bounded on $L^2$ (because $a\in \mathscr{S}^0$ due to its compact support). We infer by Cauchy--Schwarz that 
\[
\big|\big\langle \Op(a)u_h(t),u_h(t)\big\rangle\big|
\ \le\ C\,\|\varphi(\widehat H_h)u_h(t)\|_{L^2}+\mathcal{O}(h)
\ \xrightarrow[h\to0]{}\ 0
\]
thanks to \eqref{eq:phiPh-small-limit}. This proves \eqref{e:preuve-support-mesure} and concludes the proof of the lemma.
\end{proof}

Now that the phase-space distribution $W^B$ has been constructed, we make explicit its link with the configuration-space limits $\nu_t$ introduced above. In particular, the next lemma shows that $W^B$ admits a disintegration with respect to time,
\[
W^B(\dd t,\dd x,\dd\xi)=\nu_t^B(\dd x,\dd\xi)\otimes \dd t,
\]
where $(\nu_t^B)_{t\in\R}$ is a measurable family of probability measures on $T^*\T^2$. We denote by 
\begin{equation}\label{e:def-projection}
    \pi_x : (x,\xi) \in \T^2 \times \R^2 \mapsto x, 
\end{equation}
the canonical projection onto the configuration space. Then, the following lemma also identifies the projection of $\nu_t^B$ onto $\T^2$ with the original time-dependent quantum limit $\nu_t$.

\begin{lemma}[Time disintegration and projection]\label{l:Lifting-QL}
Let $(u_h)_{h \to 0^+}$ solve \eqref{e:Schrodinger-PDE} such that the initial data are spectrally localized with respect to the magnetic operator $\widehat H_h$, in the sense of \eqref{e:spectral-projection}. 
Then any accumulation point $W^B$ in $\mathcal{D}'(\R\times T^*\T^2)$ of $W_h^B(\tau_h)$ is of the form
\begin{equation}\label{e: W-proba}
    W^B(\dd t,\dd x,\dd\xi)=\nu_t^B(\dd x,\dd\xi)\otimes\dd t,
\end{equation}
where $(\nu_t^B)_{t\in\R}$ is a measurable family of probability measures on $\T^2 \times \Omega_{E_1,E_2}$. Moreover, for almost every $t\in\R$, the projection onto the base equals the configuration-space limit:
\begin{equation*}
    (\pi_x)_*\nu_t^B=\nu_t.
\end{equation*}
\end{lemma}

\begin{proof} 
Let
\begin{equation}\label{e:time-projection}
    \tilde{\pi}_t : (t,x,\xi) \in \R \times T^*\T^2 \mapsto t
\end{equation}
be the projection onto the time variable and define $\omega^B$ the time marginal of $W^B$ as follows 
\begin{equation}\label{e:time-marginal-of-W^B}
    \omega^B = (\tilde{\pi}_t)_*W^B.
\end{equation} 
From the ellipticity property \eqref{e:definition-elliptic} of the Hamiltonian $H$, the set $\T^2 \times H^{-1}([E_1,E_2])$ is compact. Since $W^B$ is a Radon measure, hence a $\sigma$-finite measure, for every $I\subset \R$ compact, using Lemma \ref{l:support-de-W^B} one has
\begin{equation}\label{e:omega-est-sigma-finie}
    \omega^B(I) = W^B(I\times \T^2 \times H^{-1}([E_1,E_2])) <\infty. 
\end{equation}
Hence, the marginal $\omega^B$ is a $\sigma$-finite measure and applying the disintegration Theorem \ref{t:desintegration} with respect to $\omega^B$ leads to $W^B(\dd t,\dd x ,\dd \xi) =\nu_t^B(\dd x ,\dd \xi)\,\omega^B(\dd t)$ with $(\nu_t^B)_t$ a $\omega^B$-measurable family of probability measures on $\T^2 \times \Omega_{E_1,E_2}$. We now identify $\omega^B$ and rule out a loss of mass at infinity in
the momentum variable. Fix
$\kappa\in\mathcal C_c^\infty(\R)$ such that
$\kappa\equiv1$ on a neighborhood of $[E_1,E_2]$. By the functional
calculus in the magnetic pseudodifferential calculus
(see Corollary~\ref{c:functional-calculus}), one has
\begin{equation}\label{e:formule-composition-pseudo}
\Op(\kappa(H))
=
\kappa\!\left(\widehat H_h\right)+\mathcal O(h)
\qquad\text{in }\mathcal L(L^2).
\end{equation} 
Moreover, the spectral localization assumption
\eqref{e:smooth-spectral-projection} implies that
\[
\kappa\!\left(\widehat H_h\right)u_h^{(0)}
=
u_h^{(0)}+o_{L^2}(1).
\]
Since \(\kappa(\widehat H_h)\) commutes with the propagator, it follows that
\begin{equation}\label{e:spectral-localization-all-times}
\kappa\!\left(\widehat H_h\right)u_h(t\tau_h)
=
u_h(t\tau_h)+o_{L^2}(1),
\end{equation}
uniformly with respect to \(t\in\R\). Let \(\zeta\in\mathcal C_c^\infty(\R)\) satisfy
\[
0\leq\zeta\leq1,
\qquad
\zeta \equiv1\ \text{on }[-1,1],
\qquad
\zeta\equiv0\ \text{on }\R\setminus[-2,2],
\]
and, for \(R>0\), set
\[
\zeta_R(\xi)
:=
\zeta\left(\frac{\lvert\xi\rvert}{R}\right).
\]
Since \(H\) is elliptic and \(\kappa\) is compactly supported, the symbol \(\kappa(H(\xi))\) is compactly supported in \(\xi\). Hence, there exists
\(R_0>0\) such that, for every \(R\geq R_0\),
\begin{equation}\label{e:cutoff-product-zero}
\bigl(1-\zeta_R(\xi)\bigr)\kappa(H(\xi))=0.
\end{equation}
Using \eqref{e:formule-composition-pseudo}, the composition formula \eqref{e:premier-ordre-composition}, and the Calderón--Vaillancourt
Theorem, we obtain, for every fixed \(R\geq R_0\),
\begin{align*}
\Op\bigl(1-\zeta_R\bigr)\kappa\bigl(\widehat H_h\bigr)
&=
\Op\bigl(1-\zeta_R\bigr)\Op\bigl(\kappa(H)\bigr)
+
\mathcal O_{L^2\to L^2}(h)
\\
&=
\Op\Bigl(\bigl(1-\zeta_R\bigr)\kappa(H)\Bigr)
+
\mathcal O_{L^2\to L^2}(h)
\\
&=
\mathcal O_{L^2\to L^2}(h).
\end{align*}
Together with \eqref{e:spectral-localization-all-times}, this yields
\begin{equation}\label{e:no-escape-momentum}
\sup_{t\in\R}
\left|
\left\langle
\Op\bigl(1-\zeta_R\bigr)u_h(t\tau_h),
u_h(t\tau_h)
\right\rangle_{L^2}
\right|
=o(1)
\end{equation}
as \(h\to0^+\), for every fixed \(R\geq R_0\). Now fix \(\psi\in\mathcal C_c^\infty(\R)\). Since
\[
1=\zeta_R(\xi)+\bigl(1-\zeta_R(\xi)\bigr),
\]
the normalization of \(u_h(t\tau_h)\) and
\eqref{e:no-escape-momentum} imply that
\begin{align*}
\left\langle
W_h^B(\tau_h),
\psi\times \zeta_R
\right\rangle
&=
\int_\R
\psi(t)
\left\langle
\Op\bigl(\zeta_R\bigr)u_h(t\tau_h),
u_h(t\tau_h)
\right\rangle_{L^2}
\dd t
\\
&=
\int_\R\psi(t)\,\dd t+o(1).
\end{align*}
Passing to the limit along the subsequence defining \(W^B\), we obtain
\begin{equation}\label{e:time-marginal-cutoff}
\int_{\R\times T^*\T^2}
\psi(t)\zeta_R(\xi)\,
W^B(\dd t,\dd x,\dd\xi)
=
\int_\R\psi(t)\,\dd t.
\end{equation}
For \(R\) sufficiently large, \(\zeta_R\equiv1\) on
\(\Omega_{E_1,E_2}\). Since \(W^B\) is supported in
\(\R\times\T^2\times\Omega_{E_1,E_2}\), the left-hand side of
\eqref{e:time-marginal-cutoff} is equal to
\[
\int_\R\psi(t)\,\omega^B(\dd t).
\]
Therefore,
\[
\forall \psi \in \mathcal{C}_c^\infty(\R), \qquad \int_\R\psi(t)\,\omega^B(\dd t)
=
\int_\R\psi(t)\,\dd t,
\]
which leads to
\[
\omega^B=\dd t.
\]
Since \(\nu_t^B\) is a probability measure for \(\omega^B\)-almost every \(t\), and since \(\omega^B=\dd t\), we finally obtain
\begin{equation}\label{e:W^B=nu_t dt}
W^B(\dd t,\dd x,\dd\xi)
=
\nu_t^B(\dd x,\dd\xi)\otimes \dd t ,
\end{equation}
where \((\nu_t^B)_{t\in\R}\) is a measurable family of probability
measures on \(\T^2\times\Omega_{E_1,E_2}\), defined for almost every
\(t\in\R\).

To prove that $(\pi_x)_*\nu_t^B=\nu_t$, fix $\psi \in \mathcal{C}_c^\infty(\R)$ and $b \in \mathcal{C}^\infty(\T^2)$. Then one has 
\begin{align*}
\int_\R \psi(t)\,\langle (\pi_x)_*\nu_t^B, b \rangle\,\dd t = \int_\R \psi(t)\left[\int_{T^*\T^2} b(x)\,\nu_t^B(\dd x,\dd\xi)\right]\dd t .
\end{align*}
Insert the energy cut-off $\kappa(H(\xi))$, which is equal to $1$ on the support of $\nu_t^B$:
\begin{align*}
\int_\R \psi(t)\Big[\int_{T^*\T^2} b(x)\,\nu_t^B(\dd x,\dd\xi)\Big]\dd t
&= \int_\R \psi(t)
\left[\int_{T^*\T^2} b(x)\kappa\!\left(H(\xi)\right)
\nu_t^B(\dd x,\dd\xi)\right]\dd t\\
&= \lim_{h\to0^+}
\int_\R \psi(t)\,
\big\langle \Op\big(b(x)\kappa\!\left(H(\xi)\right)\big)
u_h(t\tau_h),u_h(t\tau_h)\big\rangle_{L^2}\,\dd t.
\end{align*}
Using \eqref{e:formule-composition-pseudo} and the composition formula at first order \eqref{e:premier-ordre-composition} together with the Calderón-Vaillancourt Theorem, one has
\begin{align*}
&\int_\R \psi(t)\,
\big\langle \Op\big(b(x)\kappa\!\left(H(\xi)\right)\big)
u_h(t\tau_h),u_h(t\tau_h)\big\rangle_{L^2}\,\dd t\\
&\quad = \int_\R \psi(t)\,
\big\langle b(x)\kappa\!\left(\widehat H_h\right)
e^{-i\frac{t\tau_h}{h}\widehat H_h}u_h^{(0)},
e^{-i\frac{t\tau_h}{h}\widehat H_h}u_h^{(0)}\big\rangle_{L^2}\,\dd t
+\mathcal O(h).
\end{align*}
By unitarity of the propagator \eqref{e:propagator} and the property \eqref{e:smooth-spectral-projection} on the spectral localization, we can replace $\kappa(\widehat H_h)u_h^{(0)}$ by $u_h^{(0)}$ in the limit as $h\to0^+$. Thus,
\[
\int_\R \psi(t)\left[\int_{T^*\T^2} b(x)\,\nu_t^B(\dd x,\dd\xi)\right]\dd t
= \lim_{h\to0^+}\int_\R \psi(t)\,
\langle b(x)u_h(t\tau_h),u_h(t\tau_h)\rangle_{L^2}\,\dd t.
\]
By the definition of the time-dependent quantum limits $(\nu_t)_{t\in\R}$, the right-hand side is equal to
\[
\int_\R \psi(t)\left[\int_{\T^2} b(x)\,\nu_t(\dd x)\right]\dd t.
\]
Since this holds for all $\psi$ and $b$, we conclude that
\begin{equation}\label{e:proj-wigner=quantum}
    (\pi_x)_*\nu_t^B=\nu_t
\qquad \text{for almost every } t\in\R.
\end{equation}
\end{proof}

\begin{remark}\label{rem:spectral-localization}
In contrast with the $h$--oscillation assumptions commonly used in long-time semiclassical analysis, we impose here the stronger spectral localization condition
\[
\lim_{h\to0^+}
\Big\|
\mathbbm{1}_{[E_1,E_2]}(\widehat H_h)u_h^{(0)}-u_h^{(0)}
\Big\|_{L^2(\T^2,L)}=0.
\]
Under our assumptions on the energy window $[E_1,E_2]$, this guarantees that any semiclassical measure associated with the initial data is supported in
\[
\Omega_{E_1,E_2} = H^{-1}([E_1,E_2]).
\]
Moreover, the energy window $[E_1,E_2]$ is chosen so that no critical point of $H$ lies in $\Omega_{E_1,E_2}$. Therefore, the measure $\nu_t^B$ does not charge the critical set $\{\nabla H=0\}$. This restriction is convenient for our analysis, since it excludes a possible component of the limiting measures concentrated on $\{\nabla H=0\}$, for which no general description would be directly available through the upcoming arguments if one assumes only $h$--oscillation. In this sense, the spectral localization, together with the assumption $\nabla H\neq 0$ allows us to restrict the analysis from the beginning to the regular part of the energy surface.
\end{remark}

In order to formulate the propagation properties of $\nu_t^B$ in terms of an initial phase-space measure, we also introduce the auxiliary space--time distribution associated with the initial data: for $a\in \mathcal{C}_c^\infty(T^*\T^2)$,
\begin{equation}\label{e:WB-def-donnee-initiale}
\langle W_{h,0}^B,a\rangle
:=
\big\langle \Op(a)u_h^{(0)},u_h^{(0)}\big\rangle_{L^2(\T^2,L)}.
\end{equation}
Possibly after extracting a further subsequence and using the same arguments as in Lemmas \ref{l:magnetic-wigner} and \ref{l:Lifting-QL}, $(W_{h,0}^B)_{h \to 0^+}$ converges to a probability measure $\nu_0^B$ on $\T^2 \times \Omega_{E_1,E_2}$ satisfying
\begin{equation}\label{e:projection-nu_0}
(\pi_x)_*\nu_0^B=\nu_0.
\end{equation}
Before turning to the propagation properties of the family $(\nu_t^B)_{t\in \R}$, we record a complementary rigidity property of its $\xi$-marginal. Since the principal Hamiltonian $H$ depends only on $\xi$, the $\xi$-distribution can only be affected by the subprincipal perturbation $hR_h$, which becomes visible at the scale $1/h$.

\begin{lemma}[Description of the $\xi$-marginal]\label{l:xi-marginal}
Let $\pi_\xi : (x,\xi)\in T^*\T^2 \longmapsto \xi$ be the canonical projection onto the frequency variable. For almost every $t\in\R$, define
\begin{equation}\label{e:definition-measure-lambda}
\lambda_t := (\pi_\xi)_* \nu_t^B \qquad \text{and}\qquad \lambda_0 := (\pi_\xi)_* \nu_0^B.
\end{equation}
Then for almost every $t\in\R$, the following holds: 
\begin{enumerate}
\item There exists a $\lambda_t$-measurable family $(\nu_{t,\xi}^B)_\xi$ of probability measures on $\T^2$ such that 
\begin{equation}\label{e:desintegration-par-rapport-a-lambda}
\nu_t^B(\dd x,\dd \xi)=\nu_{t,\xi}^B(\dd x)\otimes\lambda_t(\dd \xi).
\end{equation}

\item There exist a probability measure $\sigma^B$ on $[E_1,E_2]$ and a $\sigma^B-$measurable family $(\lambda_{t,E})_{E\in [E_1,E_2]}$ of probability measures supported on $\SE$, such that 

\begin{equation}\label{e:desintegration-lambda-via-sigma}
    \lambda_t(\dd \xi) = \int_{E_1}^{E_2}\lambda_{t,E}(\dd \xi) \sigma^B(\dd E).
\end{equation}
Moreover, the measure $\sigma^B$ is given by 
\begin{equation}\label{e:sigma-indep-du-temps}
    \sigma^B = H_*\lambda_t = H_*\lambda_0 \qquad\text{for a.e. }t\in\R.
\end{equation}
\item If the time scale satisfies $\tau_h \ll h^{-1}$, then one has 
\begin{equation}\label{e:lambda-indep-du-temps}
    \lambda_t=\lambda_0:= (\pi_\xi)_* \nu_0^B \qquad\text{for a.e. }t\in\R.
\end{equation}
\end{enumerate}
\end{lemma}

\begin{remark} 
Equivalently, combining items 1 and 2 for $a \in \mathcal{C}_c^\infty(T^*\T^2)$, leads, for almost every $t \in \R$, to 
\[
\int_{T^*\T^2} a(x,\xi)\,\nu_t^B(\dd x,\dd \xi)
=
\int_{E_1}^{E_2}\int_{\mathbb S_E}\int_{\T^2}
a(x,\xi)\,\nu_{t,\xi}^B(\dd x)\,\lambda_{t,E}(\dd \xi)\,\sigma^B(\dd E).
\]
The remainder of the analysis will consist in identifying these conditional measures on sufficiently long time scales: first $\nu_{t,\xi}^B$ in the configuration variable, and then $\lambda_{t,E}$ along the energy levels.
\end{remark}

\begin{proof}
Since $\nu_t^B$ is a probability measure on $T^*\T^2$, one can apply the disintegration Theorem~\ref{t:desintegration}, with respect to $\lambda_t$. This gives a $\lambda_t$-measurable family $(\nu_{t,\xi}^B)_\xi$ of probability measures on $\T^2$ such that 
\begin{equation}\label{e:desintegration-par-rapport-a-lambda-t}
    \nu_t^B(\dd x , \dd \xi) = \nu_{t,\xi}^B(\dd x) \otimes \lambda_t(\dd \xi).
\end{equation}
We now prove the second item. Let $\psi,f\in\mathcal C_c^\infty(\R).$
By the semiclassical functional calculus, as in the proof of
Lemma~\ref{l:Lifting-QL}, one has
\begin{equation}\label{e:functional-calculus-energy-marginal}
\int_\R
\psi(t)
\left\langle
\Op(f\circ H)u_h(t\tau_h),
u_h(t\tau_h)
\right\rangle_{L^2}
\dd t =
\int_\R
\psi(t)
\left\langle
f(\widehat H_h)u_h(t\tau_h),
u_h(t\tau_h)
\right\rangle_{L^2}
\dd t
+
\mathcal O(h).
\end{equation}
Using the unitarity of the propagator and the fact that \(f(\widehat H_h)\) commutes with the propagator, one can write \eqref{e:functional-calculus-energy-marginal} as follows
\begin{align*}
\int_\R
\psi(t)
\left\langle
\Op(f\circ H)u_h(t\tau_h),
u_h(t\tau_h)
\right\rangle_{L^2}
\dd t
=
\left(\int_\R\psi(t)\,\dd t\right)
\left\langle
\Op(f\circ H)u_h^{(0)},
u_h^{(0)}
\right\rangle_{L^2}
+
\mathcal O(h).
\end{align*}
Passing to the limit along the subsequence defining \(\lambda_t\) and \(\lambda_0\), one gets
\begin{equation}\label{e:energy-marginal-weak-conservation}
\int_\R
\psi(t)
\left(
\int_{\R^2}
f(H(\xi))\,\lambda_t(\dd\xi)
\right)
\dd t
=
\left(\int_\R\psi(t)\,\dd t\right)
\int_{\R^2}
f(H(\xi))\,\lambda_0(\dd\xi).
\end{equation}
For almost every \(t\in\R\), set
\[
\sigma_t^B:=H_*\lambda_t,
\qquad
\sigma_0^B:=H_*\lambda_0.
\]
Since \(\lambda_t\) and \(\lambda_0\) are probability measures supported in
\(\Omega_{E_1,E_2}\), the measures \(\sigma_t^B\) and \(\sigma_0^B\)
are probability measures supported in \([E_1,E_2]\). Hence, 
\eqref{e:energy-marginal-weak-conservation} becomes
\[
\int_\R
\psi(t)
\left(
\int_{E_1}^{E_2}
f(E)\,\sigma_t^B(\dd E)
\right)
\dd t
=
\left(\int_\R\psi(t)\,\dd t\right)
\int_{E_1}^{E_2}
f(E)\,\sigma_0^B(\dd E).
\]
Hence, for every fixed \(f\in\mathcal C_c^\infty(\R)\), one has 
\[
\int_{E_1}^{E_2}
f(E)\,\sigma_t^B(\dd E)
=
\int_{E_1}^{E_2}
f(E)\,\sigma_0^B(\dd E)
\]
for almost every \(t\in\R\). 
Since this holds for every \(f\in\mathcal C_c^\infty(\R)\), it follows that
\[
\sigma_t^B=\sigma_0^B
\qquad
\text{for almost every }t\in\R.
\]
Setting
\[
\sigma^B:=\sigma_0^B=H_*\lambda_0,
\]
we have proved that
\begin{equation}\label{e:sigma-indep-du-temps-proof}
H_*\lambda_t
=
H_*\lambda_0
=
\sigma^B
\qquad
\text{for almost every }t\in\R.
\end{equation}

For such a \(t\), apply the disintegration
Theorem~\ref{t:desintegration} to the probability measure \(\lambda_t\)
and to the map
\[
H:\Omega_{E_1,E_2}\longrightarrow [E_1,E_2].
\]
Since \(H_*\lambda_t=\sigma^B\), there exists a
\(\sigma^B\)-measurable family
\((\lambda_{t,E})_{E\in[E_1,E_2]}\) of probability measures supported on $\SE$ such that \eqref{e:desintegration-lambda-via-sigma} holds.

To prove the third item, fix $a\in \mathcal{C}_c^\infty(\R^2)$ and set 
\[
F_h(t) := \langle \Op(a)u_h(t \tau_h),u_h(t \tau_h)\rangle.
\]
Using the fact that $(u_h)_{h \to 0^+}$ is a solution to \eqref{e:Schrodinger-PDE}, one obtains 
\begin{equation}\label{e:derivative-Fh}
\frac{\dd}{\dd t}F_h(t)
=
-\tau_h
\Big\langle
\frac{i}{h}[\Op(a),\widehat H_h]
u_h(t\tau_h),
u_h(t\tau_h)
\Big\rangle_{L^2}.
\end{equation}
Using the commutator expansion \eqref{e:commutator}, the Calderón--Vaillancourt Theorem \ref{t:Calderon-Vaillancourt} and the fact that $a$ and $H$ depend only on the $\xi$-variable, one has
\begin{equation}\label{e:ordre-1-commutateur}
    \frac{i}{h}\,[\Op(a),\widehat H_h]
=h\,\Op(\{a,R_h\})+\mathcal O_{L^2 \to L^2}(h).
\end{equation}
Since \(a\) is compactly supported, the symbol \(\{a,R_h\}\) belongs to \(\mathscr S^0\) uniformly in \(h\).  Hence, combining \eqref{e:derivative-Fh} and \eqref{e:ordre-1-commutateur} with the Calderón--Vaillancourt Theorem and the unitarity of the propagator, one gets
\begin{equation}\label{e:estimate-derivative-Fh}
\left| \frac{\dd}{\dd t}F_h(t) \right| \leq C_a h\tau_h,
\end{equation}
uniformly in \(t\). Integrating \eqref{e:estimate-derivative-Fh} between \(0\) and \(t\) yields
\begin{equation}\label{e:Fh-close-initial}
|F_h(t)-F_h(0)|
\leq C_a |t|\,h\tau_h.
\end{equation}
Let now \(\psi\in\mathcal C_c^\infty(\R)\). It follows from
\eqref{e:Fh-close-initial} that
\[
\left| \int_\R \psi(t)\big(F_h(t)-F_h(0)\big)\,\dd t \right| \leq C_a h\tau_h \int_\R |t\psi(t)|\,\dd t.
\]
Since \(\tau_h\ll h^{-1}\), the right-hand side tends to \(0\) as
\(h\to0^+\). Moreover, by the definition of the initial momentum marginal \(\lambda_0\),
\[
F_h(0)
=
\big\langle
\Op(a)u_h^{(0)},u_h^{(0)}
\big\rangle_{L^2}
\longrightarrow
\int_{\R^2}a(\xi)\,\lambda_0(\dd\xi).
\]
Consequently, one has
\[
\int_\R \psi(t)F_h(t)\,\dd t \longrightarrow \left(\int_\R\psi(t)\,\dd t\right) \left( \int_{\R^2}a(\xi)\,\lambda_0(\dd\xi) \right).
\]
On the other hand, by the definition of the time-dependent magnetic semiclassical measure, one has
\[
\int_\R \psi(t)F_h(t)\,\dd t \longrightarrow \int_\R \psi(t) \left(
\int_{\R^2}a(\xi)\,\lambda_t(\dd\xi) \right)
\dd t.
\]
Hence, one obtains
\[
\int_\R
\psi(t)
\left(
\int_{\R^2}a(\xi)\,\lambda_t(\dd\xi)
-
\int_{\R^2}a(\xi)\,\lambda_0(\dd\xi)
\right)
\dd t
=0.
\]
Since this holds for every
\(\psi\in\mathcal C_c^\infty(\R)\) and for every
\(a\in\mathcal C_c^\infty(\R^2)\), it follows that
\[
\lambda_t=\lambda_0
\qquad
\text{for almost every }t\in\R,
\]
which proves \eqref{e:lambda-indep-du-temps}.
\end{proof}

The previous lemma shows that the energy marginal of the limiting measure is preserved along the evolution and below the scale $h^{-1}$, the $\xi$-marginal of the lifted measure is entirely determined by the initial data. We also proved that the disintegration of $\nu_t^B$ with respect to the energy variable involves a measure on $\R$ which is completely determined by the initial data.

In order to describe $\mathcal{N}(\widehat H_h, \tau_h)$, it is sufficient to describe the $x$-marginal of $\nu_t^B$ which is easier to work with. Having identified the lifted family of phase-space measures $(\nu_t^B)_{t\in\R}$ and its link with the configuration-space limits $(\nu_t)_{t\in\R}$, we now describe the dynamical constraints satisfied by $\nu_t^B$. More precisely, we study how these constraints depend on the observation scale $\tau_h$. This will clarify in which regimes the limiting measures are transported by the Hamiltonian flow of the principal symbol $H$, and in which regimes they satisfy an invariance property under this flow. It will also illustrate the kind of arguments that will have to be refined later on.

\begin{lemma}[Propagation and invariance]\label{l:Dynamical-regimes}
Let $(\nu_t^B)_{t\in\R}$ be an accumulation point of $(W_h^B(\tau_h))_{h \to 0^+}$ in the sense of \eqref{e:W^B=nu_t dt} and let $\varphi_H^t$ be the Hamiltonian flow of $H(\xi)$ on $T^*\T^2$, namely 
\[ 
\varphi_H^t(x,\xi)=(x+t\nabla H(\xi),\xi).
\]
Then, the following holds: 
\begin{enumerate}
\item If $\tau_h \ll 1$, then $\nu_t^B$ is independent of $t$, i.e., $\nu_t^B = \nu_0^B$.
\item If $\tau_h=1$, then $\nu_t^B$ is transported by the Hamiltonian flow of $H$, i.e. $\nu_t^B=(\varphi_H^t)_*\nu_{0}^B$.
\item If $\tau_h \gg 1$, then for a.e.\ $t$, the measure $\nu_t^B$ is invariant under the flow of $H$, i.e. for all $s\in\R$,
\[
(\varphi_H^s)_*\nu_t^B=\nu_t^B .
\]
\end{enumerate}
\end{lemma}
\begin{remark}
    Note that the relations given in the first and second point of Lemma \ref{l:Dynamical-regimes} can be written in terms of Fourier coefficients for the measure $\nu_t^B$. In the sense of distributions, one can decompose $\nu_t^B$ as a Fourier series:
    \begin{equation}\label{e:nu-t^b=Fourier}
        \displaystyle \nu_t^B(x,\xi) = \sum_{k \in \Z^2} \widehat \nu_t^B(k,\xi) e^{2i\pi k\cdot x}. 
    \end{equation}
    Hence, the first relation of the previous lemma gives us that for almost every $t \in \R$, one can identify the $k$-th Fourier coefficient of $\nu_t^B$ with its initial datum $\nu_0^B$. The second one leads us to the following equality, for almost every $t \in \R$, one has
    \begin{center}
        $\displaystyle \widehat\nu_t^B(k,\xi) = e^{-2i\pi tk\cdot \nabla H(\xi)}\widehat\nu_0^B(k,\xi)$, 
    \end{center}
    in particular, \eqref{e:nu-t^b=Fourier} becomes
    \begin{equation*}
        \displaystyle \nu_t^B(x,\xi) = \sum_{k \in \Z^2} e^{-2i\pi tk\cdot \nabla H(\xi)}\widehat\nu_0^B(k,\xi) e^{2i\pi k\cdot x}.
    \end{equation*}
\end{remark}
\begin{proof}
Fix $\psi\in \mathcal{C}_c^\infty(\R)$ and $b\in \mathcal{C}_c^\infty(T^*\T^2)$. Differentiate
$t\mapsto \langle \Op(b)u_h(t \tau_h),u_h(t \tau_h)\rangle$ using the fact that $(u_h)_{h \to 0^+}$ is a solution to \eqref{e:Schrodinger-PDE} and integrate by parts:
\[
\int_\R \psi'(t)\,\langle \Op(b)u_h(t \tau_h),u_h(t \tau_h)\rangle\,\dd t
=
\tau_h\int_\R \psi(t)\Big\langle \frac{i}{h}\,[\Op(b),\widehat H_h]\,u_h(t \tau_h),
u_h(t \tau_h)\Big\rangle\,\dd t.
\]
Arguing as in \eqref{e:ordre-1-commutateur}, one has 
\begin{equation}\label{e:commutator-identity}
\begin{aligned}
\int_\R \psi'(t)\,\big\langle \Op(b)\,u_h(t\tau_h),u_h(t\tau_h)\big\rangle\,\dd t 
&= \tau_h \int_\R \psi(t)\,\big\langle \Op(\{b,H\})\,u_h(t\tau_h),u_h(t\tau_h)\big\rangle\,\dd t \\
&\quad + h\tau_h \int_\R \psi(t)\,\big\langle \Op(\{b,R_h\})\,u_h(t\tau_h),u_h(t\tau_h)\big\rangle\,\dd t + \mathcal{O}(h\tau_h).
\end{aligned}
\end{equation}

Since $b \in \mathcal{C}_c^\infty(\T^*\T^2)$, the symbols $\{b,H\}$ and $\{b,R_h\}$ are both in $\mathscr{S}^0(\T^2 \times \R^2)$ so that one can use the Calderón--Vaillancourt Theorem \ref{t:Calderon-Vaillancourt}. Then, if $\tau_h\to0$ we obtain $\int \psi'(t)\langle\nu_t^B,b\rangle\,\dd t=0$ for all $\psi$, hence $\nu_t^B$ is constant in $t$. If $\tau_h= 1$, we obtain the weak transport equation $\partial_t \nu_t^B= \{\nu_t^B,H\}$, which means precisely that the family $(\nu_t^B)_{t\in\R}$ is the image of the initial measure $\nu_0^B$ under the Hamiltonian flow, namely $\nu_t^B = (\varphi_H^t)_* \nu_0^B.$ Finally, if $\tau_h\to+\infty$, one can divide by $\tau_h$ in \eqref{e:commutator-identity} and let $h$ go to $0$ to obtain $\{\nu_t^B,H\}=0$ for a.e.\ $t$ i.e., invariance under the Hamiltonian flow $\varphi_H^s$.
\end{proof}

\subsection{Decomposition of invariant measures}\label{s:decomposition}
The long-time regime $\tau_h\to+\infty$ is the main focus of this paper: in that case, $\nu_t^B$ is invariant under $\varphi_H^s$ and can be analyzed further by decomposing phase space according to resonant directions of the integrable flow. This is the starting point of the next section. This invariant setting provides a natural and robust framework for further analysis. In the next part, we follow the strategy developed by Macià, Anantharaman-Macià and Anantharaman-Fermanian-Macià in \cite{Macia-2009, AnantharamanMacia11, Anantharaman-Macia-2014, Anantharaman-Fermanian-Kammerer-Macia-2015} to refine the description of the invariant measures $\nu_t^B$, by introducing finer time scales and studying how additional asymptotic regimes influence their structure.

On $\T^2\times \R^2$, the trajectories of the Hamiltonian flow
\[
\varphi_H^s(x,\xi)=(x+s\nabla H(\xi),\xi)
\]
are either dense, periodic or fixed depending on arithmetic properties of the direction $\nabla H(\xi)$. Following the approach of~\cite{Macia-2010,Anantharaman-Fermanian-Kammerer-Macia-2015}, we decompose the invariant measure according to these dynamical regimes.

\medskip

We denote by $\mathcal L_1$ the collection of primitive rank-one submodules $\Lambda\subset\Z^2$, i.e.\ such that $\dim\langle\Lambda\rangle=1$ and $\langle\Lambda\rangle\cap\Z^2=\Lambda$. For every $\Lambda\in\mathcal L_1$, we choose $\mathfrak e_\Lambda\in\Z^2\setminus\{0\}$ such that $\Lambda=\Z\mathfrak e_\Lambda$, we set $L_\Lambda:=|\mathfrak e_\Lambda|$, and we denote by $\Lambda^\perp:=\R\mathfrak e_\Lambda^\perp$ the orthogonal line (where $\mathfrak e_\Lambda^\perp$ is a fixed orthogonal vector of length $L_\Lambda$). We introduce the resonant sets 
\begin{equation}\label{e:definition-resonant-direction}
E_{\Lambda^\perp\setminus\{0\}}:=\{\xi\in\R^2:\ \nabla H(\xi)\in \Lambda^\perp\setminus\{0\}\},
\end{equation}
and the non-resonant set 
\begin{equation}\label{e:definition-nonresonant-direction}
\Omega_{\mathrm{nr}}
:=\Big\{\xi\in\R^2:\ k\cdot\nabla H(\xi)\neq0\ \text{for all }k\in\Z^2\setminus\{0\}\Big\}.
\end{equation}
For $\xi\in \Omega_{\mathrm{nr}}$, the orbit $x+s\nabla H(\xi)$ is dense in $\T^2$, whereas $\xi\in E_{\Lambda^\perp}$ imposes a rational constraint on the direction of motion and leads to periodic (linear) trajectories.

To make this dichotomy quantitative, let $a\in \mathcal{C}_c^\infty(T^*\T^2)$ and denote by
\[
\widehat a_k(\xi):=\int_{\T^2} a(x,\xi)\,e^{-2\pi i k\cdot x}\,\dd x,\qquad k\in\Z^2,
\]
its Fourier coefficients in the $x$-variable. For $\Lambda\in\mathcal L_1$, we define the $\Lambda$--averaging operator by
\begin{equation}\label{e:Ilambda}
\mathcal I_\Lambda(a)(x,\xi) :=\sum_{k\in\Lambda}\widehat a_k(\xi)\,e^{2\pi i k\cdot x}.
\end{equation}
Equivalently, $\mathcal I_\Lambda(a)$ keeps only the Fourier modes belonging to the submodule $\Lambda$, and it may be interpreted as the time average of $a$ along the linear flow when $\xi \in \Lambda^{\perp}$. More precisely, on the one hand, one has 
\begin{equation}\label{e:moyenne-le-long-des-orbites-periodiques}
    \displaystyle \forall (x,\xi)\in\T^2\times E_{\Lambda^\perp\setminus\{0\}},\quad  \lim_{T\rightarrow+\infty} 
\frac{1}{T}\int_0^Ta(x+t\nabla H(\xi),\xi)\dd t=\mathcal{I}_\Lambda(a)(x,\xi), 
\end{equation}
and on the other hand, one finds 
\begin{equation}\label{e:moyenne-le-long-des-orbites-denses}
    \displaystyle \forall (x,\xi) \in\T^2 \times \Omega_{\mathrm{nr}},\quad \lim_{T\rightarrow+\infty} 
\frac{1}{T}\int_0^Ta(x+t\nabla H(\xi),\xi)\dd t=\int_{\T^2}a(y,\xi)\dd y.
\end{equation}

\medskip
\begin{remark}
For $\tau_h \gg 1$, Lemma \ref{l:Dynamical-regimes} tells us that the semiclassical measure $\nu_t^B$ is invariant under the geodesic flow $\varphi_H^s$. Hence, the measure $\nu_t^B$ is a probability measure on $T^*\T^2$ that can be decomposed as follows: 
\begin{equation}\label{e:decomposition-de-nu_t^B-avec-point-critique}
    \displaystyle \nu_t^B =  \nu_t^B\lvert_{\T^2 \times \Omega_{\mathrm{nr}}} + \sum_{\Lambda\in\mathcal{L}_1}\nu_t^B|_{\T^2\times E_{\Lambda^\perp \setminus \{0\}}}+\nu_t^B\lvert_{\T^2 \times \{ \nabla H(\xi) =0\}}.
\end{equation}
Since we made the hypothesis \eqref{e:fenetre-energie} and the spectral localization \eqref{e:spectral-projection}, one knows from Lemma \ref{l:Lifting-QL} that the measure $\nu_t^B$ carries no mass on the set of the critical points of the Hamiltonian $H$ and therefore $\nu_t^B$ can be decomposed as follows
\begin{equation}\label{e:decomposition-de-nu_t^B-sans-point-critique}
    \displaystyle \nu_t^B =  \nu_t^B\lvert_{\T^2 \times \Omega_{\mathrm{nr}}} + \sum_{\Lambda\in\mathcal{L}_1}\nu_t^B|_{\T^2\times E_{\Lambda^\perp \setminus \{0\}}}.
\end{equation}
In what follows, we will mainly focus on the analysis of the regularity of each measure $\nu_t^B|_{\T^2\times E_{\Lambda^\perp \setminus \{0\}}}$. Our assumptions on the energy layers $H\in[E_1,E_2]$ are precisely made to avoid the fixed points of the Hamiltonian that would require a finer analysis.
\end{remark}

The next lemma summarizes the dynamical decomposition of $\varphi_H^s$--invariant measures on $T^*\T^2$ (see for instance~\cite[\S2]{Anantharaman-Fermanian-Kammerer-Macia-2015}).

\begin{lemma}\label{l:anantharamanmacia} Let $\mu$ be a probability measure on $T^*\T^2$ that is invariant under the Hamiltonian flow $\varphi_H^t$. Then, the following holds:
\begin{enumerate}
 \item For any $\Lambda$ in $\mathcal{L}_1$, $\mathcal{I}_\Lambda(\mu)$ is a finite nonnegative Radon measure on $T^*\T^2$.
 \item For any $\Lambda$ in $\mathcal{L}_1$, we have $\mu|_{\T^2\times E_{\Lambda^\perp\setminus\{0\}}}=\mathcal{I}_\Lambda(\mu)|_{\T^2\times E_{\Lambda^\perp\setminus\{0\}}}$.
 \item $\widehat{\mu}_0$ is a finite nonnegative Radon measure on $T^*\T^2$ and $\mu|_{\T^2\times \Omega_{\mathrm{nr}}}=\widehat{\mu}_0|_{\T^2\times  \Omega_{\mathrm{nr}}}$.
\end{enumerate}
\end{lemma}
Since $\tau_h \gg 1$, one can apply this lemma to the semiclassical measures $\nu_t^B$ given by \eqref{e:decomposition-de-nu_t^B-sans-point-critique} and decompose them accordingly:
\begin{equation}\label{e:nu_t^B=decomposition-direction-du-flot}
    \nu_t^B(x,\xi) =\widehat{\nu}_{t}^B(0,\xi)|_{\T^2\times \Omega_{\mathrm{nr}}}+\sum_{\Lambda\in\mathcal{L}_1}\mathcal{I}_\Lambda(\nu_t^B)|_{\T^2\times E_{\Lambda^\perp \setminus \{0\}}}.
\end{equation}

Therefore, to derive regularity properties for $\nu_t$ along the $x$ variable, it is enough to control how $\nu_t^B$ behaves in the $x$ variable on each resonant layer, i.e.\ for each $\Lambda\in\mathcal L_1$, to analyze $\mathcal I_\Lambda(\nu_t^B)\big|_{\T^2\times E_{\Lambda^\perp\setminus\{0\}}}$. This will be the purpose of the next section.

\section{A two-microlocal procedure to analyze periodic orbits}\label{s:periodic-orbits}
Throughout the section, we work under the spectral localization assumption \eqref{e:spectral-projection} and assume that the conditions \eqref{e:energy-window} and \eqref{e:SE-diffeo-S1} are satisfied. Let $W^B$ be an accumulation point of the space--time magnetic Wigner distributions $W_h^B(\tau_h)$ introduced in \eqref{e:WB-def}. In this section, we analyze the information carried by $W^B$ (more precisely $\nu^B$) along a fixed rank-one primitive sublattice $\Lambda\in\mathcal L_1$, namely along the periodic directions of the Hamiltonian flow. More precisely, we aim at describing the component $\mathcal{I}_\Lambda(\nu_t^B)|_{\T^2\times E_{\Lambda^\perp \setminus \{0\}}}$ in \eqref{e:nu_t^B=decomposition-direction-du-flot}. We first perform a microlocal splitting between a neighborhood of the resonant set $\{\nabla H(\xi)\in\Lambda^\perp\setminus\{0\}\}$ and its complement, and prove that both pieces define nonnegative Radon measures with absolutely continuous time marginals and the expected propagation or invariance properties (Lemmas~\ref{l:split-W}, \ref{l:RN-split} and \ref{l:dyn-split}). We then study the weakly resonant component and show, through a commutator argument based on the symbol introduced in Lemma~\ref{l:symbol-s-gamma-R}, that all its nonzero $\Lambda$-Fourier modes vanish (Lemma~\ref{l:noncompact-modes}). Next, we turn to the strong resonant component, introduce a two-microlocal lift, and establish its compactness, positivity, projection and time-disintegration properties (Lemma~\ref{l:2micro}), together with the invariance of the lifted measure under the lifted geodesic flow (Lemma~\ref{l:RN-2micro}). Finally, after fixing the zoom scaling in Definition~\ref{d:def-gamma}, we identify the effective dynamics of the lifted measure in the three regimes $\tau_h\ll h^{-1/2}$, $\tau_h= h^{-1/2}$, and $\tau_h\gg h^{-1/2}$ (Lemma~\ref{l:three-regimes-2micro}). These properties will then be used in Sections \ref{s:x-regularity} and \ref{s:long-time-propagation} to describe the measures appearing in the decomposition \eqref{e:nu_t^B=decomposition-direction-du-flot} of the time-dependent Wigner distribution $W^B$.

\subsection{A microlocal splitting around $\Lambda^\perp$}

Fix $\Lambda\in\mathcal L_1$ and let $\mathfrak e_\Lambda$ be a generator of $\Lambda$. Let $\chi\in \mathcal{C}_c^\infty(\R)$ be such that $0\leq\chi\leq1$, $\chi\equiv 1$ on $[-1,1]$ and $\chi\equiv 0$ outside $[-2,2]$. We fix a sequence $R_m\to+\infty$ (we simply write $R\to+\infty$).

Let $(\gamma_h)_{h\to0^+}$ be a positive scale (to be chosen later depending on how fast $\tau_h$ grows) satisfying 
\begin{equation}\label{e:regime-pour-gamma}
    \displaystyle  \gamma_h \xrightarrow[h\to 0^+]{} 0,  \quad \text{and }\quad \displaystyle   \gamma_h \ge h^{2-2\delta} \quad \text{for a fixed  }\delta>0.
\end{equation} 
In the applications considered below, we will in fact choose \(\delta\ge \frac12\). It represents the size of the zoom around $\Lambda^\perp$. Define the frequency cut-off
\begin{equation}\label{e:chi-resonant}
\chi_{R\sqrt{\gamma_h},\Lambda}(\xi)
:=\chi\!\left(\frac{\nabla H(\xi)\cdot \mathfrak e_\Lambda}{R\sqrt{\gamma_h}\,L_\Lambda}\right).
\end{equation}
For $a\in \mathcal{C}_c^\infty(\R\times T^*\T^2)$, we split
\[
\langle W_h^B(\tau_h),a\rangle
=\langle W_h^B(\tau_h),a\,\chi_{R\sqrt{\gamma_h},\Lambda}\rangle
+\langle W_h^B(\tau_h),a\,(1-\chi_{R\sqrt{\gamma_h},\Lambda})\rangle .
\]
This motivates the definition of the following two distributions:
\begin{align}
W_{h,R,\Lambda}^B &: a\mapsto \langle W_h^B(\tau_h),a\,\chi_{R\sqrt{\gamma_h},\Lambda}\rangle,
\label{e:def-W-compact}\\
W_{h,R}^{B,\Lambda} &: a\mapsto \langle W_h^B(\tau_h),a\,(1-\chi_{R\sqrt{\gamma_h},\Lambda})\rangle,
\label{e:def-W-noncompact}
\end{align}
for which the Calderón--Vaillancourt Theorem \ref{t:Calderon-Vaillancourt} ensures that these expressions are well defined as distributions.
Whenever we write $(\cdot)_{h\to0^+,\;R\to+\infty}$, we mean the \emph{iterated limit} $h\to0^+$ followed by $R\to+\infty$.

\begin{lemma}[Boundedness, positivity, and decomposition]\label{l:split-W}
Assume $(u_h)_{h\to0^+}$ solves the magnetic Schr\"odinger equation \eqref{e:Schrodinger-PDE} and admits a semiclassical measure $W^B$. Then the following holds: 
\begin{enumerate}
\item The families $(W_{h,R,\Lambda}^B)_{h\to0^+,\;R\to+\infty}$ and
$(W_{h,R}^{B,\Lambda})_{h\to0^+,\;R\to+\infty}$ are bounded in
$\mathcal D'(\R\times T^*\T^2)$.
\item Any accumulation points $W_\Lambda^B$ and $W^{B,\Lambda}$ of these families are nonnegative Radon measures on $\R\times \T^2 \times \Omega_{E_1,E_2}$. 
\item Recalling that \(W^B\) denotes an accumulation point of the magnetic Wigner distributions \(W_h^B(\tau_h)\) defined in \eqref{e:Wigner-magnetic-distrib}, as in Lemma~\ref{l:Lifting-QL}, one has the measure identity
\begin{equation}\label{e:W-somme-de-deux-mesures}
W^B = W_\Lambda^B + W^{B,\Lambda}.
\end{equation}
\end{enumerate}
\end{lemma}

\begin{proof} 
For items 1 and 2, the techniques involved in the proof are the same for $(W_{h,R,\Lambda}^B)_{h\to0^+,\;R\to+\infty}$ and
$(W_{h,R}^{B,\Lambda})_{h\to0^+,\;R\to+\infty}$. We give each time the proof only for one of them. 
Let $a\in \mathcal{C}_c^\infty(\R\times T^*\T^2)$. Using the rescaling formula \eqref{e:rescaling-semiclassique}, one can
rewrite
\[
\displaystyle \langle W_{h,R,\Lambda}^B,a\rangle
=\int_\R \big\langle \Opgamma \big(a(t,x,\sqrt{\gamma_h}\xi)\,
\chi(\tfrac{\nabla H(\sqrt{\gamma_h}\xi)\cdot\mathfrak e_\Lambda}{R\sqrt{\gamma_h}L_\Lambda})\big)
\,u_h(t \tau_h ),u_h(t \tau_h )\big\rangle\,\dd t.
\]
For $\gamma_h\in(0,1)$ satisfying \eqref{e:regime-pour-gamma} and $R\geq 1$, the symbols
\[
a(t,x,\sqrt{\gamma_h}\xi)\,
\chi\!\left(\frac{\nabla H(\sqrt{\gamma_h}\xi)\cdot\mathfrak e_\Lambda}
{R\sqrt{\gamma_h}L_\Lambda}\right)
\quad\text{and}\quad
a(t,x,\sqrt{\gamma_h}\xi)\,
\left(1-\chi\!\left(\frac{\nabla H(\sqrt{\gamma_h}\xi)\cdot\mathfrak e_\Lambda}
{R\sqrt{\gamma_h}L_\Lambda}\right)\right)
\]
belong to $\mathscr{S}^0$ uniformly in $h$ and $R$. Indeed, on the support of
$a(t,x,\sqrt{\gamma_h}\xi)$, the rescaled variable $\sqrt{\gamma_h}\xi$ remains in a fixed compact set. Hence all derivatives of $H$ are evaluated on a compact set, and the chain rule gives uniform bounds for all derivatives of $\chi_{R\sqrt{\gamma_h},\Lambda}$. Hence, the Calder\'on--Vaillancourt Theorem yields uniform boundedness, proving item 1 of the lemma.

For item 2, we repeat the positivity argument used in the proof of Lemma \ref{l:magnetic-wigner}. The assumption \eqref{e:regime-pour-gamma} allows us to apply the same argument in the rescaled calculus, since the corresponding semiclassical parameter \(h/\sqrt{\gamma_h}\) tends to \(0\). Fix $a\ge 0$ in $\mathcal{C}_c^\infty(\R\times T^*\T^2)$ and let $[t_1,t_2]\subset \R$ be a compact interval containing the support of $a$ in the time variable. For $\varepsilon>0$, define for each $t\in [t_1,t_2]$
\[
a_{\varepsilon,h}(t,x,\xi):=\sqrt{a(t,x,\sqrt{\gamma_h}\xi)\,\chi_{R\sqrt{\gamma_h},\Lambda}(\sqrt{\gamma_h}\xi)+\varepsilon}.
\]
Then $(a_{\varepsilon,h}(t))_{t\in [t_1,t_2]}$ is a bounded family in $\mathscr{S}^0$. Using the symbolic composition \eqref{e:composition-rule} and adjoint formula \eqref{e:formal-adjoint} in the magnetic calculus, for each \(t\in[t_1,t_2]\), one has
\[
\Opgamma(a_{\varepsilon,h}(t))^*\Opgamma(a_{\varepsilon,h}(t))
=
\Opgamma\!\big(a(t,x,\sqrt{\gamma_h}\xi)\,\chi_{R\sqrt{\gamma_h},\Lambda}(\sqrt{\gamma_h}\xi)+\varepsilon\big)
+\mathcal O_{L^2 \to L^2}\!\left(\frac{h}{\sqrt{\gamma_h}}\right),
\]
which implies positivity of the limiting functional as $\varepsilon\to0$. This shows that any accumulation point defines a nonnegative distribution, hence a finite nonnegative Radon measure on $\R\times T^*\T^2$.

Up to extracting a common subsequence, we may assume both components converge, and then the identity $W_h^B(\tau_h) = W_{h,R,\Lambda}^B + W_{h,R}^{B,\Lambda}$ passes to the limit. 

The support property of item 2 follows by Lemma~\ref{l:support-de-W^B} together with the nonnegativity of \(W_\Lambda^B\) and \(W^{B,\Lambda}\) and the identity \eqref{e:W-somme-de-deux-mesures}.
\end{proof}

\begin{remark}
The proof of Lemma~\ref{l:split-W} uses the fact that $(u_h)_{h\to0^+}$ solves the Schrödinger equation \eqref{e:Schrodinger-PDE} only to obtain the support property. The boundedness, positivity, and algebraic decomposition arguments rely only on the definition of the distributions involved and on the uniform boundedness properties of the corresponding pseudodifferential operators.
\end{remark}

Now that the limiting distributions have been identified as measures, we may perform on \(W_\Lambda^B\) and \(W^{B,\Lambda}\) the same time disintegration with respect to their time marginal as in \eqref{e:time-marginal-of-W^B}. Hence, we can write
\begin{equation}\label{e:disint-split}
W_\Lambda^B(\dd t ,\dd x,\dd \xi)=\nu_{t,\Lambda}^B(\dd x,\dd \xi)\otimes\omega_\Lambda^B(\dd t ),
\qquad
W^{B,\Lambda}(\dd t ,\dd x,\dd \xi)=\nu_t^{B,\Lambda}(\dd x,\dd \xi)\otimes \omega^{B,\Lambda}(\dd t ),
\end{equation}
where $\nu_{t,\Lambda}^B$ and $\nu_t^{B,\Lambda}$ are measurable families of probability measures on $\T^2 \times \Omega_{E_1,E_2}$, and $\omega_\Lambda^B$, $\omega^{B,\Lambda}$ are nonnegative Radon measures on $\R$. Since one can decompose $W^B$ as we did in Lemma~\ref{l:Lifting-QL} as $W^B(\dd t ,\dd x,\dd \xi)=\nu_t^B(\dd x,\dd \xi)\otimes \dd t $ with $\nu_t^B$ a probability measure for a.e. $t$, we can identify the time marginals of $\omega_\Lambda^B$ and $\omega^{B,\Lambda}$.

\begin{lemma}[Absolute continuity of the time marginals]\label{l:RN-split}
There exist $h_\Lambda,h^\Lambda\in L^1_{\mathrm{loc}}(\R)$ such that
\[
\omega_\Lambda^B(\dd t )=h_\Lambda(t)\,\dd t ,\qquad \omega^{B,\Lambda}(\dd t )=h^\Lambda(t)\,\dd t ,
\]
and
\begin{equation}\label{e:mass-de-la-somme}
h_\Lambda(t)+h^\Lambda(t)=1
\quad\text{for a.e.\ }t.
\end{equation}
\end{lemma}

\begin{proof}
Applying the Radon--Nikodym decomposition to $\omega_\Lambda^B$ and $\omega^{B,\Lambda}$ with respect to $\dd t $ leads to the existence of $h_\Lambda$ and $h^\Lambda$ both in $ L^1_{\mathrm{loc}}(\R)$ such that 
\begin{equation}\label{e:Radon-Nikodym-mesures-sigma}
    \omega_\Lambda^B(\dd t ) = h_\Lambda(t)\dd t + \omega_{\Lambda,\perp}^B(\dd t ), \qquad \omega^{B,\Lambda}(\dd t )=h^\Lambda(t)\,\dd t + \omega_\perp^{B,\Lambda}(\dd t), 
\end{equation}
where $\omega_\perp^{B,\Lambda}$ and $\omega_{\Lambda,\perp}^B$ are singular with respect to the Lebesgue measure $\dd t$. One can combine the decomposition \eqref{e:disint-split} and \eqref{e:Radon-Nikodym-mesures-sigma} with \eqref{e:W-somme-de-deux-mesures} to get 
\begin{align}\label{e:somme-radon-nikodym}
    \nu_t^B(\dd x ,\dd \xi)\otimes \dd t = &\left[\nu_{t,\Lambda}^B(\dd x,\dd \xi)h_\Lambda(t) + \nu_t^{B,\Lambda}(\dd x,\dd \xi)h^\Lambda(t) \right] \otimes \dd t \\
    &+ \nu_{t,\Lambda}^B(\dd x,\dd \xi)\otimes \omega_{\Lambda,\perp}^B(\dd t ) +\nu_t^{B,\Lambda}(\dd x,\dd \xi) \otimes  \omega_\perp^{B,\Lambda}(\dd t). \nonumber
\end{align}
Since the left-hand side of \eqref{e:somme-radon-nikodym} is absolutely continuous with respect to the Lebesgue measure \(\dd t\), the singular parts must vanish. Hence \eqref{e:somme-radon-nikodym} becomes
\begin{equation}\label{e:some-radon-nikodym-sans-singpart}
    \nu_t^B(\dd x ,\dd \xi)\otimes \dd t = \left[\nu_{t,\Lambda}^B(\dd x,\dd \xi)h_\Lambda(t) + \nu_t^{B,\Lambda}(\dd x,\dd \xi)h^\Lambda(t) \right] \otimes \dd t. 
\end{equation}
Taking total masses in \eqref{e:some-radon-nikodym-sans-singpart} and using that
\(\nu_t^B\), \(\nu_{t,\Lambda}^B\) and \(\nu_t^{B,\Lambda}\) are probability measures for almost every \(t\), we obtain
\eqref{e:mass-de-la-somme}. 
\end{proof}

The dynamics of $\nu_{t,\Lambda}^B$ and $\nu_t^{B,\Lambda}$ are governed by the same time-scale trichotomy as in Lemma~\ref{l:Dynamical-regimes}. We record the statement for later use.

\begin{lemma}[Propagation and invariance for the split components]\label{l:dyn-split}
Let $\varphi_H^s(x,\xi)=(x+s\nabla H(\xi),\xi)$ be the Hamiltonian flow of $H$.
\begin{enumerate}
\item If $\tau_h\to0$, then $\nu_{t,\Lambda}^B$ and $\nu_t^{B,\Lambda}$ are
time-independent.
\item If $\tau_h=1$, then they are transported by $\varphi_H^s$.
\item If $\tau_h\to+\infty$, then for a.e.\ $t$ they are invariant under
$\varphi_H^s$ for every $s\in\R$.
\end{enumerate}
\end{lemma}

\begin{proof}
We only treat the component \(W_\Lambda^B\), the proof for \(W^{B,\Lambda}\) is identical replacing \(\chi\) by \(1-\chi\). Let \(\psi\in \mathcal{C}_c^\infty(\R)\), let \(a\in \mathcal{C}_c^\infty(T^*\T^2)\), and fix \(R>1\). Define
\[
a_{h,R}(x,\xi):=
a(x,\xi)\,
\chi\!\left(
\frac{\nabla H(\xi)\cdot \mathfrak e_\Lambda}{R\sqrt{\gamma_h}L_\Lambda}
\right) \quad \text{and} \qquad \widetilde a_{h,R}(x,\xi) = a_{h,R}(x,\sqrt{\gamma_h}\xi).
\]
By the rescaling identity \eqref{e:rescaling-semiclassique}, one has
\[
\Op(a_{h,R})=\Opgamma(\widetilde a_{h,R}) \quad \text{and} \qquad \widehat H_h=\Opgamma(\widetilde p_h),
\]
where $ \widetilde p_h(x,\xi):= H(\sqrt{\gamma_h}\xi)+hR_h(x,\sqrt{\gamma_h}\xi)$. Since \((u_h)_{h\to0^+}\) solves \eqref{e:Schrodinger-PDE}, integration by parts in time yields
\begin{equation}\label{e:heisenberg-split-microlocalisation}
\frac1{\tau_h}\int_\R \psi'(t)\,
\big\langle
\Opgamma(\widetilde a_{h,R})u_h(t \tau_h),
u_h(t \tau_h)
\big\rangle \dd t
=
\int_\R \psi(t)\,
\left\langle
\frac{i}{h}
\big[
\Opgamma(\widetilde a_{h,R}),\widehat H_h
\big]
u_h(t \tau_h),u_h(t \tau_h)
\right\rangle \dd t.
\end{equation}
By the same argument as in Lemma~\ref{l:split-W}, the family \((\widetilde a_{h,R})_h\) is bounded in \(\mathscr{S}^0(T^*\T^2)\), uniformly in \(h\) for fixed \(R\). In particular, by Calder\'on--Vaillancourt, 
\begin{equation}\label{e:calderon-sur-atilde} \left\| \Opgamma(\widetilde a_{h,R}) \right\|_{L^2\to L^2} \le C_{a,R}. \end{equation} 
Moreover, since \(a\) is compactly supported in the \(\xi\)-variable, there exists \(C_a>0\) such that 
\[ \operatorname{supp}_\xi(\widetilde a_{h,R}) \subset \left\{ \xi\in\R^2:\ |\xi|\leq C_a\gamma_h^{-1/2} \right\}. 
\] 
Consequently, on the support of \(\widetilde a_{h,R}\), one has 
\[ 
\gamma_h^{m/2}\langle\xi\rangle^m\leq C_{a,m}. 
\] 
Since all derivatives of \(\widetilde a_{h,R}\) satisfy the same support condition, the uniform \(\mathscr S^0\)-estimates imply that for fixed $R$, $\gamma_h^{m/2}\widetilde a_{h,R}$ is uniformly bounded in $\mathscr{S}^{-m}$. We now compute the commutator. Applying \eqref{e:a-star-b} at the semiclassical parameter $h/\sqrt{\gamma_h}$,
we obtain
\begin{equation}\label{e:commutator-split-detailed-clean}
\big[
\Opgamma(\gamma_h^{m/2}\widetilde a_{h,R}),
\Opgamma(\widetilde p_h)
\big]
=
\frac{h\gamma_h^{(m-1)/2}}{i}
\Opgamma\!\big(
\{\widetilde a_{h,R},\widetilde p_h\}
\big)
+
\Opgamma(\widetilde r_{h,R}).
\end{equation}
We claim that
\begin{equation}\label{e:rhohR-estimate-final}
\widetilde r_{h,R}
=
\mathcal O_{\mathscr S^0}
\left(
h^2\gamma_h^{(m-1)/2}
\right).
\end{equation}
To prove \eqref{e:rhohR-estimate-final}, we use the precise structure of
the composition formula \eqref{e:a-star-b}. Consider first a term of
order \(k\ge2\) involving \(H(\sqrt{\gamma_h}\xi)\). Since \(H\) is
independent of \(x\), every derivative falling on this factor is a
\(\xi\)-derivative. Therefore, one has 
\[
\partial_\xi^\beta H(\sqrt{\gamma_h}\xi)
=
\gamma_h^{|\beta|/2}
(\partial^\beta H)(\sqrt{\gamma_h}\xi),
\qquad |\beta|=k.
\]
Since the derivatives of
\(\gamma_h^{m/2}\widetilde a_{h,R}\) are
\(\mathcal O_{\mathscr S^0}(\gamma_h^{m/2})\), every such contribution
is bounded by
\[
\gamma_h^{m/2}
\left(\frac{h}{\sqrt{\gamma_h}}\right)^k
\gamma_h^{k/2}
=
\gamma_h^{m/2}h^k
=
\mathcal O_{\mathscr S^0}
\left(
h^2\gamma_h^{(m-1)/2}
\right),
\]
since \(k\ge2\).

For the terms involving \(hR_h(x,\sqrt{\gamma_h}\xi)\), some
derivatives may fall on the \(x\)-variable and therefore produce no
factor \(\sqrt{\gamma_h}\). Nevertheless, these terms always carry the
additional prefactor \(h\). The worst contribution is therefore
\[
\gamma_h^{m/2}h
\left(\frac{h}{\sqrt{\gamma_h}}\right)^2
=
h^3\gamma_h^{(m-2)/2}
=
h^2\gamma_h^{(m-1)/2}
\frac{h}{\sqrt{\gamma_h}}
=
o\!\left(
h^2\gamma_h^{(m-1)/2}
\right),
\]
because \(h/\sqrt{\gamma_h}\to0\).

It remains to control the final remainder in \eqref{e:a-star-b}. Since \(\gamma_h^{m/2}\widetilde a_{h,R}\) and \(\widetilde p_h\) are uniformly bounded in \(\mathscr S^{-m}\) and \(\mathscr S^m\), respectively, the remainder of order \(N\) is bounded in \(\mathscr S^0\) by
\[
C_{N,R}
\left(\frac{h}{\sqrt{\gamma_h}}\right)^N.
\]
The assumption \eqref{e:regime-pour-gamma} implies
\[
\frac{h}{\sqrt{\gamma_h}}\leq h^\delta.
\]
Choosing \(N\) sufficiently large shows that this remainder is also of size
\[
\mathcal{O}_{\mathscr S^0}
\left(
h^2\gamma_h^{(m-1)/2}
\right).
\]
This proves \eqref{e:rhohR-estimate-final}. Hence, dividing \eqref{e:commutator-split-detailed-clean} by \(\gamma_h^{m/2}\) yields
\begin{equation}\label{e:commutator-split-unrenormalized}
\big[
\Opgamma(\widetilde a_{h,R}),
\Opgamma(\widetilde p_h)
\big]
=
\frac{h}{i\sqrt{\gamma_h}}
\Opgamma\!\big(
\{\widetilde a_{h,R},\widetilde p_h\}
\big)
+
\Opgamma(r_{h,R}),
\end{equation}
where
\begin{equation}\label{e:rhR-estimate-final1}
r_{h,R}
:=
\gamma_h^{-m/2}\widetilde r_{h,R}
=
\mathcal O_{\mathscr S^0}
\left(
\frac{h^2}{\sqrt{\gamma_h}}
\right).
\end{equation}
By the Calderón--Vaillancourt
Theorem~\ref{t:Calderon-Vaillancourt}, one obtains
\[
\left\|
\Opgamma(r_{h,R})
\right\|_{L^2\to L^2}
\leq
C_R\frac{h^2}{\sqrt{\gamma_h}}.
\]
Consequently,
\begin{equation}\label{e:commutator-divided}
\frac{i}{h}
\big[
\Opgamma(\widetilde a_{h,R}),
\widehat H_h
\big]
=
\Opgamma\!\left(
\frac{1}{\sqrt{\gamma_h}}
\{\widetilde a_{h,R},\widetilde p_h\}
\right)
+
\mathcal O_{L^2\to L^2}
\left(
\frac{h}{\sqrt{\gamma_h}}
\right).
\end{equation}
We next split the Poisson bracket:
\[
\frac1{\sqrt{\gamma_h}}\{\widetilde a_{h,R},\widetilde p_h\}
=
\frac1{\sqrt{\gamma_h}}\{\widetilde a_{h,R},H(\sqrt{\gamma_h}\xi)\}
+
\frac{h}{\sqrt{\gamma_h}}
\{\widetilde a_{h,R},R_h(x,\sqrt{\gamma_h}\xi)\}.
\]
For the second term, note that \((\widetilde a_{h,R})_h\) is uniformly bounded in
\(\mathscr{S}^0\), while every positive-order \(\xi\)-derivative of
\(R_h(x,\sqrt{\gamma_h}\xi)\) is \(\mathcal O(\sqrt{\gamma_h})\) on the support of
\(\widetilde a_{h,R}\). It follows that
\[
\frac{h}{\sqrt{\gamma_h}}
\{\widetilde a_{h,R},R_h(x,\sqrt{\gamma_h}\xi)\}
=
\mathcal{O}_{\mathscr{S}^0}\!\left(\frac{h}{\sqrt{\gamma_h}}\right).
\]
Hence, thanks to the Calderón-Vaillancourt Theorem \eqref{e:commutator-divided} becomes
\begin{equation}\label{e:commutator-divided-avec-H}
\frac{i}{h}
\big[
\Opgamma\!\big(\widetilde a_{h,R}\big),\widehat H_h
\big]
=
\Opgamma\!\left(
\frac1{\sqrt{\gamma_h}}
\big\{\widetilde a_{h,R},H(\sqrt{\gamma_h}\xi )\big\}
\right)
+
\mathcal O_{L^2\to L^2}\!\left(\frac{h}{\sqrt{\gamma_h}}\right).
\end{equation}
Computing explicitly the Poisson bracket, one gets
\begin{equation}\label{e:poisson-bracket-atilde-H}
    \frac1{\sqrt{\gamma_h}}\{\widetilde a_{h,R},H(\sqrt{\gamma_h}\xi)\}
=
-\nabla H(\sqrt{\gamma_h}\xi)\cdot \partial_x a(x,\sqrt{\gamma_h}\xi)\,
\chi\!\left(
\frac{\nabla H(\sqrt{\gamma_h}\xi)\cdot \mathfrak e_\Lambda}
{R\sqrt{\gamma_h}L_\Lambda}
\right).
\end{equation}
Finally, inserting \eqref{e:poisson-bracket-atilde-H} in \eqref{e:heisenberg-split-microlocalisation} leads to 
\begin{align}\label{e:3dynamical-regimes}
\frac1{\tau_h}&\int_\R \psi'(t)\,
\big\langle
\Opgamma(\widetilde a_{h,R})u_h(t \tau_h),
u_h(t \tau_h)
\big\rangle \dd t
\\
&=\nonumber
-\int_\R \psi(t)\,
\left\langle
\Opgamma\left(\nabla H(\sqrt{\gamma_h}\xi)\cdot \partial_x a(x,\sqrt{\gamma_h}\xi)\,
\chi\!\left(
\frac{\nabla H(\sqrt{\gamma_h}\xi)\cdot \mathfrak e_\Lambda}
{R\sqrt{\gamma_h}L_\Lambda}
\right)\right)
u_h(t \tau_h),u_h(t \tau_h)
\right\rangle \dd t\\
&+\nonumber
\mathcal O\!\left(\frac{h}{\sqrt{\gamma_h}}\right).
\end{align}
From \eqref{e:regime-pour-gamma}, one can infer that $h/\sqrt{\gamma_h} \to 0$ as $h \to 0^+$, and from \eqref{e:calderon-sur-atilde} the left-hand side of \eqref{e:3dynamical-regimes} passes to the limit. We now distinguish the three regimes.

If \(\tau_h\to0\), multiplying \eqref{e:3dynamical-regimes} by \(\tau_h\) and passing to the limit \(h\to0^+\), then \(R\to+\infty\), yields
\[
\int_\R \psi'(t)\,\langle a,\nu_{t,\Lambda}^B\rangle\,h_\Lambda(t)\,\dd t=0.
\]
Hence \(\nu_{t,\Lambda}^B\) is independent of \(t\) for \(h_\Lambda(t)\,\dd t \)-almost every \(t\).

If \(\tau_h=1\), passing to the limit in \eqref{e:3dynamical-regimes} and using Theorem \ref{t:Calderon-Vaillancourt}, together with \(h/\sqrt{\gamma_h}\to0\), we obtain
\[
\int_\R \psi'(t)\,\langle a,\nu_{t,\Lambda}^B\rangle\,h_\Lambda(t)\,\dd t
=
-\int_\R \psi(t)\,
\langle \nabla H(\xi)\cdot \partial_x a,\nu_{t,\Lambda}^B\rangle\,
h_\Lambda(t)\,\dd t.
\]
This is the weak form of the transport equation
\[
\partial_t \nu_{t,\Lambda}^B+\nabla H(\xi)\cdot \partial_x \nu_{t,\Lambda}^B=0.
\]
Thus \(\nu_{t,\Lambda}^B\) is transported by \(\varphi_H^s\), in the weak sense, for \(h_\Lambda(t)\,\dd t \)-almost every \(t\).

If \(\tau_h\to+\infty\), the left-hand side of \eqref{e:3dynamical-regimes} tends to \(0\), and the same limit procedure yields
\[
\int_\R \psi(t)\,
\langle \nabla H(\xi)\cdot \partial_x a,\nu_{t,\Lambda}^B\rangle\,
h_\Lambda(t)\,\dd t=0.
\]
Therefore, for \(h_\Lambda(t)\,\dd t\)-almost every \(t\),
\[
\langle \nabla H(\xi)\cdot \partial_x a,\nu_{t,\Lambda}^B\rangle=0,
\]
which is exactly the infinitesimal formulation of the invariance of
\(\nu_{t,\Lambda}^B\) under \(\varphi_H^s\).
\end{proof}

We shall use this splitting together with the standard Fourier structure of invariant measures under the completely integrable flow generated by \(H\). Indeed, in the long-time regime \(\tau_h\to+\infty\), Lemma~\ref{l:dyn-split} implies that both split components are invariant under \(\varphi_H^s\), for almost every time with respect to their time marginals. Applying Lemma~\ref{l:anantharamanmacia} to each component gives the following consequence.
\begin{corollary}\label{l:split-invariant-fourier-structure}
Assume that \(\tau_h\to+\infty\). Then, for every \(\Lambda\in\mathcal L_1\), the measures $\mathcal I_\Lambda(\nu_{t,\Lambda}^B)$ and $\mathcal I_\Lambda(\nu_t^{B,\Lambda})$ are finite nonnegative Radon measures such that 
\begin{equation}\label{e:lemme-anantharaman-applique-aux-microlocalisations}
\nu_{t,\Lambda}^B\big|_{\T^2\times E_{\Lambda^\perp\setminus\{0\}}}
=
\mathcal I_\Lambda(\nu_{t,\Lambda}^B)\big|_{\T^2\times E_{\Lambda^\perp\setminus\{0\}}}, \qquad
\nu_t^{B,\Lambda}\big|_{\T^2\times E_{\Lambda^\perp\setminus\{0\}}}
=
\mathcal I_\Lambda(\nu_t^{B,\Lambda})\big|_{\T^2\times E_{\Lambda^\perp\setminus\{0\}}},
\end{equation}
and
\begin{equation}\label{e:lemme-anantharaman-applique-aux-microlocalisations-coeff0}
\nu_{t,\Lambda}^B(x,\xi)|_{\T^2\times \Omega_{\mathrm{nr}}}
=
\widehat\nu_{t,\Lambda}^B(0,\xi)|_{\T^2\times \Omega_{\mathrm{nr}}},
\qquad
\nu_t^{B,\Lambda}(x,\xi)|_{\T^2\times \Omega_{\mathrm{nr}}}
=
\widehat\nu_t^{B,\Lambda}(0,\xi)|_{\T^2\times \Omega_{\mathrm{nr}}}.
\end{equation}
\end{corollary}
\begin{proof}
By Lemma~\ref{l:dyn-split}, for \(\omega_\Lambda^B\)-almost every \(t\), the measure \(\nu_{t,\Lambda}^B\) is invariant under \(\varphi_H^s\), and for \(\omega^{B,\Lambda}\)-almost every \(t\), the measure \(\nu_t^{B,\Lambda}\) is invariant under \(\varphi_H^s\). Since both measures are probability measures, Lemma~\ref{l:anantharamanmacia} can be applied to each of them.
\end{proof}

\subsection{The weakly-resonant part: vanishing of $\Lambda$-modes}
From now on, we assume that $ \tau_h\longrightarrow+\infty$, so that the decomposition of Corollary \ref{l:split-invariant-fourier-structure} holds. We now turn to the weakly resonant component $\nu_t^{B,\Lambda}\big|_{\T^2\times E_{\Lambda^\perp\setminus\{0\}}}$ and prove that its nonzero $\Lambda$-Fourier modes vanish. Let \(k\in\Lambda\setminus\{0\}\) and \(b\in \mathcal C_c^\infty(\R^2)\).
The idea is to test the \(k\)-th Fourier coefficient away from the resonant set
associated with \(\Lambda\), namely the set where
\[
\nabla H(\xi)\cdot \mathfrak e_\Lambda=0.
\]
For this purpose, we introduce a cutoff which removes a neighborhood of this
set at scale \(R\sqrt{\gamma_h}\). We define, for \(\gamma_h>0\) and \(R>0\),
\begin{equation}\label{e:definition-s-gammah-R}
s_{\gamma_h,R}(x,\xi)
=
e^{2i\pi k \cdot x}\,
\frac{b(\sqrt{\gamma_h}\xi)}
{\nabla H(\sqrt{\gamma_h}\xi)\cdot\mathfrak e_\Lambda}
\left(
1-\chi\left(
\frac{\nabla H(\sqrt{\gamma_h}\xi)\cdot\mathfrak e_\Lambda}
{R L_\Lambda\sqrt{\gamma_h}}
\right)
\right).
\end{equation}

The role of this symbol is to make the desired Fourier mode appear as the principal symbol in the commutator with the Hamiltonian part. More precisely, when taking the commutator between the symbol \(s_{\gamma_h,R}\) and the Hamiltonian \(H(\sqrt{\gamma_h}\xi)\), the first-order term in the symbolic expansion
\eqref{e:a-star-b} is given by the Poisson bracket
\[
\left\{
s_{\gamma_h,R},
H(\sqrt{\gamma_h}\xi)
\right\}.
\]
By construction of \(s_{\gamma_h,R}\), this Poisson bracket is, up to an explicit nonzero constant and a power of \(\gamma_h\), precisely the truncated Fourier mode. More precisely,
\begin{equation}\label{e:poisson-bracket-s-et-H-intro-sec}
\left\{s_{\gamma_h,R},\,H(\sqrt{\gamma_h}\xi)\right\}
=
-2i\pi \,\frac{k \cdot \frak{e}_\Lambda}{L_\Lambda^2}\,
\sqrt{\gamma_h}\,
e^{2i\pi k\cdot x}\,
b(\sqrt{\gamma_h}\xi)
\left(
1-\chi\!\left(
\frac{\nabla H(\sqrt{\gamma_h}\xi)\cdot \frak{e}_\Lambda}{R L_\Lambda\sqrt{\gamma_h}}
\right)
\right).
\end{equation}
Thus, after the appropriate normalization, the commutator recovers the \(k\)-th Fourier coefficient that we want to prove vanishes.

The next lemma gives the symbol estimates needed to use
\(s_{\gamma_h,R}\) in the commutator argument. In particular,
\(\sqrt{\gamma_h}s_{\gamma_h,R}\) is an admissible test symbol in the
\(h/\sqrt{\gamma_h}\)-semiclassical calculus, uniformly with respect to \(h\)
for fixed \(R\).

\begin{lemma}\label{l:symbol-s-gamma-R}
Let $k\in \Lambda\setminus\{0\}$ and $b\in \mathcal{C}_c^\infty(\R^2)$. For $(\gamma_h)_{h \to 0^+}$ satisfying \eqref{e:regime-pour-gamma} and $R>0$, the following holds: 
\begin{enumerate}
\item $\sqrt{\gamma_h}\,s_{\gamma_h,R}\in \mathscr{S}^0(\T^2\times\R^2)$, uniformly in $h$.
\item For every $m\in\N^*$, $\gamma_h^{m/2}\,s_{\gamma_h,R}\in \mathscr{S}^{1-m}(\T^2\times\R^2)$, uniformly in $h$.
\end{enumerate}
Moreover, for \(R\ge1\), the following estimates hold uniformly with respect to
\(h\):
\begin{equation}\label{e:estimation-s-gamma-h-en-un-sur-R}
    \sqrt{\gamma_h}\,s_{\gamma_h,R} =\mathcal O_{\mathscr S^0}\!\left(\frac1R\right).
\end{equation}
For every \(j\in\{1,2\}\) and every \(m\in\N^*\), one also has
\begin{equation}\label{e:estimation-partialx-de-sgammah}
\partial_{x_j}\!\left(\gamma_h^{m/2}s_{\gamma_h,R}\right)
=
\mathcal O_{\mathscr S^0}\!\left(\frac{\gamma_h^{(m-1)/2}}{R}\right)=
\mathcal O_{\mathscr S^{1-m}}\!\left(\frac1R\right),
\end{equation}
and
\begin{equation}\label{e:estimation-partialxi-de-sgammah}
\partial_{\xi_j}\!\left(\gamma_h^{m/2}s_{\gamma_h,R}\right)
=
\mathcal O_{\mathscr S^0}\!\left(\frac{\gamma_h^{m/2}}{R}\right)
+
\mathcal O_{\mathscr S^0}\!\left(\frac{\gamma_h^{(m-1)/2}}{R^2}\right)
=\mathcal O_{\mathscr S^{1-m}}\!\left(\frac1R\right).
\end{equation}
\end{lemma}

\begin{proof}
Fix $R>0$. Write
\[
\sqrt{\gamma_h}\,s_{\gamma_h,R}(x,\xi)=e^{2i\pi k \cdot x }\,b(\sqrt{\gamma_h}\xi)\,
F_R\!\left(G_{\gamma_h}(\xi)\right),
\qquad
G_{\gamma_h}(\xi):=\frac{\nabla H(\sqrt{\gamma_h}\xi)\cdot\mathfrak e_\Lambda}{\sqrt{\gamma_h}},
\]
with
\[
F_R(\eta):=\frac{1-\chi(\eta/(RL_\Lambda))}{\eta}.
\]
Since $1-\chi$ vanishes near $0$ and is identically equal to $1$ near infinity, $F_R\in \mathcal{C}^\infty(\R)$ and one has for every $N\in\N$,
\begin{equation}\label{eq:FR-bdd}
\sup_{\eta\in\R}|F_R^{(N)}(\eta)|\le \frac{C_N}{R^{N+1}}.
\end{equation}
Moreover, for any multi-index $\alpha\neq 0$,
\begin{equation}\label{eq:G-deriv-short}
\partial_\xi^\alpha G_{\gamma_h}(\xi)
=\gamma_h^{(|\alpha|-1)/2}\,
\bigl(\partial^\alpha(\nabla H\cdot\mathfrak e_\Lambda)\bigr)(\sqrt{\gamma_h}\xi),
\end{equation}
hence $|\partial_\xi^\alpha G_{\gamma_h}(\xi)|\le C_\alpha\,\gamma_h^{(|\alpha|-1)/2}$ on $\operatorname{supp}(b(\sqrt{\gamma_h}\cdot))$ as $b$ is compactly supported.
Using Leibniz’ rule and repeated chain rule, \eqref{eq:FR-bdd}--\eqref{eq:G-deriv-short} yield, for all $\alpha,\beta$,
\[
\sup_{x,\xi}\bigl|\partial_x^\alpha\partial_\xi^\beta\bigl(\sqrt{\gamma_h}s_{\gamma_h,R}\bigr)(x,\xi)\bigr|
\le \frac{C_{\alpha,\beta}}{R},
\]
which proves both item 1 and \eqref{e:estimation-s-gamma-h-en-un-sur-R}. Item 2 follows from item 1. Indeed, for $m\geq 1$, one has
\[
\gamma_h^{m/2}s_{\gamma_h,R}
=
\gamma_h^{(m-1)/2}\,\sqrt{\gamma_h}\,s_{\gamma_h,R}.
\]
By item 1, the family $\sqrt{\gamma_h}\,s_{\gamma_h,R}$ is bounded in $\mathscr{S}^0$, uniformly in $h$ (for each fixed $R>0$). Moreover, since $b$ is compactly supported, there exists $M>0$ such that
\[
|\xi|\leq M\gamma_h^{-1/2}
\qquad\text{on }\operatorname{supp}(s_{\gamma_h,R}).
\]
It follows that, on this support,
\begin{equation}\label{e:gammah-bornee-par-xi}
\gamma_h^{(m-1)/2}\leq C_m\langle\xi\rangle^{1-m}.
\end{equation}
Therefore, for all multi-indices $\alpha,\beta$,
\[
\bigl|\partial_x^\alpha\partial_\xi^\beta(\gamma_h^{m/2}s_{\gamma_h,R})(x,\xi)\bigr|
\leq C_{\alpha,\beta,R}\,\gamma_h^{(m-1)/2}
\leq C_{\alpha,\beta,m,R}\,\langle\xi\rangle^{1-m},
\]
which proves that $\gamma_h^{m/2}s_{\gamma_h,R}\in \mathscr{S}^{1-m}$ uniformly in $h$. It remains to prove the refined estimates \eqref{e:estimation-partialx-de-sgammah} and \eqref{e:estimation-partialxi-de-sgammah}. For the $x$-derivatives, the only dependence on $x$ is the factor $e^{2i\pi k \cdot x }$. Hence, for $j \in \{1,2\}$ one has 
\[
\partial_{x_j}\left( \gamma_h^\frac{m}{2}s_{\gamma_h,R}\right)= 2i\pi k_j \gamma_h^\frac{m-1}{2} \left(\sqrt{\gamma_h}s_{\gamma_h,R} \right). 
\]
By \eqref{e:estimation-s-gamma-h-en-un-sur-R}, this gives
\[
\partial_{x_j}\left(\gamma_h^{m/2}s_{\gamma_h,R}\right)
=
\mathcal O_{\mathscr S^0}\!\left(
\frac{\gamma_h^{(m-1)/2}}{R}
\right).
\]
Using again the fact that \eqref{e:gammah-bornee-par-xi} holds on $\operatorname{supp}(s_{\gamma_h,R})$, we obtain
\[
\partial_{x_j}\left(\gamma_h^{m/2}s_{\gamma_h,R}\right)
=
\mathcal O_{\mathscr S^0}\!\left(
\frac{\gamma_h^{(m-1)/2}}{R}
\right)
=
\mathcal O_{\mathscr S^{1-m}}\!\left(\frac1R\right).
\]
Then, \eqref{e:estimation-partialx-de-sgammah} follows from \eqref{e:estimation-s-gamma-h-en-un-sur-R}. We finally estimate the \(\xi\)-derivatives. Differentiating
\[
\sqrt{\gamma_h}s_{\gamma_h,R}
=
e^{2i\pi k\cdot x}
b(\sqrt{\gamma_h}\xi)
F_R(G_{\gamma_h}(\xi))
\]
with respect to \(\xi_j\), we get
\begin{equation}\label{e:derive-xi-racinegammah-s-gammah}
\partial_{\xi_j}
\left(
\sqrt{\gamma_h}s_{\gamma_h,R}
\right)
=
e^{2i\pi k\cdot x}
\sqrt{\gamma_h}
(\partial_{\xi_j}b)(\sqrt{\gamma_h}\xi)
F_R(G_{\gamma_h}(\xi))
+
e^{2i\pi k\cdot x}
b(\sqrt{\gamma_h}\xi)
F_R'(G_{\gamma_h}(\xi))
\partial_{\xi_j}G_{\gamma_h}(\xi).
\end{equation}
The first term in \eqref{e:derive-xi-racinegammah-s-gammah} is 
\[
\mathcal O_{\mathscr S^0}\!\left(\frac{\sqrt{\gamma_h}}{R}\right).
\]
Indeed, derivatives falling on \(b(\sqrt{\gamma_h}\xi)\) only produce
nonnegative powers of \(\sqrt{\gamma_h}\), while \eqref{eq:FR-bdd},
\eqref{eq:G-deriv-short}, and the repeated chain rule give
\[
F_R(G_{\gamma_h})
=
\mathcal O_{\mathscr S^0}\!\left(\frac1R\right).
\]
Similarly, using again \eqref{eq:FR-bdd}, \eqref{eq:G-deriv-short}, Leibniz'
rule, and the repeated chain rule, one obtains
\[
F_R'(G_{\gamma_h})\,\partial_{\xi_j}G_{\gamma_h}
=
\mathcal O_{\mathscr S^0}\!\left(\frac1{R^2}\right).
\]
Therefore
\[
\partial_{\xi_j}
\left(
\sqrt{\gamma_h}s_{\gamma_h,R}
\right)
=
\mathcal O_{\mathscr S^0}\!\left(\frac{\sqrt{\gamma_h}}{R}\right)
+
\mathcal O_{\mathscr S^0}\!\left(\frac1{R^2}\right)
\]
which by multiplying by \(\gamma_h^{(m-1)/2}\), leads to
\[
\partial_{\xi_j}
\left(
\gamma_h^{m/2}s_{\gamma_h,R}
\right)
=
\mathcal O_{\mathscr S^0}\!\left(\frac{\gamma_h^{m/2}}{R}\right)
+
\mathcal O_{\mathscr S^0}\!\left(\frac{\gamma_h^{(m-1)/2}}{R^2}\right).
\]
Since \(0<\gamma_h\leq1\) for \(h\) small enough and \(R\geq1\), we have
\begin{equation}\label{e:estime-somme-gammah}
    \frac{\gamma_h^{m/2}}{R}\leq \frac{\gamma_h^{(m-1)/2}}{R},
\qquad
\frac{\gamma_h^{(m-1)/2}}{R^2} \leq \frac{\gamma_h^{(m-1)/2}}{R}.
\end{equation}
Consequently, combining \eqref{e:estime-somme-gammah} with \eqref{e:gammah-bornee-par-xi}, we finally obtain
\[
\partial_{\xi_j} \left( \gamma_h^{m/2}s_{\gamma_h,R} \right) = \mathcal O_{\mathscr S^0}\!\left( \frac{\gamma_h^{(m-1)/2}}{R} \right) = \mathcal O_{\mathscr S^{1-m}}\!\left(\frac1R\right),
\]
which proves \eqref{e:estimation-partialxi-de-sgammah}.

\end{proof}
So far, the scaling parameter $(\gamma_h)_{h\to0^+}$ was only assumed to satisfy \eqref{e:regime-pour-gamma}. We now make a definite choice for $(\gamma_h)_{h\to0^+}$, which will be used throughout the remainder of this article.

\begin{definition}\label{d:def-gamma}
The scaling parameter $(\gamma_h)_{h\to0^+}$ is defined by
\begin{equation}\label{e:definition-gammah}
\gamma_h=
\begin{cases}
\dfrac{1}{\tau_h^2}, & \text{if } \tau_h \ll \frac{1}{\sqrt{h}},\\[0.3em]
h, & \text{if } \tau_h \gtrsim \frac{1}{\sqrt{h}}.
\end{cases}
\end{equation}
\end{definition}

\begin{remark}
This choice may seem somewhat ad hoc at first. However, the proof of the next lemma shows that this dichotomy is naturally dictated by the asymptotic behavior of the time scale $(\tau_h)_{h\to0^+}$, and the corresponding computations justify the definition above.
\end{remark}

\begin{lemma}[Vanishing of $\Lambda$-Fourier modes]\label{l:noncompact-modes}
Let $k\in\Lambda\setminus\{0\}$, $(\gamma_h)_{h\to 0^+}$ satisfying \eqref{e:definition-gammah} and let $(\psi,b)\in \mathcal{C}_c^\infty(\R)\times \mathcal{C}_c^\infty(\R^2)$. Then, one has 
\begin{equation}
    \int_\R \psi(t)\left[\int_{T^*\T^2} e^{2i\pi k\cdot x}\,b(\xi)\,\nu_t^{B,\Lambda}(\dd x,\dd \xi)\right]
h^\Lambda(t)\,\dd t =0.
\end{equation}
In particular, $\mathcal I_\Lambda(\nu_t^{B,\Lambda})$ is independent of $x$ (for a.e.\ $t$).
\end{lemma}
\begin{proof}
The goal is to prove that the following limit vanishes: 
\begin{equation}\label{e:symbol-voulu}
    \displaystyle \lim_{R\to \infty} \lim_{h \to 0^+} \int_\R \psi(t) \left\langle \Op\left( e^{2i\pi k\cdot x} b(\xi) \left(1-\chi\left(\frac{\nabla H(\xi)\cdot \frak{e}_\Lambda}{RL_\Lambda \sqrt{\gamma_h}}\right)\right)\right) u_h(t\tau_h),u_h(t\tau_h)\right\rangle \dd t. 
\end{equation}

To do so, we have already seen that $\sqrt{\gamma_h}s_{\gamma_h,R}\in \mathscr{S}^0$ and, more generally, $\gamma_h^{(m+1)/2}s_{\gamma_h,R}\in \mathscr{S}^{-m}$ by Lemma \ref{l:symbol-s-gamma-R}. This symbol was introduced to analyze the following commutator expression:
\begin{equation}\label{e:non-compact-commutator}
    \displaystyle  \int_\R \psi(t) \left\langle \left[ \Opgamma\left( \gamma_h^{\frac{m+1}{2}}s_{\gamma_h,R}\right) , \Op(H+h R_h) \right] u_h(t\tau_h),u_h(t\tau_h)\right\rangle \dd t. 
\end{equation}
Using that $(u_h)_{h\to 0^+}$ solves \eqref{e:Schrodinger-PDE}, we obtain that \eqref{e:non-compact-commutator} is equal to 
\begin{equation}\label{e:calcul-crochet-de-lie-et-uh-solution}
\displaystyle -\frac{ih \gamma_h^\frac{m}{2}}{\tau_h} \int_\R \psi'(t) \left\langle \Opgamma\left(\sqrt{\gamma_h}s_{\gamma_h,R}\right)u_h(t\tau_h) , u_h(t\tau_h) \right\rangle \dd t .
\end{equation}
By Lemma~\ref{l:symbol-s-gamma-R} and the Calder\'on--Vaillancourt theorem, the quantity in \eqref{e:calcul-crochet-de-lie-et-uh-solution} is $ \mathcal O\!\left(\frac{h\gamma_h^{m/2}}{\tau_h R}\right)$.
As we shall see below, the leading Hamiltonian contribution to the commutator is of order \(h\gamma_h^{(m+1)/2}\) and produces precisely the quantity appearing in \eqref{e:symbol-voulu}. Therefore, after normalization by \(h\gamma_h^{(m+1)/2}\), the contribution of
\eqref{e:calcul-crochet-de-lie-et-uh-solution} is
\[
\mathcal O\!\left(\frac{1}{R\tau_h\sqrt{\gamma_h}}\right) = \mathcal O\!\left(\frac1R\right),
\]
where the last estimate follows from \eqref{e:definition-gammah}, since \(\tau_h\sqrt{\gamma_h}\geq1\). Thus, this term vanishes as \(R\to+\infty\). On the other hand, using the rescaling formula \eqref{e:rescaling-semiclassique} in \eqref{e:non-compact-commutator} to write $\widehat H_h$ in the $\frac{h}{\sqrt{\gamma_h}}$-quantization, we are left with the following commutator, which we split into its Hamiltonian and perturbative parts:
\begin{equation}\label{e:split-commutateur}
\left[ \Opgamma\left( \gamma_h^{\frac{m+1}{2}}s_{\gamma_h,R}\right) ,\Opgamma(H(\sqrt{\gamma_h}\xi)+h R_h(x,\sqrt{\gamma_h}\xi)) \right] = \mathcal C_h^{\rm Ham}+h\mathcal C_h^{\rm pert}.
\end{equation}

Since $\gamma_h^{\frac{m+1}{2}}s_{\gamma_h,R}$ is in the class $\mathscr{S}^{-m}$ (uniformly in $h$) and since the Hamiltonian $H$ belongs to $\mathscr{S}^m$ (and the perturbation $R_h \in \mathscr{S}^{m}$ with $m>1$), all symbols arising in the commutator expansion are uniformly bounded in $\mathscr{S}^0$.

\textbf{1. Hamiltonian contribution.} We first analyze the Hamiltonian contribution using the formula \eqref{e:commutator} to see that the principal symbol arising from the Poisson bracket is exactly the one in \eqref{e:symbol-voulu}.
We now compute the commutator
\begin{equation}\label{e:commutateur-s-et-H}
\mathcal C_h^{\rm Ham} = \left[\Opgamma\!\left(\gamma_h^{\frac{m+1}{2}}s_{\gamma_h,R}\right),
\Opgamma\!\left(H(\sqrt{\gamma_h}\xi)\right)\right].
\end{equation}
Applying the composition formula \eqref{e:a-star-b} in the rescaled \(h/\sqrt{\gamma_h}\) quantization, we obtain for every integer
\(N\ge 4\)
\begin{equation}\label{e:full-odd-expansion-final}
\left[\Opgamma\!\left(\gamma_h^{\frac{m+1}{2}}s_{\gamma_h,R}\right),
\Opgamma\!\left(H(\sqrt{\gamma_h}\xi)\right)\right] = \Opgamma\!\Big( \mathcal{H}_{h,N}\Big) + \Opgamma(\rho_{N,h,R}),
\end{equation}
with
\begin{equation}\label{e:definition_H_N}
\mathcal{H}_{h,N} =\sum_{p=0}^{\lfloor (N-2)/2\rfloor}
\frac{2}{(2p+1)!}
\left(\frac{ih}{2\sqrt{\gamma_h}}\right)^{2p+1}
\mathrm{L}_{h}^{2p+1}
\Big[
\gamma_h^{\frac{m+1}{2}}s_{\gamma_h,R}(x,\xi)\,
H(\sqrt{\gamma_h}\xi')
\Big]
\Big|_{(x,\xi)=(x',\xi')},
\end{equation}
where
\[
\mathrm{L}_{h}
=
\Omega_{\frac{h}{\sqrt{\gamma_h}}B}
(D_x,D_\xi,D_{x'},D_{\xi'})
\]
and, for each fixed \(R>0\),
\begin{equation}\label{e:reste-partie-Hamiltonienne}
\rho_{N,h,R}
=
\mathcal{O}_{\mathscr{S}^0,R}\!\left(
\left(\frac{h}{\sqrt{\gamma_h}}\right)^N
\right)
=
\mathcal{O}_{\mathscr{S}^0,R}\!\left(h^{N/2}\right),
\end{equation}
because Definition~\ref{d:def-gamma} implies \(\gamma_h\ge h\). We now isolate the contribution \(p=0\) from the remaining odd terms. Making the term $p=0$ explicit gives
\begin{align}
&\left[\Opgamma\!\left(\gamma_h^{\frac{m+1}{2}}s_{\gamma_h,R}\right),
\Opgamma\!\left(H(\sqrt{\gamma_h}\xi)\right)\right]
\notag\\
&=
\frac{h}{i\sqrt{\gamma_h}}
\Opgamma\!\left(
\left\{\gamma_h^{\frac{m+1}{2}}s_{\gamma_h,R},\,H(\sqrt{\gamma_h}\xi)\right\}
\right)
\label{e:ordre1-commutateur-s-et-H}\\
&+\frac{\widehat B_0 h^2}{i\gamma_h}
\Opgamma\!\left(
\partial_{\xi_2}\!\left(\gamma_h^{\frac{m+1}{2}}s_{\gamma_h,R}\right)
\partial_{\xi_1}\bigl(H(\sqrt{\gamma_h}\xi)\bigr)
-\partial_{\xi_1}\!\left(\gamma_h^{\frac{m+1}{2}}s_{\gamma_h,R}\right)
\partial_{\xi_2}\bigl(H(\sqrt{\gamma_h}\xi)\bigr)
\right)\label{e:terme-reste-commutateur-s-et-H}\\
&+ \Opgamma(\widetilde{\mathcal{H}}_{N,h,R}),
\label{e:reste-commutateur-s-et-H}
\end{align}
where \(\widetilde{\mathcal{H}}_{N,h,R}\) collects the terms with \(p\ge 1\) in \eqref{e:definition_H_N} together with the symbolic remainder \(\rho_{N,h,R}\). Using \eqref{e:poisson-bracket-s-et-H-intro-sec}, and since
\(\gamma_h^{\frac{m+1}{2}}\) is independent of \((x,\xi)\), one directly obtains
\begin{equation}\label{e:poisson-bracket-s-et-H}
\left\{
\gamma_h^{\frac{m+1}{2}}s_{\gamma_h,R},
H(\sqrt{\gamma_h}\xi)
\right\}
=
-2i\pi \,\frac{k \cdot \frak{e}_\Lambda}{L_\Lambda^2}\,
\gamma_h^{\frac{m+2}{2}}
e^{2i\pi k\cdot x}\,
b(\sqrt{\gamma_h}\xi)
\left(
1-\chi\!\left(
\frac{\nabla H(\sqrt{\gamma_h}\xi)\cdot \frak{e}_\Lambda}
{R L_\Lambda\sqrt{\gamma_h}}
\right)
\right).
\end{equation}
Up to the factor $\gamma_h^{\frac{m+2}{2}}$, this is exactly the quantity in which we are interested in to compute the $k$-Fourier coefficient of our measure in \eqref{e:symbol-voulu}. 

We next turn to the magnetic contribution of order \(h^2/\gamma_h\) coming from the symbol in \eqref{e:terme-reste-commutateur-s-et-H}. By the estimate \eqref{e:estimation-partialxi-de-sgammah} in
Lemma~\ref{l:symbol-s-gamma-R}, for \(j=1,2\), one has
\begin{equation}\label{e:estimation-derivée-gamma-s}
\partial_{\xi_j}\!\left(\gamma_h^{\frac{m+1}{2}}s_{\gamma_h,R}\right)
=
\mathcal{O}_{\mathscr{S}^0}\!\left(\frac{\gamma_h^{\frac{m+1}{2}}}{R}\right)
+
\mathcal{O}_{\mathscr{S}^0}\!\left(\frac{\gamma_h^{m/2}}{R^2}\right).
\end{equation}
Inserting \eqref{e:estimation-derivée-gamma-s} into
\eqref{e:terme-reste-commutateur-s-et-H}, and using that
\(\partial_{\xi_j}(H(\sqrt{\gamma_h}\xi))=\mathcal{O}_{\mathscr{S}^0}(\sqrt{\gamma_h})\) on the support of
\(b(\sqrt{\gamma_h}\xi)\), yields
\begin{equation}\label{e:estimation-terme-reste-commutateur-s-et-H}
\partial_{\xi_2}\!\left(\gamma_h^{\frac{m+1}{2}}s_{\gamma_h,R}\right)
\partial_{\xi_1}\bigl(H(\sqrt{\gamma_h}\xi)\bigr)
-
\partial_{\xi_1}\!\left(\gamma_h^{\frac{m+1}{2}}s_{\gamma_h,R}\right)
\partial_{\xi_2}\bigl(H(\sqrt{\gamma_h}\xi)\bigr)
=
\mathcal{O}_{\mathscr{S}^0}\!\left(\frac{\gamma_h^{\frac{m+2}{2}}}{R}\right)
+
\mathcal{O}_{\mathscr{S}^0}\!\left(\frac{\gamma_h^{\frac{m+1}{2}}}{R^2}\!\right)\!.
\end{equation}
Using \eqref{e:estimation-terme-reste-commutateur-s-et-H} and the Calder\'on--Vaillancourt theorem, it follows that
\begin{align}\label{e:estimation-terme-magnetique-hamiltonien}
   \frac{\widehat B_0 h^2}{i\gamma_h}&\Opgamma\!\left(
\partial_{\xi_2}\!\left(\gamma_h^{\frac{m+1}{2}}s_{\gamma_h,R}\right)
\partial_{\xi_1}\bigl(H(\sqrt{\gamma_h}\xi)\bigr)
-\partial_{\xi_1}\!\left(\gamma_h^{\frac{m+1}{2}}s_{\gamma_h,R}\right)
\partial_{\xi_2}\bigl(H(\sqrt{\gamma_h}\xi)\bigr)\right) \nonumber \\
&= \mathcal{O}_{L^2 \to L^2}\left(\frac{h^2\gamma_h^{\frac{m}{2}}}{R} \right)+\mathcal{O}_{L^2 \to L^2}\left( \frac{h^2\gamma_h^{\frac{m-1}{2}}}{R^2}\right).
\end{align}
It remains to control the higher-order odd terms in \eqref{e:reste-commutateur-s-et-H}. Each term with \(p\ge 1\) carries at least the factor \((h/\sqrt{\gamma_h})^3\), since \(2p+1\ge 3\). Moreover, \(\displaystyle \Omega_{\frac{h}{\sqrt{\gamma_h}}B}^{\,2p+1}\) involves only finitely many derivatives of the two symbols. The derivatives of \(H(\sqrt{\gamma_h}\xi)\) are uniformly bounded on the support of \(b(\sqrt{\gamma_h}\xi)\). The estimate \eqref{e:estimation-s-gamma-h-en-un-sur-R} of Lemma \ref{l:symbol-s-gamma-R} implies that, for every pair of multi-indices \(\alpha,\beta\), there exists a constant \(C_{\alpha,\beta}>0\), independent of \(h\) and \(R\), such that
\[
\bigl|\partial_x^\alpha\partial_\xi^\beta \gamma_h^\frac{m+1}{2}s_{\gamma_h,R}(x,\xi)\bigr|=\gamma_h^\frac{m}{2}\bigl|\partial_x^\alpha\partial_\xi^\beta \sqrt{\gamma_h}s_{\gamma_h,R}(x,\xi)\bigr|
\le \frac{\gamma_h^\frac{m}{2}C_{\alpha,\beta}}{R},
\]
where \(C_{\alpha,\beta}\) depends on finitely many derivatives of \(H\), \(b\), and \(\chi\). Therefore, for each fixed \(N\), all the terms with \(p\ge 1\) in \eqref{e:full-odd-expansion-final}, together with the estimate \eqref{e:reste-partie-Hamiltonienne} of the symbolic remainder \(\rho_{N,h,R}\), can be absorbed into
\begin{equation}\label{e:remainder-symbol-final}
\widetilde{\mathcal{H}}_{N,h,R}
=
\gamma_h^{\frac{m+1}{2}}
\,\mathcal{O}_{\mathscr{S}^0, R}\!\left(
\frac{h^3}{\sqrt{\gamma_h}}
\right)+\mathcal{O}_{\mathscr{S}^0,R}\!\left(h^{N/2}\right)=\mathcal{O}_{\mathscr{S}^0,R}\!\left(
h^3\gamma_h^{\frac m2}
\right)+\mathcal{O}_{\mathscr{S}^0,R}\!\left(h^{N/2}\right),
\end{equation}
where the constant entering into the remainder depends on finitely many derivatives of \(H\), \(b\), and \(\chi\), and has at most polynomial growth in \(R^{-1}\). Since \(\gamma_h\ge h\), choosing \(N\) large enough allows us to absorb the second term into the first one. Hence, one has 
\[
\widetilde{\mathcal H}_{N,h,R}
= \mathcal O_{\mathscr S^0,R}\!\left( h^3\gamma_h^{\frac m2}
\right).
\]
By the Calder\'on--Vaillancourt theorem, it follows that
\begin{equation}\label{e:remainder-operator-final}
\Opgamma(\widetilde{\mathcal H}_{N,h,R})
=
\mathcal O_{L^2\to L^2,R}\!\left(
h^3\gamma_h^{\frac m2}
\right).
\end{equation}
Combining \eqref{e:poisson-bracket-s-et-H}, \eqref{e:estimation-terme-magnetique-hamiltonien}, and \eqref{e:remainder-operator-final}, we finally obtain the decomposition
\begin{align}
&\left[\Opgamma\!\left(\gamma_h^{\frac{m+1}{2}}s_{\gamma_h,R}\right),
\Opgamma\!\left(H(\sqrt{\gamma_h}\xi)\right)\right]
\notag\\
&\qquad =
-\frac{2\pi (k\cdot \frak{e}_\Lambda) }{L_\Lambda^2}\,
h\,\Opgamma\!\left(
\gamma_h^{\frac{m+1}{2}}
e^{2i\pi k\cdot x}\,
b(\sqrt{\gamma_h}\xi)
\left(
1-\chi\!\left(
\frac{\nabla H(\sqrt{\gamma_h}\xi)\cdot \frak{e}_\Lambda}{R L_\Lambda\sqrt{\gamma_h}}
\right)
\right)
\right)
\label{e:final-commutator-decomposition}\\
&\qquad\quad
+\mathcal{O}_{L^2 \to L^2}\left(\frac{h^2\gamma_h^{\frac{m}{2}}}{R} \right)+\mathcal{O}_{L^2 \to L^2}\left( \frac{h^2\gamma_h^{\frac{m-1}{2}}}{R^2}\right)
+\mathcal{O}_{L^2\to L^2,R}\!\left(h^3\gamma_h^{\frac{m}{2}}\right).
\notag
\end{align}
\textbf{2. Perturbative contribution.}
Now that we have analyzed the Hamiltonian contribution, we show that the perturbative contribution is also negligible in the limit as for \eqref{e:remainder-symbol-final}. The key point is that the perturbation carries an additional factor \(h\), which makes the whole contribution negligible in the limit. Applying \eqref{e:commutator} in the \(h/\sqrt{\gamma_h}\)-quantization, we obtain
\begin{align}
\mathcal{C}_h^{\rm pert}&=\left[\Opgamma\!\left(\gamma_h^{\frac{m+1}{2}}s_{\gamma_h,R}\right),\Opgamma(R_h(x,\sqrt{\gamma_h}\xi))\right]\\
&= \frac{h}{i\sqrt{\gamma_h}}
\Opgamma\!\left(
\left\{\gamma_h^{\frac{m+1}{2}}s_{\gamma_h,R},R_h(x,\sqrt{\gamma_h}\xi)\right\}
\right)
\label{e:crochet-poisson-s-et-R}\\
&+
\frac{\widehat B_0 h^2}{i\gamma_h}
\Opgamma\!\left(
\partial_{\xi_2}\!\left(\gamma_h^{\frac{m+1}{2}}s_{\gamma_h,R}\right)\partial_{\xi_1}(R_h(x,\sqrt{\gamma_h}\xi))
-
\partial_{\xi_1}\!\left(\gamma_h^{\frac{m+1}{2}}s_{\gamma_h,R}\right)\partial_{\xi_2}(R_h(x,\sqrt{\gamma_h}\xi))
\right)
\label{e:reste-commutateur-s-et-R}\\
&+\Opgamma(\widetilde{\mathcal R}_{N,h,R}), \notag
\end{align}
where \(\widetilde{\mathcal R}_{N,h,R}\) denotes the remainder obtained by collecting all the terms of order \(2p+1\ge 3\) in the odd symbolic expansion of the commutator, together with the symbolic remainder, exactly as in the Hamiltonian contribution.

On the support of \(b(\sqrt{\gamma_h}\xi)\), one has for $j=1,2$
\[
\partial_{x_j} (R_h(x,\sqrt{\gamma_h}\xi)) = \mathcal{O}_{\mathscr{S}^0}(1),
\qquad
\partial_{\xi_j} (R_h(x,\sqrt{\gamma_h}\xi))= \mathcal{O}_{\mathscr{S}^0}(\sqrt{\gamma_h}).
\]
Moreover, from Lemma \ref{l:symbol-s-gamma-R}, one can use \eqref{e:estimation-partialx-de-sgammah} to have $\partial_{x_j} (\gamma_h^\frac{m+1}{2}s_{\gamma_h,R})=\mathcal{O}_{\mathscr{S}^0}\!\left(\frac{\gamma_h^\frac{m}2{}}{R}\right)$. Hence, using \eqref{e:estimation-derivée-gamma-s} one gets
\begin{equation}\label{e:estimation-crochet-de-poissonb-s-et-R}
\left\{\gamma_h^{\frac{m+1}{2}}s_{\gamma_h,R},R_h(x,\sqrt{\gamma_h}\xi)\right\}
=
\mathcal{O}_{\mathscr{S}^0}\!\left(\frac{\gamma_h^{\frac{m+1}{2}}}{R}\right)
+
\mathcal{O}_{\mathscr{S}^0}\!\left(\frac{\gamma_h^{m/2}}{R^2}\right).
\end{equation}
Similarly, using again \eqref{e:estimation-derivée-gamma-s}, we obtain
\begin{align}\label{e:estimation-ordre-2-commutateur-s-et-R}
\partial_{\xi_2}\!\left(\gamma_h^{\frac{m+1}{2}}s_{\gamma_h,R}\right)\partial_{\xi_1}(R_h(x,\sqrt{\gamma_h}\xi))
&- \partial_{\xi_1}\!\left(\gamma_h^{\frac{m+1}{2}}s_{\gamma_h,R}\right)\partial_{\xi_2}(R_h(x,\sqrt{\gamma_h}\xi))\nonumber\\
&= \mathcal{O}_{\mathscr{S}^0}\!\left(\frac{\gamma_h^{\frac{m+2}{2}}}{R}\right)+\mathcal{O}_{\mathscr{S}^0}\!\left(\frac{\gamma_h^{\frac{m+1}{2}}}{R^2}\right).
\end{align}
Finally, the higher odd terms in the full symbolic expansion are treated as in
the Hamiltonian contribution. Using the estimates above, together with the
uniform symbol bounds on \(R_h(x,\sqrt{\gamma_h}\xi)\), the terms of order
\(2p+1\ge3\) are bounded, for fixed \(R>0\), by
\[
\mathcal O_{\mathscr S^0,R}\!\left(
h^3\gamma_h^{\frac m2}
\right).
\]
It remains to control the symbolic remainder in the composition formula. Since
\(\gamma_h\ge h\), the same argument as in \eqref{e:reste-partie-Hamiltonienne}
gives
\[
\rho_{N,h,R}^{\rm pert}
=
\mathcal O_{\mathscr S^0,R}\!\left(
\left(\frac{h}{\sqrt{\gamma_h}}\right)^N
\right)
=
\mathcal O_{\mathscr S^0,R}\!\left(h^{N/2}\right).
\]
Thus the full perturbative remainder satisfies
\[
\widetilde{\mathcal R}_{N,h,R}
=
\mathcal O_{\mathscr S^0,R}\!\left(
h^3\gamma_h^{\frac m2}
\right)
+
\mathcal O_{\mathscr S^0,R}\!\left(h^{N/2}\right).
\]
Choosing \(N\) large enough allows us to absorb the second term into the first one. Hence, one has 
\begin{equation}\label{e:terme-de-reste-s-et-R}
\widetilde{\mathcal R}_{N,h,R}
=
\mathcal O_{\mathscr S^0,R}\!\left(
h^3\gamma_h^{\frac m2}
\right).
\end{equation}
The constants implicit in these estimates depend on finitely many seminorms of
\(R_h\), uniformly with respect to \(h\), and on finitely many derivatives of
\(H\), \(b\), and \(\chi\). They may depend on \(R\), but are independent of
\(h\). Combining \eqref{e:estimation-crochet-de-poissonb-s-et-R},
\eqref{e:estimation-ordre-2-commutateur-s-et-R} and
\eqref{e:terme-de-reste-s-et-R} and using the Calderón-Vaillancourt theorem, we obtain
\begin{align}
h\mathcal{C}_h^{\rm pert}&=\Opgamma\!\left(
\mathcal{O}_{\mathscr{S}^0}\!\left(\frac{h^2\,\gamma_h^{m/2}}{R}\right)
+
\mathcal{O}_{\mathscr{S}^0}\!\left(\frac{h^2\,\gamma_h^{\frac{m-1}{2}}}{R^2}\right) +
\mathcal{O}_{\mathscr{S}^0}\!\left(\frac{h^3\gamma_h^{m/2}}{R}\right)
+
\mathcal{O}_{\mathscr{S}^0}\!\left(\frac{h^3\gamma_h^{\frac{m-1}{2}}}{R^2}\right)
\right)\label{e:estimation-commutateur-perturbation}\\
&+
\mathcal{O}_{L^2\to L^2,R}\!\left(h^4\gamma_h^{\frac{m}{2}}
\right).
\notag
\end{align}
Finally, combining \eqref{e:final-commutator-decomposition} with
\eqref{e:estimation-commutateur-perturbation} in
\eqref{e:split-commutateur}, and applying the Calder\'on--Vaillancourt theorem
to the symbolic remainders, we obtain
\begin{align}
\mathcal C_h^{\rm Ham}+h\mathcal C_h^{\rm pert}
=&
-\frac{2\pi (k\cdot \frak{e}_\Lambda)}{L_\Lambda^2}\,
h\gamma_h^{\frac{m+1}{2}}\,
\Opgamma\!\left(
e^{2i\pi k\cdot x}\,
b(\sqrt{\gamma_h}\xi)
\left(
1-\chi\!\left(
\frac{\nabla H(\sqrt{\gamma_h}\xi)\cdot \frak e_\Lambda}
{R L_\Lambda\sqrt{\gamma_h}}
\right)
\right)
\right) \nonumber\\
& +\mathcal{O}_{L^2\to L^2,R}\left( \frac{h^2\gamma_h^{m/2}}{R}
+
\frac{h^2\gamma_h^{\frac{m-1}{2}}}{R^2}
+
h^3\gamma_h^{\frac{m}{2}} \right).
\label{e:final-operator-splitting}
\end{align}
We now conclude the proof. Inserting the decomposition \eqref{e:final-operator-splitting} into the commutator expression \eqref{e:non-compact-commutator}, and comparing the resulting identity with \eqref{e:calcul-crochet-de-lie-et-uh-solution}, we obtain, after using the rescaling formula \eqref{e:rescaling-semiclassique}, that for every fixed \(R>1\),
\begin{align}
&2\pi \frac{k \cdot \frak{e}_\Lambda}{L_\Lambda^2}
\int_\R \psi(t)
\Big\langle
\Op\!\left(
e^{2i\pi k\cdot x}
b(\xi)
\left(
1-\chi\!\left(
\frac{\nabla H(\xi)\cdot \frak e_\Lambda}
{R L_\Lambda\sqrt{\gamma_h}}
\right)
\right)
\right)
u_h(t \tau_h),u_h(t \tau_h)
\Big\rangle\,\dd t 
\label{e:identity-before-conclusion}\\
&\qquad=
\frac{i}{\tau_h\sqrt{\gamma_h}}
\int_\R \psi'(t)
\Big\langle
\Opgamma\!\left(\sqrt{\gamma_h}s_{\gamma_h,R}\right)
u_h(t \tau_h),u_h(t \tau_h)
\Big\rangle\,\dd t  \nonumber\\
&\qquad\quad
+\mathcal{O}_{R}\!\left(
\frac{h}{R\sqrt{\gamma_h}}
+\frac{h}{\gamma_hR^2}
+\frac{h^2}{\sqrt{\gamma_h}}
\right).
\nonumber
\end{align}
Here the last term follows from \eqref{e:final-operator-splitting}, after pairing the operator remainder with \(u_h(t\tau_h)\), integrating against \(\psi(t)\), and using the conservation of the \(L^2\)-norm. Moreover, by the estimate \eqref{e:estimation-s-gamma-h-en-un-sur-R} of Lemma~\ref{l:symbol-s-gamma-R}, one has
\[
\sqrt{\gamma_h}s_{\gamma_h,R}
=
\mathcal{O}_{\mathscr{S}^0}\!\left(\frac1R\right),
\]
uniformly in \(h\). Therefore, by the Calder\'on--Vaillancourt theorem and the conservation of the \(L^2\)-norm, one has
\[
\int_\R \psi'(t)
\Big\langle
\Opgamma\!\left(\sqrt{\gamma_h}s_{\gamma_h,R}\right)
u_h(t \tau_h),u_h(t \tau_h)
\Big\rangle\,\dd t 
=
\mathcal O\!\left(\frac1R\right),
\]
uniformly in \(h\). Hence \eqref{e:identity-before-conclusion} yields
\begin{align}
2\pi \frac{k \cdot \frak{e}_\Lambda}{L_\Lambda^2}
\int_\R &\psi(t)
\Big\langle
\Op\!\left(
e^{2i\pi k\cdot x}
b(\xi)
\left(
1-\chi\!\left(
\frac{\nabla H(\xi)\cdot \frak e_\Lambda}
{R L_\Lambda\sqrt{\gamma_h}}
\right)
\right)
\right)
u_h(t \tau_h),u_h(t \tau_h)
\Big\rangle\,\dd t 
\label{e:identite-finale}\\
&=
\mathcal{O}_{R}\!\left(
\frac{1}{\tau_h\sqrt{\gamma_h}R}
+\frac{h}{R\sqrt{\gamma_h}}
+\frac{h}{\gamma_hR^2}
+\frac{h^2}{\sqrt{\gamma_h}}
\right).
\nonumber
\end{align}
By the property of \((\gamma_h)_{h\to0^+}\) in Definition~\ref{d:def-gamma}, one gets
\[
\frac{1}{\tau_h\sqrt{\gamma_h}}=\mathcal{O}_{h \to 0^+}(1),\qquad
\frac{h}{\sqrt{\gamma_h}}\to0,\qquad
\frac{h}{\gamma_h}=\mathcal{O}_{h \to 0^+}(1),\qquad
\frac{h^2}{\sqrt{\gamma_h}}\to0.
\]
Therefore, for every fixed \(R>1\), one has 
\[
\limsup_{h\to0^+}
\left|
\int_\R \psi(t)
\Big\langle
\Op\!\left(
e^{2i\pi k\cdot x}
b(\xi)
\left(
1-\chi\!\left(
\frac{\nabla H(\xi)\cdot \frak e_\Lambda}
{R L_\Lambda\sqrt{\gamma_h}}
\right)
\right)
\right)
u_h(t \tau_h),u_h(t \tau_h)
\Big\rangle\,\dd t 
\right| = \mathcal{O}\left(\frac1R+\frac1{R^2}\right).
\]
Letting \(R\to+\infty\) proves that the limit in \eqref{e:symbol-voulu} vanishes.
\end{proof}

\subsection{The resonant part: a two-microlocal lift}

To capture the fine structure near $\nabla H(\xi)\cdot\mathfrak e_\Lambda=0$, we introduce a two-microlocal variable
\begin{equation}\label{e:definition-eta-gamma-h}
\eta_{\gamma_h}(\xi):=\frac{\nabla H(\xi)\cdot\mathfrak e_\Lambda}{L_\Lambda\sqrt{\gamma_h}},
\end{equation}
where $(\gamma_h)_{h\to 0^+}$ is defined as in \eqref{e:definition-gammah}. The cutoff $\chi\!\left(\frac{\nabla H(\xi)\cdot\mathfrak e_\Lambda}{R L_\Lambda\sqrt{\gamma_h}}\right)$ localizes to a $R\sqrt{\gamma_h}$-neighborhood of the resonant set $\{\nabla H(\xi)\cdot\mathfrak e_\Lambda=0\}$, and the variable $\eta$ remains of order $R$ in this region.

\begin{definition}[Two-microlocal distribution near $\Lambda$]\label{def:2micro}
For $a\in \mathcal{C}_c^\infty(\R\times T^*\T^2\times\R)$, set
\begin{equation}\label{e:definition-distrib-liftée}
\langle \widetilde W_{h,R,\Lambda}^B,a\rangle
:=\int_\R \Big\langle \Op\!\left( a\!\left(t,x,\xi,\eta_{\gamma_h}(\xi)\right) \chi\!\left(\frac{\eta_{\gamma_h}(\xi)}{R}\right)
\right)u_h(t\tau_h),u_h(t\tau_h)\Big\rangle\,\dd t  .
\end{equation}
\end{definition}
This two-microlocal object is a refined version of the previous Wigner distributions, with test symbols depending on the additional variable $\eta$. We first establish its compactness, positivity, projection and time-disintegration properties in Lemma~\ref{l:2micro}. We then prove its invariance under the lifted geodesic flow in Lemma \ref{l:RN-2micro}, before deriving the effective dynamics of its accumulation points in Lemma~\ref{l:three-regimes-2micro}.

\begin{lemma}[Compactness and projection]\label{l:2micro} Let $(u_h)_{h\to 0^+}$ solve \eqref{e:Schrodinger-PDE} and be spectrally localized in the sense of \eqref{e:spectral-projection}. The two-microlocal distribution introduced above satisfies the following: 
\begin{enumerate}
\item $(\widetilde W_{h,R,\Lambda}^B)_{h\to0^+,\;R\to+\infty}$ is bounded in
$\mathcal D'(\R\times T^*\T^2\times\R)$.
\item Any accumulation point $\widetilde W_\Lambda^B$ is a nonnegative Radon measure supported on $\R\times \T^2\times \Omega_{E_1,E_2} \times\R$.
\item The marginal in $\eta$ recovers the resonant part:
\begin{equation}\label{e:eta-marginal}
W_\Lambda^B(t , x, \xi)=\int_\R \widetilde W_\Lambda^B( t , x, \xi,\dd\eta).
\end{equation}
\item There exists a measurable family $(\widetilde\nu_{t,\Lambda}^B)$ of probability measures on $\T^2 \times \Omega_{E_1,E_2} \times\R$ such that 
\begin{equation}\label{e:2micro-disint}
    \widetilde W_\Lambda^B(\dd t ,\dd x,\dd \xi,\dd\eta)
=\widetilde\nu_{t,\Lambda}^B(\dd x,\dd \xi,\dd\eta)\otimes h_\Lambda(t) \dd t, 
\end{equation}
where $h_\Lambda \in L_{\mathrm{loc}}^1(\R)$ is given by Lemma \ref{l:RN-split}. Moreover, for a.e $t\in \R$, one has
\begin{equation}\label{e:projection-x-xi-du-lift-2micro}
\nu_{t,\Lambda}^B(\dd x,\dd\xi)
=
\int_\R
\widetilde\nu_{t,\Lambda}^B(\dd x,\dd\xi,\dd\eta).
\end{equation}
Equivalently,
\[
\left( \pi_{x,\xi}\right)_*\widetilde\nu_{t,\Lambda}^B
=
\nu_{t,\Lambda}^B,
\]
where $\pi_{x,\xi}$ denotes the canonical projection defined by $\pi_{x,\xi} : (x,\xi,\eta) \in \T^2 \times \R^2 \times \R \mapsto (x,\xi)$.
\end{enumerate}
\end{lemma}
\begin{proof}
Using the same argument as in Lemma~\ref{l:magnetic-wigner}, after the rescaling $\xi\mapsto \sqrt{\gamma_h}\xi$ the symbol is uniformly bounded in $\mathscr{S}^0$ uniformly in time, so Theorem \ref{t:Calderon-Vaillancourt} gives the boundedness of $(\widetilde W_{h,R,\Lambda}^B)$ in $\mathcal D'(\R\times T^*\T^2\times\R)$. Positivity follows by the usual regularization argument: for $a\geq 0$, set $a_\varepsilon=\sqrt{a+\varepsilon}$ and use the magnetic adjoint \eqref{e:formal-adjoint} and composition
formulas \eqref{e:composition-rule} exactly as in the proof of Lemma~\ref{l:magnetic-wigner}. Hence any accumulation point $\widetilde W_\Lambda^B$ is a nonnegative Radon measure on $\R\times T^*\T^2\times\R$. The support property in the $\xi$ variable is proved exactly as for $W^B$; see Lemma \ref{l:support-de-W^B}. The marginal identity is obtained by testing against $a(t,x,\xi)\chi(\eta/R_1)$, then letting $R_1\to\infty$ and comparing with the definition of $W_\Lambda^B$. More precisely, one has
\begin{align*}
    \displaystyle &\int_{\R \times T^*\T^2 \times \R} a(t,x,\xi)\chi\left(\frac{\eta}{R_1}\right) \widetilde{W}_\Lambda^B(\dd t,\dd x,\dd\xi,\dd\eta) \\
    &=\lim_{R \to \infty}\lim_{h\to 0^+} \int_\R \left\langle \Op\left[ a(t,x,\xi)\chi\left( \frac{\nabla H(\xi) \cdot \frak{e}_\Lambda}{R_1 L_\Lambda \sqrt{\gamma_h}}\right)\chi\left(\frac{\nabla H(\xi) \cdot \frak{e}_\Lambda}{RL_\Lambda \sqrt{\gamma_h}}\right)\right] u_h(t\tau_h) , u_h(t\tau_h) \right\rangle \dd t.
\end{align*}
Since $\chi\equiv 1$ near $0$, for every fixed $R_1>0$ there exists $C_{R_1}>0$ such that, for every $R\geq C_{R_1}$,
\begin{center}
    $\displaystyle \chi\left( \frac{\nabla H(\xi) \cdot \frak{e}_\Lambda}{ RL_\Lambda \sqrt{\gamma_h}}\right)\chi\left(\frac{\nabla H(\xi) \cdot \frak{e}_\Lambda}{R_1L_\Lambda \sqrt{\gamma_h}}\right) =\chi\left(\frac{\nabla H(\xi) \cdot \frak{e}_\Lambda}{R_1L_\Lambda \sqrt{\gamma_h}}\right).$
\end{center}
Hence, for every fixed $R_1>0$, one gets
\begin{align*}
\int_{\R \times T^*\T^2 \times \R}
a(t,x,\xi)&\chi\left(\frac{\eta}{R_1}\right)
\widetilde{W}_\Lambda^B(\dd t,\dd x,\dd\xi,\dd\eta)\\
&=\lim_{h\to 0^+} \int_\R \left\langle
\Op\left[
a(t,x,\xi)\chi\left(
\frac{\nabla H(\xi) \cdot \frak{e}_\Lambda}{R_1L_\Lambda\sqrt{\gamma_h}}
\right)\right]u_h(t\tau_h),u_h(t\tau_h)
\right\rangle \dd t.
\end{align*}
Letting $R_1\to\infty$ and using the definition of $W_\Lambda^B$, we obtain
\[
\lim_{R_1\to\infty}
\int_{\R \times T^*\T^2 \times \R}
a(t,x,\xi)\chi\left(\frac{\eta}{R_1}\right)
\widetilde{W}_\Lambda^B(\dd t,\dd x,\dd\xi,\dd\eta)
=
\int_{\R\times T^*\T^2}
a(t,x,\xi)\,W_\Lambda^B(\dd t,\dd x,\dd\xi).
\]
This proves that the marginal of $\widetilde W_\Lambda^B$ in the $\eta$ variable is $W_\Lambda^B$, namely
\[
W_\Lambda^B(t,x,\xi)=\int_\R \widetilde W_\Lambda^B(t,x,\xi,\dd\eta).
\]
In order to prove the time disintegration of $\widetilde W_\Lambda^B$, define
\begin{equation}\label{e:projection}
\begin{aligned}
\pi_{(t,x,\xi)} &: \R \times \T^2 \times \R^2 \times \R \longrightarrow \R \times \T^2 \times \R^2,
& (t,x,\xi,\eta) &\longmapsto (t,x,\xi), \\
\pi_t &: \R \times \T^2 \times \R^2 \longrightarrow \R,
& (t,x,\xi) &\longmapsto t, \\
\widetilde\pi_t &:= \pi_t \circ \pi_{(t,x,\xi)}.
\end{aligned}
\end{equation}
Then the marginal identity \eqref{e:eta-marginal} can be rewritten as
\begin{equation}\label{e:eta-marginal-projection}
(\pi_{(t,x,\xi)})_* \widetilde W_\Lambda^B = W_\Lambda^B.
\end{equation}
Applying the push-forward by $\pi_t$ to both sides of
\eqref{e:eta-marginal-projection}, we obtain
\begin{equation}\label{e:double-push-forward}
(\pi_t)_*\big[(\pi_{(t,x,\xi)})_* \widetilde W_\Lambda^B\big]
=
(\pi_t)_* W_\Lambda^B.
\end{equation}
By functoriality of push-forwards, the left-hand side is
\[
(\pi_t\circ \pi_{(t,x,\xi)})_* \widetilde W_\Lambda^B
=
(\widetilde\pi_t)_* \widetilde W_\Lambda^B.
\]
On the other hand, by \eqref{e:disint-split} and Lemma~\ref{l:RN-split}, one has
\[
(\pi_t)_* W_\Lambda^B = \omega_\Lambda^B(\dd t)=h_\Lambda(t)\,\dd t.
\]
Therefore, one gets
\[
(\widetilde\pi_t)_* \widetilde W_\Lambda^B = h_\Lambda(t)\,\dd t.
\]
Since $h_\Lambda\in L^1_{\mathrm{loc}}(\R)$, this measure is $\sigma$-finite. The disintegration Theorem \ref{t:desintegration} may therefore be applied to $\widetilde W_\Lambda^B$ with respect to the time marginal, yielding a measurable family $(\widetilde\nu_{t,\Lambda}^B)_{t\in\R}$ of probability measures such that
\begin{equation}\label{e:desintegration-du-lift-2-microlocal}
\widetilde W_\Lambda^B(\dd t,\dd x,\dd\xi,\dd\eta)
=
\widetilde\nu_{t,\Lambda}^B(\dd x,\dd\xi,\dd\eta)\otimes h_\Lambda(t)\,\dd t.
\end{equation}
The identity \((\pi_{x,\xi})_*\widetilde\nu_{t,\Lambda}^B=\nu_{t,\Lambda}^B\) is then a direct consequence of the uniqueness of disintegration. Indeed, since
\[
(\pi_{t,x,\xi})_*\widetilde W_\Lambda^B=W_\Lambda^B,
\]
pushing forward the above disintegration \eqref{e:desintegration-du-lift-2-microlocal} gives
\[
W_\Lambda^B
=
\bigl((\pi_{x,\xi})_*\widetilde\nu_{t,\Lambda}^B\bigr)
\otimes h_\Lambda(t)\dd t.
\]
Comparing with the disintegration $W_\Lambda^B=\nu_{t,\Lambda}^B\otimes h_\Lambda(t)\dd t$ yields
\[
(\pi_{x,\xi})_*\widetilde\nu_{t,\Lambda}^B=\nu_{t,\Lambda}^B
\]
for \(h_\Lambda(t)\dd t\)-almost every \(t\).
\end{proof}

In order to analyze the properties of the measure \(\widetilde\nu_{t,\Lambda}^B\), we first introduce a rescaled symbol adapted to the \(h/\sqrt{\gamma_h}\)-calculus. The purpose of the next lemma is to ensure that these symbols belong to the appropriate symbol classes uniformly in \(h\), so that the same commutator strategy as in Lemma~\ref{l:dyn-split} can be applied.
\begin{lemma}\label{l:classe-de-s-tilde}
Let $a \in \mathcal{C}_c^\infty(T^*\T^2\times \R)$, $\chi \in \mathcal{C}_c^\infty(\R)$ and $(\gamma_h)_{h\to 0^+}$ satisfying \eqref{e:regime-pour-gamma}. The symbol $\widetilde{s}_{\gamma_h,R}$ defined by 
\begin{equation}\label{e:definition-s-tilde}
    \displaystyle \widetilde{s}_{\gamma_h,R} : (x,\xi) \in T^*\T^2 \mapsto a\left(x,\sqrt{\gamma_h}\xi, \eta_{\gamma_h}(\sqrt{\gamma_h}\xi)\right)\chi\left(\frac{\eta_{\gamma_h}(\sqrt{\gamma_h}\xi)}{R}\right), 
\end{equation}
belongs to the class $\displaystyle \mathscr{S}^0(T^*\T^2)$ uniformly in $h$ (with seminorms depending on $R$). Moreover, one has $\displaystyle \gamma_h^{m/2}\widetilde{s}_{\gamma_h,R} \in \mathscr{S}^{-m}(T^*\T^2)$ uniformly in $h$ for any $m \ge 0$. Finally, for every $j\in\{1,2\}$, one has in $\mathscr S^0(T^*\T^2)$, uniformly with respect to $h$,
\begin{equation}\label{e:refined-xi-derivative-s-tilde}
\begin{aligned}
\partial_{\xi_j}\widetilde s_{\gamma_h,R}
&=
\left(
\mathrm{Hess}(H)(\sqrt{\gamma_h}\xi)
\frac{\mathfrak e_\Lambda}{L_\Lambda}
\cdot e_j
\right)
\partial_\eta a\left(x,\sqrt{\gamma_h}\xi,\eta_{\gamma_h}(\sqrt{\gamma_h}\xi)\right)
\chi\left(\frac{\eta_{\gamma_h}(\sqrt{\gamma_h}\xi)}{R}\right)
\\
&\quad
+\mathcal O_{\mathscr S^0,R}(\sqrt{\gamma_h})
+
\mathcal O_{\mathscr S^0,R}\!\left(\frac1R\right),
\end{aligned}
\end{equation}
where $(e_1,e_2)$ denotes the canonical basis of $\R^2$.
\end{lemma}

\begin{proof}
Fix $R>1$. From the definition \eqref{e:definition-eta-gamma-h}, one can write 
\[
\widetilde{s}_{\gamma_h,R}(x,\xi)
=
a\bigl(x,\sqrt{\gamma_h}\xi,\eta_{\gamma_h}(\sqrt{\gamma_h}\xi)\bigr)\,
\chi\!\left(\frac{\eta_{\gamma_h}(\sqrt{\gamma_h}\xi)}{R}\right).
\]
Since $a\in \mathcal{C}_c^\infty(T^*\T^2\times\R)$, there exists a compact set $K\subset \R^2\times \R$ such that $\operatorname{supp} a \subset \T^2\times K$. Let $K_1\subset \R^2$ be the projection of $K$ onto the $\R^2$ variable. Then, on $\operatorname{supp}(\widetilde{s}_{\gamma_h,R})$, one necessarily has $\sqrt{\gamma_h}\xi$ bounded. Moreover, for every $|\beta|\ge 2$, one has $\partial_\xi^\beta (\sqrt{\gamma_h}\xi)=0.$ In particular, all derivatives of \(\sqrt{\gamma_h}\xi\) are uniformly bounded. One also has that, for every $|\beta|\geq 1$,
\[
\partial_\xi^\beta \left(\eta_{\gamma_h}(\sqrt{\gamma_h}\xi)\right)
=
\frac{\gamma_h^{(|\beta|-1)/2}}{L_\Lambda}\,
\bigl(\partial^\beta(\nabla H\cdot \mathfrak e_\Lambda)\bigr)(\sqrt{\gamma_h}\xi).
\]
Since $(\gamma_h)_{h \to 0^+}$ satisfies \eqref{e:regime-pour-gamma}, in particular $\gamma_h\leq 1$ for $h$ small, and since $\sqrt{\gamma_h}\xi$ stays in the fixed compact set $K_1$ on the support of $\widetilde{s}_{\gamma_h,R}$, one has uniformly in $h$, 
\[
|\partial_\xi^\alpha (\sqrt{\gamma_h}\xi)|\leq C_\alpha,
\qquad
|\partial_\xi^\alpha \eta_{\gamma_h}(\sqrt{\gamma_h}\xi)|\leq C_\alpha \qquad \text{for } |\alpha|\geq 1.
\]
We do not need any uniform bound on $\eta_{\gamma_h}(\sqrt{\gamma_h}\xi)$ itself due to the fact that when differentiating the composition above, the chain rule only involves derivatives of $\eta_{\gamma_h}(\sqrt{\gamma_h}\xi)$ of order at least one, while the factors involving $a$, $\chi$ and their derivatives are uniformly bounded. Since $a$ and $\chi$ are smooth with compact support, all their derivatives are bounded. Therefore, Leibniz' rule and repeated chain rule imply that for every multi-index $\alpha,\beta$,
\[
\sup_{x,\xi}
\bigl|
\partial_x^\alpha\partial_\xi^\beta \widetilde{s}_{\gamma_h,R}(x,\xi)
\bigr|
\leq C_{\alpha,\beta,R},
\]
which proves that $\widetilde{s}_{\gamma_h,R}\in \mathscr{S}^0(T^*\T^2)$ uniformly in $h$.

Moreover, since $a$ is compactly supported in the $\xi$ variable, there exists $M>0$ such that
\[
|\xi|\leq M\gamma_h^{-1/2}
\qquad\text{on }\operatorname{supp} \widetilde{s}_{\gamma_h,R}.
\]
Therefore, on this support, one has
\[
\gamma_h^{m/2}\leq C_m\langle \xi\rangle^{-m}.
\]
Since $\widetilde{s}_{\gamma_h,R}$ is uniformly bounded in $\mathscr{S}^0$, we obtain for all multi-indices $\alpha,\beta$,
\[
\bigl|
\partial_x^\alpha\partial_\xi^\beta
\bigl(\gamma_h^{m/2}\widetilde{s}_{\gamma_h,R}\bigr)(x,\xi)
\bigr|
\leq
C_{\alpha,\beta,R}\,\gamma_h^{m/2}
\leq
C_{\alpha,\beta,m,R}\,\langle \xi\rangle^{-m},
\]
which proves the claim. In order to prove the estimate \eqref{e:refined-xi-derivative-s-tilde}, one has by differentiating $\widetilde s_{\gamma_h,R}$ with respect to $\xi_j$
\begin{align}
\partial_{\xi_j}\widetilde s_{\gamma_h,R}
&=
\sqrt{\gamma_h}\,
\partial_{\xi_j}a\!\left(x,\sqrt{\gamma_h}\xi,\eta_{\gamma_h}(\sqrt{\gamma_h}\xi)\right)
\chi\left(\frac{\nabla H(\sqrt{\gamma_h}\xi)\cdot \frak{e}_\Lambda}{RL_\Lambda\sqrt{\gamma_h}}\right)
\label{e:derivative-symbol-xi-stilde}
\\
&\quad
+\partial_\eta a\!\left(x,\sqrt{\gamma_h}\xi,\eta_{\gamma_h}(\sqrt{\gamma_h}\xi)\right)
\partial_{\xi_j}\!\bigl(\eta_{\gamma_h}(\sqrt{\gamma_h}\xi)\bigr)
\chi\left(\frac{\nabla H(\sqrt{\gamma_h}\xi)\cdot \frak{e}_\Lambda}{RL_\Lambda\sqrt{\gamma_h}}\right)
\nonumber\\
&\quad
+\frac1R\,
a\!\left(x,\sqrt{\gamma_h}\xi,\eta_{\gamma_h}(\sqrt{\gamma_h}\xi)\right)
\chi'\left(\frac{\nabla H(\sqrt{\gamma_h}\xi)\cdot \frak{e}_\Lambda}{RL_\Lambda\sqrt{\gamma_h}}\right)
\partial_{\xi_j}\!\bigl(\eta_{\gamma_h}(\sqrt{\gamma_h}\xi)\bigr).
\nonumber
\end{align}
The first term in \eqref{e:derivative-symbol-xi-stilde} is $\mathcal O_{\mathscr S^0}(\sqrt{\gamma_h})$. Indeed, further derivatives only fall on smooth compactly supported factors composed with $\sqrt{\gamma_h}\xi$ and $\eta_{\gamma_h}(\sqrt{\gamma_h}\xi)$ and derivatives of $\eta_{\gamma_h}(\sqrt{\gamma_h}\xi)$ of order at least one are uniformly bounded while the prefactor $\sqrt{\gamma_h}$ remains. For the second term, using 
\[ 
\partial_{\xi_j}\eta_{\gamma_h}(\sqrt{\gamma_h}\xi) = \mathrm{Hess}(H)(\sqrt{\gamma_h}\xi) \frac{\mathfrak e_\Lambda}{L_\Lambda}\cdot e_j, 
\] we obtain the leading contribution. Finally, the third term is $\mathcal O_{\mathscr S^0}(1/R)$. The factor $1/R$ is explicit, and further derivatives produce uniformly bounded factors, possibly with additional powers of $1/R$ from derivatives of the cutoff. Therefore, one has \begin{align}\label{e:estimation-derivee-s-tilde} 
\partial_{\xi_j}\widetilde s_{\gamma_h,R} 
&= \mathrm{Hess}(H)(\sqrt{\gamma_h}\xi) \frac{\mathfrak e_\Lambda}{L_\Lambda} \cdot e_j\, \partial_\eta a\left(x,\sqrt{\gamma_h}\xi,\eta_{\gamma_h}(\sqrt{\gamma_h}\xi)\right) \chi\left(\frac{\eta_{\gamma_h}(\sqrt{\gamma_h}\xi)}{R}\right) \\ 
&\quad +\mathcal{O}_{\mathscr{S}^0,R}(\sqrt{\gamma_h}) +\mathcal{O}_{\mathscr{S}^0,R}\left( \frac{1}{R}\right). \nonumber 
\end{align} 
\end{proof}

\begin{lemma}\label{l:RN-2micro}
The measure $\widetilde\nu_{t,\Lambda}^B$ is invariant under the lifted geodesic flow \[
\widetilde\varphi_H^s(x,\xi,\eta):=(x+s\nabla H(\xi),\xi,\eta).
\]
\end{lemma}

\begin{proof}
Let $\psi\in \mathcal{C}_c^\infty(\R)$ and let
$a\in \mathcal{C}_c^\infty(T^*\T^2\times \R)$.
For $R>1$, define
\[
\mathfrak{s}_{h,R}(x,\xi)
:=
a\!\left(
x,\xi,\frac{\nabla H(\xi)\cdot \mathfrak e_\Lambda}{L_\Lambda\sqrt{\gamma_h}}
\right)
\chi\!\left(
\frac{\nabla H(\xi)\cdot \mathfrak e_\Lambda}{R L_\Lambda\sqrt{\gamma_h}}
\right).
\]
By the rescaling identity \eqref{e:rescaling-semiclassique}, in the $h/\sqrt{\gamma_h}$ quantization one has
\[
\Op(\mathfrak{s}_{h,R})
=
\Opgamma\!\big(\widetilde s_{\gamma_h,R}\big),
\]
where $\widetilde s_{\gamma_h,R}$ is defined as in \eqref{e:definition-s-tilde} and also
\begin{equation}\label{e:definition-p-tilde}
    \widehat H_h= \Opgamma(\widetilde p_h),
\qquad
\widetilde p_h(x,\xi)
:=
H(\sqrt{\gamma_h}\xi)+h\,R_h(x,\sqrt{\gamma_h}\xi).
\end{equation}
Since $(u_h)_{h\to0^+}$ solves \eqref{e:Schrodinger-PDE}, an integration by parts
in time yields
\begin{equation}\label{e:heisenberg-tilde-s}
\begin{aligned}
&\frac{\gamma_h^{m/2}}{\tau_h}
\int_\R \psi'(t)\,
\Big\langle
\Opgamma\!\big(\widetilde s_{\gamma_h,R}\big)
u_h(t\tau_h),u_h(t\tau_h)
\Big\rangle \dd t \\
&\qquad =
\int_\R \psi(t)\,
\Bigg\langle
\frac{i}{h}
\big[
\Opgamma\!\big(\gamma_h^{m/2}\widetilde s_{\gamma_h,R}\big),\widehat H_h
\big]
u_h(t\tau_h),
u_h(t\tau_h)
\Bigg\rangle \dd t .
\end{aligned}
\end{equation}
By Lemma~\ref{l:classe-de-s-tilde}, the family $\widetilde s_{\gamma_h,R}$ is bounded in $\mathscr{S}^0(T^*\T^2)$, uniformly in $h$ (for fixed $R$). Hence, by the Calder\'on--Vaillancourt theorem~\ref{t:Calderon-Vaillancourt} and the unitarity of the propagator, the left-hand side of \eqref{e:heisenberg-tilde-s} is
\begin{equation}\label{e:estimation-terme-e:heisenberg-tilde-s}
\mathcal O_R\!\left(\frac{\gamma_h^{m/2}}{\tau_h}\right) = o_R\!\left(\gamma_h^{m/2}\right),
\end{equation}
since $\tau_h\to+\infty$.

We now compute the commutator, using the fact that $\widetilde p_h \in \mathscr{S}^m$ and $\gamma_h^{m/2}\widetilde s_{\gamma_h,R} \in \mathscr{S}^{-m}$ by Lemma \ref{l:classe-de-s-tilde}. Applying the commutator formula \eqref{e:commutator} with semiclassical parameter $h/\sqrt{\gamma_h}$, one gets
\begin{equation}\label{e:commutator-delta}
\big[
\Opgamma\!\big(\gamma_h^\frac{m}{2}\widetilde s_{\gamma_h,R}\big),
\Opgamma(\widetilde p_h)
\big]
=
\frac{h}{i\sqrt{\gamma_h}}\,
\Opgamma\!\Big(
\big\{\gamma_h^\frac{m}{2}\widetilde s_{\gamma_h,R},\widetilde p_h\big\}
\Big)
+
\Opgamma(r_{h,R}),
\end{equation}
where the remainder symbol $r_{h,R}$ satisfies
\begin{equation}\label{e:rhR-estimate-final}
r_{h,R}= \mathcal{O}_{\mathscr{S}^0}\!\left(h^2\gamma_h^{(m-1)/2}\right).
\end{equation}
Using the commutator formula \eqref{e:commutator} in the rescaled setting, a direct application first gives that $r_{h,R}$ is $\mathcal{O}_{\mathscr{S}^0}\!\left(h^2\gamma_h^{(m-2)/2}\right)$, but this estimate can be improved. Indeed, the symbolic expansion consists of a finite sum of explicit terms, together with a remainder controlled by the same symbol seminorms. The explicit terms are products of derivatives of $\gamma_h^{m/2}\widetilde s_{\gamma_h,R}$ and derivatives of $\widetilde p_h$, multiplied by powers of $h/\sqrt{\gamma_h}$ of order at least $2$. By Lemma~\ref{l:classe-de-s-tilde}, all derivatives of \(\widetilde s_{\gamma_h,R}\) are uniformly bounded, and therefore all derivatives of $\gamma_h^{m/2}\widetilde s_{\gamma_h,R}$ are bounded by $\mathcal{O}_{\mathscr{S}^0}(\gamma_h^{m/2})$.

Moreover, for all multi-indices $\alpha$ and $\beta$ with \(|\alpha|\ge1\) and \(|\beta|\ge1\) one has
\[
\partial_\xi^\beta H(\sqrt{\gamma_h}\xi)
=
\gamma_h^{|\beta|/2}
(\partial^\beta H)(\sqrt{\gamma_h}\xi), \quad \text{and \quad}
\partial_x^\alpha\partial_\xi^\beta\!\big(hR_h(x,\sqrt{\gamma_h}\xi)\big)
=
h\,\gamma_h^{|\beta|/2}
(\partial_x^\alpha\partial_\xi^\beta R_h)(x,\sqrt{\gamma_h}\xi).
\]
On the support of \(\widetilde s_{\gamma_h,R}\) and of all its derivatives, the variable \(\sqrt{\gamma_h}\xi\) stays in a fixed compact set, since $a$ is compactly supported in the \(\xi\)-variable. Hence, on that support, every \(\xi\)-derivative of \(H(\sqrt{\gamma_h}\xi)\) is \(\mathcal{O}(\sqrt{\gamma_h})\), and every \(\xi\)-derivative of \(hR_h(x,\sqrt{\gamma_h}\xi)\) is \(\mathcal{O}(h\sqrt{\gamma_h})\).  The derivatives in $x$ of the perturbative
term do not give a factor $\sqrt{\gamma_h}$, but they still carry the factor
$h$, which is enough since $h/\sqrt{\gamma_h}\to 0$. This yields \eqref{e:rhR-estimate-final}.
By another use of Calder\'on--Vaillancourt,
\[
\left\|
\Opgamma(r_{h,R})
\right\|_{L^2\to L^2}
\le C_{a,R}h^2\gamma_h^{(m-1)/2}.
\]
Hence, one has
\begin{equation}\label{e:commutator-divided-lifted-flow}
\frac{i}{h}
\big[
\Opgamma\!\big(\gamma_h^{m/2}\widetilde s_{\gamma_h,R}\big),\widehat H_h
\big]
=
\Opgamma\!\left(
\gamma_h^{(m-1)/2}
\big\{\widetilde s_{\gamma_h,R},\widetilde p_h\big\}
\right)
+
\mathcal O_{L^2\to L^2}\!\left(h\gamma_h^{(m-1)/2}\right).
\end{equation}
We next split the Poisson bracket:
\[
\gamma_h^{(m-1)/2}
\big\{\widetilde s_{\gamma_h,R},\widetilde p_h\big\}
=
\gamma_h^{(m-1)/2}
\big\{\widetilde s_{\gamma_h,R},H(\sqrt{\gamma_h}\xi)\big\}
+
h\gamma_h^{(m-1)/2}
\big\{\widetilde s_{\gamma_h,R},R_h(x,\sqrt{\gamma_h}\xi)\big\}.
\]
For the principal term, since $H(\sqrt{\gamma_h}\xi)$ depends only on $\xi$, one has 
\begin{equation}\label{e:crochet-poisson-sgamma-H}
\gamma_h^{(m-1)/2}
\big\{\widetilde s_{\gamma_h,R},H(\sqrt{\gamma_h}\xi)\big\}
=
-\gamma_h^{m/2}\nabla H(\sqrt{\gamma_h}\xi)\cdot
\partial_x\big(\widetilde s_{\gamma_h,R}\big).
\end{equation}
Note that the rescaled quantization of the subprincipal term can be absorbed into the remainder $\mathcal O_{L^2\to L^2}\!\left(h\gamma_h^{(m-1)/2}\right)$ of \eqref{e:commutator-divided-lifted-flow} from the fact that $\widetilde s_{\gamma_h,R}$ is bounded in $\mathscr{S}^0$ and the fact that the derivatives of $R_h(x,\sqrt{\gamma_h}\xi)$ are defined on the support of $a$ which is compact, hence they are bounded. Combining \eqref{e:commutator-divided-lifted-flow} and \eqref{e:crochet-poisson-sgamma-H}, we obtain
\begin{equation}\label{e:final-commutator-estimate}
\frac{i}{h}
\big[
\Opgamma\!\big(\gamma_h^{m/2}\widetilde s_{\gamma_h,R}\big),\widehat H_h
\big]
=
-\gamma_h^{m/2}\Opgamma\!\big(\nabla H(\sqrt{\gamma_h}\xi)\cdot \partial_x(\widetilde s_{\gamma_h,R})\big)
+
\mathcal O_{L^2\to L^2}\!\left(h\gamma_h^{(m-1)/2}\right).
\end{equation}
Dividing \eqref{e:heisenberg-tilde-s} by $\gamma_h^{m/2}$, the estimate \eqref{e:estimation-terme-e:heisenberg-tilde-s} shows that the left-hand side of \eqref{e:heisenberg-tilde-s} is
\[
\mathcal O_R\!\left(\frac{1}{\tau_h}\right)=o_R(1),
\]
since $\tau_h\to+\infty$. Moreover, by \eqref{e:regime-pour-gamma}, one has $h/\sqrt{\gamma_h}\to0$. Hence, letting $h \to 0^+$ first and then $R \to \infty$ in \eqref{e:heisenberg-tilde-s}, we obtain
\begin{equation}\label{e:infinitesimal-invariance-flot}
\int_\R \psi(t)\,
\Big\langle
-\nabla H(\xi)\cdot \partial_x a(x,\xi,\eta),
\widetilde \nu_{t,\Lambda}^B
\Big\rangle h_\Lambda(t)\,\dd t
=0.
\end{equation}
Since \eqref{e:infinitesimal-invariance-flot} holds for all $\psi\in \mathcal{C}_c^\infty(\R)$, we infer the infinitesimal form of the invariance of $\widetilde \nu_{t,\Lambda}^B$ under the lifted flow $\widetilde\varphi_H^s$.
\end{proof}

By construction, the measure $\widetilde{\nu}_{t,\Lambda}^B$ is supported, in the $\xi$ variable, on $E_{\Lambda^\perp\setminus\{0\}}\cap\Omega_{E_1,E_2}$. Together with the invariance under the lifted geodesic flow and Lemma~\ref{l:anantharamanmacia}, this implies that, for every $k \notin \Lambda$, $\psi \in \mathcal{C}_c^\infty(\R)$ and $b \in \mathcal{C}_c^\infty(\R^2 \times \R)$,
\begin{center}
    $\displaystyle \int_\R \psi(t) \left\langle e^{2i\pi k\cdot x}b(\xi , \eta), \widetilde{\nu}_{t,\Lambda}^B \right\rangle h_\Lambda(t) \dd t= 0,$
\end{center}
which implies that 
\begin{equation}\label{e:tilde-nu-que-des-coef-en-Lambda}
    \displaystyle \widetilde{\nu}_{t,\Lambda}^B = \mathcal{I}_\Lambda(\widetilde{\nu}_{t,\Lambda}^B).
\end{equation}

\subsection{Effective dynamics for the two-microlocal lift}\label{ss:Effective-dynamics-for-the-two-microlocal-lift}
To formulate the extra propagation properties of the measure $\widetilde\nu_{t,\Lambda}^B$ in terms of an initial two-microlocal measure, we also introduce the auxiliary space--time distribution associated with the initial data. For $a\in \mathcal{C}_c^\infty( T^*\T^2\times\R)$, we define
\begin{equation}\label{e:WtildeB-def-donnee-initiale}
\langle \widetilde \nu_{0,h,R,\Lambda}^B,a\rangle
:=
\Big\langle
\Op\!\left( a\left(x,\xi,\eta_{\gamma_h}(\xi)\right)\chi\left(\frac{\nabla H(\xi)\cdot \frak{e}_\Lambda}{RL_\Lambda \sqrt{\gamma_h}}\right)\right)u_h^{(0)},u_h^{(0)}
\Big\rangle_{L^2(\T^2,L)}.
\end{equation}
Possibly after extracting a further subsequence and arguing exactly as in the proof of the compactness results for the two-microlocal distributions, the family $(\widetilde \nu_{0,h,R,\Lambda}^B)_{h\to0^+, R\to +\infty}$ converges to a nonnegative Radon measure $\widetilde \nu_{0,\Lambda}^B$ on $ T^*\T^2\times\R$. The two-microlocal measure satisfies additional transport/invariance constraints depending on the size of $\tau_h\sqrt{h}$. 
\begin{lemma}\label{l:three-regimes-2micro}
Let $X_\Lambda$ be the vector field 
\begin{equation}\label{e:champs-de-vecteur-X}
    X_\Lambda:=\eta\,\frac{\mathfrak e_\Lambda}{L_\Lambda}\partial_x -G_{\widehat B_0, \Lambda}(x,\xi)\partial_\eta, 
\end{equation}
where 
\begin{equation}\label{e:definition-G-B}
G_{\widehat B_0, \Lambda}(x,\xi):=\mathrm{Hess}(H)(\xi)\frac{\mathfrak e_\Lambda}{L_\Lambda}\cdot
\big(\nabla H(\xi)^\perp\,\widehat B_0+\partial_x\mathcal{I}_\Lambda(\mathscr{R})(x,\xi)\big),
\end{equation}
with the following convention $\begin{pmatrix} u_1 \\ u_2 \end{pmatrix}^\perp = \begin{pmatrix} -u_2 \\ u_1\end{pmatrix}$. We denote by $\phi^t_{X_\Lambda}$ the flow on $(x,\xi,\eta)$ generated by $X_\Lambda$, that is, 
\begin{equation}\label{e:definition-nouveau-flot}
    \dot x(t)=\eta(t)\frac{\mathfrak e_\Lambda}{L_\Lambda},\qquad \dot\xi(t)=0,\qquad
\dot\eta(t)=-G_{\widehat B_0, \Lambda}(x(t),\xi(t)).
\end{equation}

Then the following holds for $\widetilde\nu_{t,\Lambda}^B$:
\begin{enumerate}
\item If $\tau_h\ll h^{-1/2}$, then $\widetilde\nu_{t,\Lambda}^B$ is transported by the free flow
\begin{equation}\label{e:transport-tau_h-petit-sur-racinedeh}
\phi_0^t(x,\xi,\eta):=\Bigl(x+t\,\eta\,\tfrac{\mathfrak e_\Lambda}{L_\Lambda},\,\xi,\,\eta\Bigr),
\qquad\text{i.e.}\qquad
\widetilde\nu_{t,\Lambda}^B=(\phi_0^{t})_*\mathcal{I}_\Lambda(\widetilde\nu_{0,\Lambda}^B).
\end{equation}
\item If $\tau_h= h^{-1/2}$, then $\widetilde\nu_{t,\Lambda}^B$ is transported
by the flow $\phi^t_{X_\Lambda}$:
\begin{equation}\label{e:transport-tau_h-taille-racinedeh}
    \widetilde\nu_{t,\Lambda}^B = (\phi_{X_\Lambda}^t)_*\mathcal{I}_\Lambda(\widetilde\nu_{0,\Lambda}^B).
\end{equation}
\item If $\tau_h\gg h^{-1/2}$, then $\widetilde\nu_{t,\Lambda}^B$ is invariant under the flow $\phi_{X_\Lambda}^s$ generated by $X_\Lambda$, namely
\begin{equation}\label{e:extra-invariance-tau_h-grand-sur-racinedeh}
(\phi_{X_\Lambda}^s)_*\widetilde\nu_{t,\Lambda}^B=\widetilde\nu_{t,\Lambda}^B
\qquad \forall s\in\R.
\end{equation}
\end{enumerate}
\end{lemma}

This lemma shows that the evolution of $\widetilde{\nu}_{t,\Lambda}^B$ is explicit in the subcritical and critical regimes $\tau_h\ll h^{-1/2}$ and $\tau_h= h^{-1/2}$, since in both cases the measure is completely determined by its initial data through the relevant transport flow. As a consequence, its projection $\nu_{t,\Lambda}^B:=(\pi_{x,\xi})_*\widetilde{\nu}_{t,\Lambda}^B$ onto $T^*\T^2$ given by \eqref{e:projection-x-xi-du-lift-2micro} is also entirely determined by the initial two-microlocal measure $\widetilde{\nu}_{0,\Lambda}^B$.

The proof of the lemma follows the same strategy as in Lemma~\ref{l:noncompact-modes}: we test the invariance identity through a commutator argument. The Schrödinger equation \eqref{e:Schrodinger-PDE} gives the integration by parts in time, while the magnetic symbolic commutator formula identifies the first-order term, namely the sum of the classical Poisson bracket and the magnetic correction.

As a consequence of Lemma \ref{l:classe-de-s-tilde}, $\Opgamma(\widetilde s_{\gamma_h,R})$ is an admissible pseudodifferential test operator, and we can repeat the commutator argument used in Lemma~\ref{l:noncompact-modes}. We emphasize that the zoom scaling $\gamma_h$ has been fixed (depending on the time regime of $\tau_h$) in Definition~\ref{d:def-gamma}.

\begin{proof}
Fix $\psi\in \mathcal{C}_c^\infty(\R)$ and $a\in \mathcal{C}_c^\infty(T^*\T^2\times\R)$. By \eqref{e:tilde-nu-que-des-coef-en-Lambda}, it is enough to assume that $\mathcal I_\Lambda(a)=a$. As in the proof of Lemma~\ref{l:noncompact-modes}, we consider the commutator
\begin{equation}\label{e:commutateur-total}
\int_\R \psi(t)
\left\langle
\left[
\Opgamma\!\left(\gamma_h^{m/2}\widetilde s_{\gamma_h,R}\right),
\Op\!\left(H(\xi)+hR_h(x,\xi)\right)
\right]
u_h(t\tau_h),u_h(t\tau_h)
\right\rangle\dd t.
\end{equation}
The symbols entering this commutator are uniformly bounded in the expected classes: by Lemma \ref{l:classe-de-s-tilde}, the family $\gamma_h^{m/2}\widetilde s_{\gamma_h,R}$ is uniformly bounded in $\mathscr{S}^{-m}$, whereas $H(\sqrt{\gamma_h}\xi)+hR_h(x,\sqrt{\gamma_h}\xi)$ is uniformly bounded in $\mathscr{S}^m$. Since $(u_h)_{h\to0^+}$ solves \eqref{e:Schrodinger-PDE}, an integration by parts in time shows that \eqref{e:commutateur-total} is equal to
\begin{equation}\label{e:ipp-commutateur-total}
-i\frac{h\gamma_h^{m/2}}{\tau_h}
\int_\R \psi'(t)
\left\langle
\Opgamma(\widetilde s_{\gamma_h,R})u_h(t\tau_h),
u_h(t\tau_h)
\right\rangle\,\dd  t.
\end{equation}
On the other hand, using the rescaling formula \eqref{e:rescaling-semiclassique}, we rewrite the commutator in the $\frac{h}{\sqrt{\gamma_h}}$-quantization and split the commutator into its Hamiltonian and perturbative contributions:
\begin{equation}\label{e:split-commutator-3regimes-proof}
\left[
\Opgamma\!\left(\gamma_h^{m/2}\widetilde s_{\gamma_h,R}\right),
\Opgamma\!\left(H(\sqrt{\gamma_h}\xi)+hR_h(x,\sqrt{\gamma_h}\xi)\right)
\right]
=
\mathcal C_h^{\rm Ham}+h\,\mathcal C_h^{\rm pert},
\end{equation}
where
\[
\mathcal C_h^{\rm Ham}
:=
\left[
\Opgamma\!\left(\gamma_h^{m/2}\widetilde s_{\gamma_h,R}\right),
\Opgamma\!\left(H(\sqrt{\gamma_h}\xi)\right)
\right],
\]
and
\[
\mathcal C_h^{\rm pert}
:=
\left[
\Opgamma\!\left(\gamma_h^{m/2}\widetilde s_{\gamma_h,R}\right),
\Opgamma\!\left(R_h(x,\sqrt{\gamma_h}\xi)\right)
\right].
\]

Since $\gamma_h^{m/2}\widetilde s_{\gamma_h,R}\in \mathscr{S}^{-m}$ uniformly and $H(\sqrt{\gamma_h}\xi)+hR_h(x,\sqrt{\gamma_h}\xi)$ is uniformly bounded in $\mathscr{S}^m$, the magnetic symbolic calculus in the $\frac{h}{\sqrt{\gamma_h}}$-quantization gives, for every integer $N$ large enough,
\begin{align}
\mathcal C_h^{\rm Ham}
&=
\frac{h}{i\sqrt{\gamma_h}}
\Opgamma\!\left(
\left\{\gamma_h^{m/2}\widetilde s_{\gamma_h,R},H(\sqrt{\gamma_h}\xi)\right\}
\right)
\label{e:ham-expansion-1-3regimes-proof}
\\
&\quad
+\frac{\widehat B_0 h^2}{i\gamma_h}
\Opgamma\!\left(
\partial_{\xi_2}\!\left(\gamma_h^{m/2}\widetilde s_{\gamma_h,R}\right)
\partial_{\xi_1}\!\left(H(\sqrt{\gamma_h}\xi)\right)
-
\partial_{\xi_1}\!\left(\gamma_h^{m/2}\widetilde s_{\gamma_h,R}\right)
\partial_{\xi_2}\!\left(H(\sqrt{\gamma_h}\xi)\right)
\right)
\label{e:ham-expansion-2-3regimes-proof}
\\
&\quad
+\Opgamma(\mathfrak {H}_{N,h,R})+ \Opgamma(\widetilde \rho_{N,h,R}),
\label{e:ham-expansion-3-3regimes-proof}
\end{align}
where the symbol $\mathfrak{H}_{N,h,R}$ collects the explicit odd terms of order $\ge 3$ in the symbolic expansion up to order N. Here, $\widetilde \rho_{N,h,R}$ denotes the symbolic remainder exactly as in the proof of Lemma~\ref{l:noncompact-modes} and satisfies for each fixed \(R>0\),
\begin{equation}\label{e:reste-partie-Hamiltonienne-non-compacte}
\widetilde \rho_{N,h,R}
=
\mathcal{O}_{\mathscr{S}^0,R}\!\left(
\left(\frac{h}{\sqrt{\gamma_h}}\right)^N
\right)
=
\mathcal{O}_{\mathscr{S}^0,R}\!\left(h^{N/2}\right),
\end{equation}
because Definition~\ref{d:def-gamma} implies \(\gamma_h\ge h\).

\medskip
\noindent
\textbf{1. Hamiltonian contribution.}
We begin with the Poisson bracket. Since $H(\sqrt{\gamma_h}\xi)$ is independent of $x$ and since $\mathcal I_\Lambda(a)=a$, all the $x$-Fourier modes of $a$ are parallel to $\mathfrak e_\Lambda$. Computing the Poisson bracket explicitly, we obtain 
\begin{equation}\label{e:crochet-poisson-H-et-s-compact-modes}
\left\{
\gamma_h^{m/2}\widetilde s_{\gamma_h,R},H(\sqrt{\gamma_h}\xi)
\right\}=
-\gamma_h^{\frac m2+1}
\eta_{\gamma_h}(\sqrt{\gamma_h}\xi)\,
\frac{\mathfrak e_\Lambda}{L_\Lambda}\cdot
\partial_x(a)\!\left(x,\sqrt{\gamma_h}\xi,\eta_{\gamma_h}(\sqrt{\gamma_h}\xi)\right)
\chi_{R\sqrt{\gamma_h},\Lambda}(\sqrt{\gamma_h}\xi),
\end{equation}
where we recall that the notation $\chi_{R\sqrt{\gamma_h},\Lambda}$ was introduced in \eqref{e:chi-resonant}.

We next compute the magnetic correction in \eqref{e:ham-expansion-2-3regimes-proof}. This is where the Hessian of $H$ appears. Using the estimate \eqref{e:refined-xi-derivative-s-tilde} of Lemma \ref{l:classe-de-s-tilde}, one has 
\begin{align}\label{e:estimation-derivee-s-tilde-preuve}
    \partial_{\xi_j}\widetilde s_{\gamma_h,R} 
    &= \mathrm{Hess}(H)(\sqrt{\gamma_h}\xi)\frac{\frak{e}_\Lambda}{L_\Lambda} \cdot e_j \partial_\eta(a)\left(x,\sqrt{\gamma_h}\xi,\eta_{\gamma_h}(\sqrt{\gamma_h}\xi)\right)
    \chi_{R\sqrt{\gamma_h},\Lambda}(\sqrt{\gamma_h}\xi) \\
    &+ \mathcal{O}_{\mathscr{S}^0,R}(\sqrt{\gamma_h}) +\mathcal{O}_{\mathscr{S}^0,R}\left( \frac{1}{R}\right), \nonumber
\end{align}
where $(e_1,e_2)$ denotes the canonical basis of $\R^2$. Since $\partial_{\xi_j}(H(\sqrt{\gamma_h}\xi)) = \sqrt{\gamma_h}\,\partial_{\xi_j}H(\sqrt{\gamma_h}\xi)$, and since $\sqrt{\gamma_h}\xi$ remains in a fixed compact set on the support of $\widetilde s_{\gamma_h,R}$, we infer from
\eqref{e:estimation-derivee-s-tilde-preuve} that
\begin{align}
&\partial_{\xi_2}\!\left(\gamma_h^{m/2}\widetilde s_{\gamma_h,R}\right)\,
\partial_{\xi_1}\!\left(H(\sqrt{\gamma_h}\xi)\right)
-
\partial_{\xi_1}\!\left(\gamma_h^{m/2}\widetilde s_{\gamma_h,R}\right)\,
\partial_{\xi_2}\!\left(H(\sqrt{\gamma_h}\xi)\right)
\notag\\
&\qquad=
\gamma_h^{\frac{m+1}{2}}
\Bigg[
\mathrm{Hess}(H)(\sqrt{\gamma_h}\xi)\frac{\mathfrak e_\Lambda}{L_\Lambda}
\cdot \nabla H(\sqrt{\gamma_h}\xi)^\perp
\Bigg]
\partial_\eta a\!\left(x,\sqrt{\gamma_h}\xi,\eta_{\gamma_h}(\sqrt{\gamma_h}\xi)\right)
\chi_{R\sqrt{\gamma_h},\Lambda}(\sqrt{\gamma_h}\xi)
\notag\\
&\qquad\quad
+\mathcal O_{\mathscr{S}^0,R}\left(\gamma_h^{\frac{m+2}{2}}\right)
+\mathcal O_{\mathscr{S}^0,R}\!\left(\frac{\gamma_h^{\frac{m+1}{2}}}{R}\right).
\label{e:magnetic-principal-after-weight-3regimes-proof}
\end{align}
It remains to control the remainder symbol $\widetilde{\mathfrak H}_{N,h,R}$. By definition, it collects all odd terms of order $2p+1\ge 3$ in the symbolic expansion of the commutator together with the symbolic remainder. Since $H(\sqrt{\gamma_h}\xi')$ is independent of $x'$, every surviving monomial in
\[
\Omega_{\frac{h}{\sqrt{\gamma_h}}B}
(D_x,D_\xi,D_{x'},D_{\xi'})^{2p+1}
\Big[
\gamma_h^{m/2}\widetilde s_{\gamma_h,R}(x,\xi)\,
H(\sqrt{\gamma_h}\xi')
\Big]
\]
contains at least $2p+1$ derivatives in the $\xi'$ variable falling on $H(\sqrt{\gamma_h}\xi')$. Hence it gains a factor $\gamma_h^{(2p+1)/2}$. Since $\widetilde s_{\gamma_h,R}$ and all its derivatives are uniformly bounded in $\mathscr{S}^0$ for each fixed $R>1$, we infer that, for every $p\ge 1$,
\[
\left(\frac{h}{\sqrt{\gamma_h}}\right)^{2p+1}
\Omega_{\frac{h}{\sqrt{\gamma_h}}B}
(D_x,D_\xi,D_{x'},D_{\xi'})^{2p+1}
\Big[
\gamma_h^{m/2}\widetilde s_{\gamma_h,R}(x,\xi)\,
H(\sqrt{\gamma_h}\xi')
\Big]
=
\mathcal O_{\mathscr{S}^0,R}\!\left(h^{2p+1}\gamma_h^{m/2}\right),
\]
where the constant involved depends on finitely many seminorms of $H$, $a$, and $\chi$. In particular, using \eqref{e:reste-partie-Hamiltonienne-non-compacte}, all odd terms of order $\ge 3$ are absorbed into
\begin{equation}\label{e:ham-remainder-symbol-3regimes-proof}
\widetilde{\mathfrak H}_{N,h,R} := \mathfrak{H}_{N,h,R}+\widetilde \rho_{N,h,R}
=
\mathcal O_{\mathscr{S}^0,R}\!\left(h^3\gamma_h^{m/2}\right)+\mathcal{O}_{\mathscr{S}^0,R}\!\left(h^{N/2}\right).
\end{equation}
Again, taking $N$ large enough leads to 
\begin{equation}\label{e:ham-remainder-symbol-3regimes-proof-final}
\widetilde{\mathfrak H}_{N,h,R} = \mathcal O_{\mathscr{S}^0,R}\!\left(h^3\gamma_h^{m/2}\right).
\end{equation}
Hence, using the Calderón--Vaillancourt Theorem and combining \eqref{e:crochet-poisson-H-et-s-compact-modes}, \eqref{e:magnetic-principal-after-weight-3regimes-proof} and \eqref{e:ham-remainder-symbol-3regimes-proof-final} all together in the Hamiltonian commutator, one gets
\begin{align}\label{e:commutateur-hamiltonien-H-et-s-tilde}
&\mathcal C_h^{\rm Ham}
=
-\gamma_h^{\frac{m+1}{2}}\frac{h}{i}
\Opgamma\!\left(
 \eta_{\gamma_h}(\sqrt{\gamma_h}\xi)\,\frac{\mathfrak e_\Lambda}{L_\Lambda}\cdot \partial_x(a)\!\left(x,\sqrt{\gamma_h}\xi,\eta_{\gamma_h}(\sqrt{\gamma_h}\xi)\right) \chi_{R\sqrt{\gamma_h},\Lambda}(\sqrt{\gamma_h}\xi) 
\right)
\\
&
+\gamma_h^{\frac{m-1}{2}}\!\frac{\widehat B_0 h^2}{i}
\Opgamma\!\left(
\!\Bigg[ \mathrm{Hess}(H)(\sqrt{\gamma_h}\xi)\frac{\mathfrak e_\Lambda}{L_\Lambda} \!\cdot\! \nabla H(\sqrt{\gamma_h}\xi)^\perp \!\Bigg] \partial_\eta a\!\left(x,\!\sqrt{\gamma_h}\xi,\eta_{\gamma_h}(\sqrt{\gamma_h}\xi)\right) \chi_{R\sqrt{\gamma_h},\Lambda}(\sqrt{\gamma_h}\xi) 
\right)
\nonumber
\\
&+ \mathcal O_{L^2 \to L^2}\left(h^2\gamma_h^{\frac{m}{2}}\right)
+\mathcal O_{L^2 \to L^2}\!\left(h^2\frac{\gamma_h^{\frac{m-1}{2}}}{R}\right) +\,\mathcal O_{L^2 \to L^2,R}\!\left(h^3\gamma_h^{m/2}\right).
\nonumber
\end{align}
\textbf{2. Perturbative contribution.}
We now argue similarly to treat the perturbative contribution, which is given by
\begin{align}
\mathcal C_h^{\rm pert}
&=
\frac{h}{i\sqrt{\gamma_h}}
\Opgamma\!\left(
\left\{\gamma_h^{m/2}\widetilde s_{\gamma_h,R},R_h(x,\sqrt{\gamma_h}\xi)\right\}
\right)
\label{e:pert-expansion-1-3regimes-proof}
\\
&
+\frac{\widehat B_0 h^2}{i\gamma_h}
\Opgamma\!\left(
\partial_{\xi_2}\!\left(\gamma_h^{m/2}\widetilde s_{\gamma_h,R}\right)
\partial_{\xi_1}\!\left(R_h(x,\sqrt{\gamma_h}\xi)\right)
-
\partial_{\xi_1}\!\left(\gamma_h^{m/2}\widetilde s_{\gamma_h,R}\right)
\partial_{\xi_2}\!\left(R_h(x,\sqrt{\gamma_h}\xi)\right)
\right)
\label{e:pert-expansion-2-3regimes-proof}
\\
&
+\Opgamma(\widetilde{\mathfrak R}_{N,h,R}),
\label{e:pert-expansion-3-3regimes-proof}
\end{align}
where $\widetilde{\mathfrak R}_{N,h,R}$ collects the odd terms of order at least $3$ together with the symbolic remainder involved in the composition formula \eqref{e:composition-rule}.
We first discuss the Poisson bracket in \eqref{e:pert-expansion-1-3regimes-proof} which will give the main contribution. One has
\[
\left\{
\widetilde s_{\gamma_h,R},R_h(x,\sqrt{\gamma_h}\xi)
\right\}
=
\partial_\xi\widetilde s_{\gamma_h,R}\cdot \partial_xR_h(x,\sqrt{\gamma_h}\xi)
-
\partial_x\widetilde s_{\gamma_h,R}\cdot
\partial_\xi\!\left(R_h(x,\sqrt{\gamma_h}\xi)\right).
\]
Using \eqref{e:estimation-derivee-s-tilde} and $\partial_\xi\!\left(R_h(x,\sqrt{\gamma_h}\xi)\right) = \sqrt{\gamma_h}\,\partial_\xi R_h(x,\sqrt{\gamma_h}\xi),$
we obtain
\begin{align}
&\left\{
 \gamma_h^\frac{m}{2}\widetilde s_{\gamma_h,R},  R_h(x,\sqrt{\gamma_h}\xi)
\right\} \nonumber\\
&\hspace{1cm}= \gamma_h^\frac{m}{2}\Bigg[
\mathrm{Hess}(H)(\sqrt{\gamma_h}\xi)\frac{\mathfrak e_\Lambda}{L_\Lambda}
\cdot \partial_xR_h(x,\sqrt{\gamma_h}\xi)
\Bigg]
\partial_\eta a\!\left(x,\sqrt{\gamma_h}\xi,\eta_{\gamma_h}(\sqrt{\gamma_h}\xi)\right)\chi_{R\sqrt{\gamma_h},\Lambda}(\sqrt{\gamma_h}\xi)
\notag\\
&\hspace{1cm}+\mathcal O_{\mathscr{S}^0}(\gamma_h^\frac{m+1}{2}) +\mathcal O_{\mathscr{S}^0}\!\left(\frac{\gamma_h^\frac{m}{2}}{R}\right).
\label{e:crochet-de-poisson-R-et-s-tilde}
\end{align}
We now turn to the magnetic correction in \eqref{e:pert-expansion-2-3regimes-proof}. Since $R_h$ is uniformly bounded in $\mathscr{S}^{m}$ and since, on the support of $\widetilde s_{\gamma_h,R}$ and of its derivatives, the variable $\sqrt{\gamma_h}\xi$ remains in a fixed compact set, one has
\[
\partial_{\xi_j}\!\left(R_h(x,\sqrt{\gamma_h}\xi)\right) = \mathcal O_{\mathscr S^0,R}(\sqrt{\gamma_h}).
\]
Moreover, Lemma~\ref{l:classe-de-s-tilde} implies that
\[
\partial_{\xi_j}\!\left(\gamma_h^{m/2}\widetilde s_{\gamma_h,R}\right) = \gamma_h^{m/2}\mathcal O_{\mathscr S^0,R}(1).
\]
Hence, one gets
\begin{equation}\label{e:estimation-term-magnetique-s-et-R}
\partial_{\xi_2}\!\left(\gamma_h^{m/2}\widetilde s_{\gamma_h,R}\right)
\partial_{\xi_1}\!\left(R_h(x,\sqrt{\gamma_h}\xi)\right)
-
\partial_{\xi_1}\!\left(\gamma_h^{m/2}\widetilde s_{\gamma_h,R}\right)
\partial_{\xi_2}\!\left(R_h(x,\sqrt{\gamma_h}\xi)\right)
=
\mathcal O_{\mathscr{S}^0,R}\!\left(\gamma_h^{\frac{m+1}{2}}\right).
\end{equation}

Finally, the higher odd terms and the symbolic remainder collected by $\widetilde{\mathfrak R}_{N,h,R}$ in the expansion of $\mathcal C_h^{\rm pert}$ are handled exactly as for $\widetilde{\mathcal R}_{N,h,R}$ in \eqref{e:terme-de-reste-s-et-R} and yield
\begin{equation}\label{e:pert-remainder-operator-3regimes-proof}
\widetilde{\mathfrak R}_{N,h,R}
=
\mathcal O_{\mathscr{S}^0,R}\!\left(
h^3\gamma_h^{\frac{m-3}{2}}
\right).
\end{equation}
Finally, combining \eqref{e:crochet-de-poisson-R-et-s-tilde}, \eqref{e:estimation-term-magnetique-s-et-R}, \eqref{e:pert-remainder-operator-3regimes-proof} and using the Calderón--Vaillancourt theorem in the perturbative commutator, we obtain 

\begin{align}\label{e:commutateur-perturbation-R-et-s-tilde}
\mathcal C_h^{\rm pert}
&=
\!\frac{h\gamma_h^{\frac{m-1}{2}}}{i}
\!\Opgamma\!\left(
\!\mathrm{Hess}(H)(\sqrt{\gamma_h}\xi)\frac{\mathfrak e_\Lambda}{L_\Lambda}
\!\cdot \!\partial_xR_h(x,\!\sqrt{\gamma_h}\xi)\,
\partial_\eta a\!\left(x,\!\sqrt{\gamma_h}\xi,\eta_{\gamma_h}(\!\sqrt{\gamma_h}\xi)\right)
\chi_{R\sqrt{\gamma_h},\Lambda}(\sqrt{\gamma_h}\xi)
\!\right)\\
&\quad
+\mathcal O_{L^2\to L^2}\!\left(h\gamma_h^{m/2}\right)
+\mathcal O_{L^2\to L^2}\!\left(h\frac{\gamma_h^{\frac{m-1}{2}}}{R}\right)
+\mathcal O_{L^2\to L^2}\!\left(h^2\gamma_h^{\frac{m-1}{2}}\right)
+\mathcal O_{L^2\to L^2,R}\!\left(h^3\gamma_h^{\frac{m-3}{2}}\right).
\notag
\end{align}

\medskip
\noindent
\textbf{3. Renormalized identity.}
Combining \eqref{e:commutateur-hamiltonien-H-et-s-tilde},
\eqref{e:commutateur-perturbation-R-et-s-tilde} together in \eqref{e:commutateur-total}, linking it to \eqref{e:ipp-commutateur-total}, and dividing by $\displaystyle \frac{h\gamma_h^\frac{m+1}{2}}{i}$, we obtain 
\begin{align}
&\frac{1}{\tau_h\sqrt{\gamma_h}}
\int_\R \psi'(t)
\left\langle
\Opgamma(\widetilde s_{\gamma_h,R})u_h(t\tau_h),
u_h(t\tau_h)
\right\rangle\,\dd  t
\notag\\
&=
\!\int_\R \!\psi(t)
\Bigg\langle
\!\Opgamma\!\Bigg(\!\!
-\eta_{\gamma_h}(\sqrt{\gamma_h}\xi)\,
\frac{\mathfrak e_\Lambda}{L_\Lambda}\!\cdot\!
\partial_x a\!\left(x,\sqrt{\gamma_h}\xi,\eta_{\gamma_h}(\sqrt{\gamma_h}\xi)\right)
\chi_{R\sqrt{\gamma_h},\Lambda}(\sqrt{\gamma_h}\xi)
\!\Bigg)
u_h(t\tau_h),u_h(t\tau_h)
\!\Bigg\rangle\dd  t
\notag\\
&
+\frac{h}{\gamma_h}
\int_\R \psi(t)
\Bigg\langle
\Opgamma\!\Bigg( \widetilde{G}_h(x,\sqrt{\gamma_h}\xi)
\partial_\eta a\!\left(x,\sqrt{\gamma_h}\xi,\eta_{\gamma_h}(\sqrt{\gamma_h}\xi)\right)
\chi_{R\sqrt{\gamma_h},\Lambda}(\sqrt{\gamma_h}\xi)
\Bigg)
u_h(t\tau_h),u_h(t\tau_h)
\Bigg\rangle\dd  t
\notag\\
&
+\mathcal O\!\left(\frac{h}{\sqrt{\gamma_h}}\right)
+\mathcal O\!\left(\frac{h}{\gamma_hR}\right)
+\mathcal O\!\left(\frac{h^2}{\gamma_h}\right)
+\mathcal{O}_R\!\left(\frac{h}{\sqrt{\gamma_h}} \right)
+\mathcal{O}_R\!\left(\frac{h^3}{\gamma_h^2} \right),
\label{e:identite-finale-avant-limite-compact-mode}
\end{align}
where 
\begin{equation}
    \widetilde{G}_h(x,\sqrt{\gamma_h}\xi)=\mathrm{Hess}(H)(\sqrt{\gamma_h}\xi)\frac{\mathfrak e_\Lambda}{L_\Lambda}
\cdot \Bigg[
\widehat B_0\,
\nabla H(\sqrt{\gamma_h}\xi)^\perp
+ \partial_xR_h(x,\sqrt{\gamma_h}\xi)
\Bigg].
\end{equation}
From the definition \eqref{e:definition-Rh} of $R_h$ and using that $\sqrt{\gamma_h}\xi$ remains in a compact set on the support of $a$, one can write
\begin{equation}
    \widetilde{G}_h(x,\sqrt{\gamma_h}\xi)=\mathrm{Hess}(H)(\sqrt{\gamma_h}\xi)\frac{\mathfrak e_\Lambda}{L_\Lambda}
\cdot \Bigg[
\widehat B_0\,
\nabla H(\sqrt{\gamma_h}\xi)^\perp
+ \partial_x (\mathscr{R})(x,\sqrt{\gamma_h}\xi) + \mathcal{O}_{\mathscr{S}^0}(h)
\Bigg].
\end{equation}
Hence, we finally write on the support of $a(x,\sqrt{\gamma_h}\xi)$
\begin{equation}\label{e:introduction-G-B}
    \widetilde{G}_h(x,\sqrt{\gamma_h}\xi) = G_{\widehat B_0, \Lambda}(x,\sqrt{\gamma_h}\xi) + \mathcal{O}_{\mathscr{S}^0}(h), 
\end{equation}
where
\[
G_{\widehat B_0, \Lambda}(x,\sqrt{\gamma_h}\xi) = \mathrm{Hess}(H)(\sqrt{\gamma_h}\xi)\frac{\mathfrak e_\Lambda}{L_\Lambda}
\cdot \Bigg[
\widehat B_0\,
\nabla H(\sqrt{\gamma_h}\xi)^\perp
+ \partial_x (\mathscr{R})(x,\sqrt{\gamma_h}\xi)
\Bigg].
\]
We may now let $h\to0^+$ and then $R\to+\infty$, using the fact that $a = \mathcal{I}_\Lambda(a)$, Definition \ref{d:def-gamma} of $(\gamma_h)_{h\to 0^+}$ and the definition of the two-microlocal measure $\widetilde\nu_{t,\Lambda}^B$.

\smallskip
\noindent
\emph{Case 1: $\tau_h\ll h^{-1/2}$.}
In this regime, $\gamma_h=\tau_h^{-2}$, so that
\[
\frac1{\tau_h\sqrt{\gamma_h}}=1,
\qquad
\frac{h}{\gamma_h}=h\tau_h^2\longrightarrow 0.
\]
Passing to the limit in \eqref{e:identite-finale-avant-limite-compact-mode},
we obtain
\[
\int_\R \psi'(t)\left\langle a , \widetilde\nu_{t,\Lambda}^B\right\rangle \,h_\Lambda(t)\dd  t = 
\int_\R \psi(t)\left\langle 
-\eta\,\frac{\mathfrak e_\Lambda}{L_\Lambda}\cdot\partial_x (a) ,
\widetilde\nu_{t,\Lambda}^B\right\rangle \,h_\Lambda(t)\dd  t .
\]
This is the weak formulation of \eqref{e:transport-tau_h-petit-sur-racinedeh},
hence
\[
\widetilde\nu_{t,\Lambda}^B
=
(\phi_0^{t})_*\mathcal{I}_\Lambda(\widetilde\nu_{0,\Lambda}^B).
\]

\smallskip
\noindent
\emph{Case 2: $\tau_h= h^{-1/2}$.}
Here $\gamma_h=h$, so that
\[
\frac{h}{\gamma_h}=1, \qquad \tau_h\sqrt{\gamma_h} = \tau_h\sqrt{h} \to 1.
\]
Passing to the limit in \eqref{e:identite-finale-avant-limite-compact-mode}, we obtain
\[
\int_\R \psi'(t)\left\langle a , \widetilde\nu_{t,\Lambda}^B\right\rangle h_\Lambda(t)\dd  t
= 
\int_\R \psi(t)
\left\langle 
-\eta\,\frac{\mathfrak e_\Lambda}{L_\Lambda}\cdot\partial_x a
+
G_{\widehat B_0, \Lambda}(x,\xi)\,\partial_\eta a,\widetilde\nu_{t,\Lambda}^B
\right\rangle\,h_\Lambda(t) \dd  t.
\]
This is exactly the weak formulation of \eqref{e:transport-tau_h-taille-racinedeh}, namely
\[
\widetilde\nu_{t,\Lambda}^B=(\phi_{X_\Lambda}^t)_*\mathcal{I}_\Lambda(\widetilde\nu_{0,\Lambda}^B).
\]

\smallskip
\noindent
\emph{Case 3: $\tau_h\gg h^{-1/2}$.}
Again $\gamma_h=h$, but now one has 
\[
\frac1{\tau_h\sqrt{\gamma_h}}
=
\frac1{\tau_h\sqrt h}
\longrightarrow 0,
\qquad
\frac{h}{\gamma_h}=1.
\]
Passing to the limit in \eqref{e:identite-finale-avant-limite-compact-mode}, the time derivative disappears and we are left with
\[
\int_\R \psi(t)
\left\langle 
-\eta\,\frac{\mathfrak e_\Lambda}{L_\Lambda}\cdot\partial_x a
+
G_{\widehat B_0, \Lambda}(x,\xi)\,\partial_\eta a,\widetilde\nu_{t,\Lambda}^B
\right\rangle\,h_\Lambda(t) \dd t =0.
\]
Since $\psi$ is arbitrary, this implies that for a.e.\ $t$,
\[
\left\langle
-\eta\,\frac{\mathfrak e_\Lambda}{L_\Lambda}\cdot\partial_x a
+
G_{\widehat B_0, \Lambda}(x,\xi)\,\partial_\eta a,\widetilde\nu_{t,\Lambda}^B
\right\rangle
=0
\qquad
\forall a\in \mathcal{C}_c^\infty(T^*\T^2\times\R).
\]
This is precisely the infinitesimal form of the invariance
\[
(\phi_{X_\Lambda}^s)_*\widetilde\nu_{t,\Lambda}^B=\widetilde\nu_{t,\Lambda}^B,
\qquad \forall s\in\R.
\]
This proves \eqref{e:extra-invariance-tau_h-grand-sur-racinedeh}.
\end{proof}

\section{Regularity of configuration-space quantum limits}\label{s:x-regularity}

The purpose of this section is to prove the configuration-space
equidistribution part of Theorem~\ref{t:intro-main-structure}. Combined with
Lemma~\ref{l:xi-marginal}, this will also yield
item~\textup{(i)} of that theorem. We shall further recover
Theorem~\ref{t:intro-simplified-equidistribution} and its specialization to
the magnetic Laplacian by verifying the geometric non-vanishing condition
along the periodic directions of the Hamiltonian flow.

Throughout this section, we assume that \(H\) and \(R_h\) satisfy the
standing assumptions of Section~\ref{ss:magnetic-operator}, that the initial
data satisfy the spectral localization condition
\eqref{e:spectral-projection}, and that
\[
\tau_h\gg h^{-1/2}.
\]
Let \(W^B\) be a time-dependent magnetic semiclassical measure associated
with the rescaled evolution \((u_h(t\tau_h))_{h\to0^+}\). More precisely,
\(W^B\) is a subsequential limit of the magnetic Wigner distributions
\(W_h^B(\tau_h)\) defined in \eqref{e:Wigner-magnetic-distrib}; see also
Section~\ref{s:semiclassicalmeasure}. By
Lemma~\ref{l:Lifting-QL}, it admits the time disintegration
\[
W^B(\dd t,\dd x,\dd\xi)
=
\nu_t^B(\dd x,\dd\xi)\otimes\dd t,
\]
where, for almost every \(t\in\R\), \(\nu_t^B\) is a probability measure
supported in $\T^2\times\Omega_{E_1,E_2}$. We denote its momentum marginal, introduced in Lemma~\ref{l:xi-marginal}, by
\[
\lambda_t:=(\pi_\xi)_*\nu_t^B.
\]
Our goal is to prove the product decomposition
\eqref{e:intro-product-form}, namely
\[
\nu_t^B(\dd x,\dd\xi)
=
\dd x\otimes\lambda_t(\dd\xi)
\qquad
\text{for almost every }t\in\R.
\]
Equivalently, we shall prove that all nonzero Fourier coefficients of
\(\nu_t^B\) in the configuration variable vanish.
\begin{remark}
The analysis of the compact two-microlocal measures carried out in
Subsection~\ref{ss:Effective-dynamics-for-the-two-microlocal-lift}, and
summarized in Lemma~\ref{l:three-regimes-2micro}, explains why
\(h^{-1/2}\) is the relevant threshold. Indeed, the first two assertions of
Lemma~\ref{l:three-regimes-2micro} show that, in the regimes
\[
\tau_h\ll h^{-1/2}
\qquad\text{and}\qquad
\tau_h=h^{-1/2},
\]
the compact two-microlocal measures are transported from their initial
two-microlocal data. Consequently, there is in general no mechanism forcing
their nonzero Fourier modes in the configuration variable to vanish. By
contrast, when
\[
\tau_h\gg h^{-1/2},
\]
the third assertion of Lemma~\ref{l:three-regimes-2micro} shows that these measures become invariant under an effective flow. This is the
supercritical regime considered throughout the present section.
\end{remark}

The strategy is the following. Since \(\tau_h\to+\infty\),
Lemma~\ref{l:Dynamical-regimes} implies that, for almost every \(t\in\R\),
the measure \(\nu_t^B\) is invariant under the Hamiltonian flow
\[
\varphi_H^s(x,\xi)
=
(x+s\nabla H(\xi),\xi).
\]
Using the resonant decomposition established in
Section~\ref{s:decomposition}, and more precisely
\eqref{e:nu_t^B=decomposition-direction-du-flot}, this invariance already
implies spatial equidistribution on the non-resonant region
\(\Omega_{\mathrm{nr}}\). Moreover,
Lemma~\ref{l:noncompact-modes} shows that the non-compact resonant
contributions carry only the zero Fourier mode in the configuration
variable. Consequently, the only possible obstruction to spatial
equidistribution is carried by the compact two-microlocal measures
introduced in Lemma~\ref{l:2micro}, associated with the resonant sets
\[
E_{\Lambda^\perp\setminus\{0\}},
\qquad
\Lambda\in\mathcal L_1.
\]

We first make this reduction precise in
Subsection~\ref{ss:reduction-compact-2micro}. We then use, in
Subsection~\ref{ss:supercritical-equidistribution}, the invariance of the
compact two-microlocal measures under the effective vector field
\(X_\Lambda\) introduced in Lemma~\ref{l:three-regimes-2micro}. Under the
non-vanishing condition
\[
G_{\widehat B_0,\Lambda}(x,\xi)\neq0,
\]
the corresponding dynamics escapes in the two-microlocal variable
\(\eta\), which rules out nonzero finite invariant compact measures. This
yields the spatial equidistribution result of
Theorem~\ref{t:x-equidistribution}. We conclude the section by applying
this criterion to the magnetic Laplacian.

\subsection{Reduction to the compact two-microlocal contribution}
\label{ss:reduction-compact-2micro}

We first prove that all nonzero Fourier modes in the configuration variable are carried by the compact two-microlocal contributions. Throughout this subsection, we only need to assume that \(\tau_h\to+\infty\). By Lemma~\ref{l:Dynamical-regimes}, for a.e. $t\in\R$, the measure $\nu_t^B$ is invariant under the Hamiltonian flow $\varphi_H^s$. We also work under the spectral localization assumption \eqref{e:spectral-projection}, and we assume that the corresponding energy window $[E_1,E_2]$ contains no critical value of $H$, see \eqref{e:energy-window}. Under these assumptions, we recall from \eqref{e:nu_t^B=decomposition-direction-du-flot} that, for a.e. \(t\in\R\), one has the following identity of finite measures
\begin{equation}\label{e:nu-decomposition-sec5}
\nu_t^B
=
\nu_t^B\big|_{\T^2\times\Omega_{\mathrm{nr}}}
+
\sum_{\Lambda\in\mathcal L_1}
\mathcal I_\Lambda(\nu_t^B)
\big|_{\T^2\times E_{\Lambda^\perp\setminus\{0\}}}.
\end{equation}
On the non-resonant region $\Omega_\mathrm{nr}$, the flow $x\mapsto x+s\nabla H(\xi)$ is dense in $\T^2$ for every $\xi\in\Omega_{\mathrm{nr}}$. Hence the invariance of $\nu_t^B$ forces the conditional measure in the $x$-variable to be the Haar measure. Equivalently,
\begin{equation}\label{e:nonresonant-product-sec5}
\nu_t^B\big|_{\T^2\times\Omega_{\mathrm{nr}}}
=
\dd x\otimes \lambda_t\big|_{\Omega_{\mathrm{nr}}},
\end{equation}
where the measure $\lambda_t$ is defined as $ \lambda_t:=(\pi_\xi)_*\nu_t^B$. Therefore, one gets 
\begin{equation}\label{e:nu-decomposition-product-sec5}
\nu_t^B
=
\dd x\otimes \lambda_t\big|_{\Omega_{\mathrm{nr}}}
+
\sum_{\Lambda\in\mathcal L_1}
\mathcal I_\Lambda(\nu_t^B)
\big|_{\T^2\times E_{\Lambda^\perp\setminus\{0\}}}.
\end{equation}
Thus the non-resonant part is already spatially equidistributed. It remains to analyze the resonant contributions.

\begin{lemma}\label{l:reduction-compact-sec5}
For a.e. $t\in\R$, the measure $\nu_t^B$ can be written in the form
\begin{equation}\label{e:global-reduction-sec5}
\nu_t^B
=
\dd x\otimes
\left(
\lambda_t\big|_{\Omega_{\mathrm{nr}}}
+
\sum_{\Lambda\in\mathcal L_1}
h^\Lambda(t)\beta_t^{\Lambda}
\right)
+
\sum_{\Lambda\in\mathcal L_1}
h_\Lambda(t)\,
(\pi_{x,\xi})_*\widetilde\nu_{t,\Lambda}^B
\big|_{\T^2\times E_{\Lambda^\perp\setminus\{0\}}},
\end{equation}
where, for each $\Lambda\in\mathcal L_1$, $\beta_t^{\Lambda}$ is a finite nonnegative measure supported on $E_{\Lambda^\perp\setminus\{0\}}$ and where $\widetilde\nu_{t,\Lambda}^B$ denotes the compact two-microlocal measure appearing in the time disintegration \eqref{e:2micro-disint}. 
\end{lemma}
The preceding lemma shows that all terms in \eqref{e:global-reduction-sec5} are already independent of $x$, except possibly the weighted compact two-microlocal contributions
\[
h_\Lambda(t)(\pi_{x,\xi})_*\widetilde\nu_{t,\Lambda}^B.
\]
Thus the question of spatial equidistribution is reduced to the analysis of these terms.

\begin{proof}
Fix $\Lambda\in\mathcal L_1$. Using the microlocal splitting of Section~\ref{s:periodic-orbits} and the time disintegrations of Lemma~\ref{l:RN-split}, one has
\[
\nu_t^B\,\dd t
=
\nu_{t,\Lambda}^B\,h_\Lambda(t)\dd t
+
\nu_t^{B,\Lambda}\,h^\Lambda(t)\dd t.
\]
Therefore, for a.e. $t$, as finite measures on $T^*\T^2$, one has
\[
\nu_t^B
=
h_\Lambda(t)\nu_{t,\Lambda}^B
+
h^\Lambda(t)\nu_t^{B,\Lambda}.
\]
Applying $\mathcal I_\Lambda$ and restricting to $\T^2\times E_{\Lambda^\perp\setminus\{0\}}$, we obtain
\begin{equation}\label{e:I-Lambda-split-sec5}
\mathcal I_\Lambda(\nu_t^B)
\big|_{\T^2\times E_{\Lambda^\perp\setminus\{0\}}}
=
h^\Lambda(t)\,
\mathcal I_\Lambda(\nu_t^{B,\Lambda})
\big|_{\T^2\times E_{\Lambda^\perp\setminus\{0\}}}+
h_\Lambda(t)\,
\mathcal I_\Lambda(\nu_{t,\Lambda}^B)
\big|_{\T^2\times E_{\Lambda^\perp\setminus\{0\}}}.
\end{equation}
We first consider the non-compact term. By Lemma~\ref{l:noncompact-modes}, all nonzero $\Lambda$-Fourier modes of $\nu_t^{B,\Lambda}$ vanish on $E_{\Lambda^\perp\setminus\{0\}}$. Hence, one has 
\[
\mathcal I_\Lambda(\nu_t^{B,\Lambda})
\big|_{\T^2\times E_{\Lambda^\perp\setminus\{0\}}}
=
\widehat\nu_t^{B,\Lambda}(0,\dd\xi)
\big|_{\T^2\times E_{\Lambda^\perp\setminus\{0\}}}.
\]
Equivalently, there exists a finite nonnegative measure $\beta_t^{\Lambda}$ on $E_{\Lambda^\perp\setminus\{0\}}$ such that
\begin{equation}\label{e:noncompact-equi-sec5}
\mathcal I_\Lambda(\nu_t^{B,\Lambda})
\big|_{\T^2\times E_{\Lambda^\perp\setminus\{0\}}}
=
\dd x\otimes \beta_t^{\Lambda}.
\end{equation}
We now consider the compact term. By Lemma~\ref{l:2micro}, the compact component admits the two-microlocal lift $\widetilde\nu_{t,\Lambda}^B$, and by \eqref{e:projection-x-xi-du-lift-2micro} its $(x,\xi)$-marginal is precisely the compact resonant component:
\[
(\pi_{x,\xi})_*\widetilde\nu_{t,\Lambda}^B
=
\nu_{t,\Lambda}^B
\qquad
\text{for }h_\Lambda(t)\dd t\text{-a.e. }t.
\]
Moreover, by \eqref{e:tilde-nu-que-des-coef-en-Lambda}, this measure has only
$\Lambda$-Fourier modes in the $x$-variable. Therefore, one has
\begin{equation}\label{e:compact-contribution-sec5}
h_\Lambda(t)\,
\mathcal I_\Lambda(\nu_{t,\Lambda}^B)
=
h_\Lambda(t)\,
(\pi_{x,\xi})_*\widetilde\nu_{t,\Lambda}^B.
\end{equation}
Combining \eqref{e:I-Lambda-split-sec5}, \eqref{e:noncompact-equi-sec5} and \eqref{e:compact-contribution-sec5}, we get
\begin{equation}\label{e:resonant-reduction-sec5}
\mathcal I_\Lambda(\nu_t^B)
\big|_{\T^2\times E_{\Lambda^\perp\setminus\{0\}}}
=
\dd x\otimes h^\Lambda(t)\beta_t^{\Lambda}
+
h_\Lambda(t)
(\pi_{x,\xi})_*\widetilde\nu_{t,\Lambda}^B
\big|_{\T^2\times E_{\Lambda^\perp\setminus\{0\}}}.
\end{equation}
Inserting this identity into \eqref{e:nu-decomposition-product-sec5} and summing over $\Lambda\in\mathcal L_1$ gives \eqref{e:global-reduction-sec5}.
\end{proof}

\subsection{A geometric criterion for equidistribution}\label{ss:supercritical-equidistribution}

We now consider the regime
\[
\tau_h\gg h^{-1/2}.
\]
By Lemma~\ref{l:reduction-compact-sec5}, all nonzero Fourier modes of
\(\nu_t^B\) are carried by the weighted compact two-microlocal measures $h_\Lambda(t)\widetilde\nu_{t,\Lambda}^B$.
Thus, to prove spatial equidistribution, it is enough to prove that these
weighted measures vanish.
In this regime, the third assertion of Lemma~\ref{l:three-regimes-2micro}
implies that, for every $\Lambda\in\mathcal L_1$, the compact two-microlocal
measure $\widetilde\nu_{t,\Lambda}^B$ is invariant, for
$h_\Lambda(t)\dd t$-a.e. $t$, under the flow generated by
\begin{equation}\label{e:vector-field-X-Lambda-sec5}
X_\Lambda
=
\eta\,\frac{\mathfrak e_\Lambda}{L_\Lambda}\cdot \partial_x
-
G_{\widehat B_0,\Lambda}(x,\xi)\partial_\eta,
\end{equation}
where $G_{\widehat B_0,\Lambda}$ is defined in Lemma~\ref{l:three-regimes-2micro}.
We now give a criterion which prevents the existence of nonzero finite invariant
measures for this flow.
Recall that $E_1$ and $E_2$ were fixed so that
\[
\Omega_{E_1,E_2}:=\{\xi\in\R^2:\ E_1\le H(\xi)\le E_2\}
\]
contains no critical point of $H$. 

\begin{theorem}[Equidistribution above the two-microlocal scale]\label{t:x-equidistribution}
Let \((u_h)_{h\to0^+}\) be a family of solutions to \eqref{e:Schrodinger-PDE}. Assume that \(H\) and \(R_h\) satisfy the standing assumptions of Section~\ref{ss:magnetic-operator}, and that \((u_h^{(0)})_{h\to0^+}\) satisfies the spectral localization condition \eqref{e:spectral-projection}. Assume that
\[
\tau_h\gg h^{-1/2},
\]
and that, for every \((\Lambda,x,\xi)\in \mathcal L_1 \times \mathbb T^2\times \big( E_{\Lambda^\perp\setminus\{0\}}\cap\Omega_{E_1,E_2} \big) \), one has 
\begin{equation}\label{e:G-sign-sec5}
G_{\widehat B_0,\Lambda}(x,\xi)\neq 0.
\end{equation}
Then every time-dependent magnetic semiclassical measure is independent of the
configuration variable. More precisely, for a.e. $t\in\R$, one has
\begin{equation}\label{e:phase-space-equi-sec5}
\nu_t^B(\dd x,\dd\xi)
=
\dd x\otimes\lambda_t(\dd\xi),
\qquad
\lambda_t:=(\pi_\xi)_*\nu_t^B.
\end{equation}
Consequently, the associated configuration-space quantum limit is spatially
equidistributed:
\[
\nu_t=\dd x
\qquad\text{for a.e. }t.
\]
If, in addition, $\tau_h\ll h^{-1}$, then one has
\[
\lambda_t=\lambda_0:=(\pi_\xi)_*\nu_0^B
\qquad\text{for a.e. }t,
\]
and therefore
\[
\nu_t^B(\dd x,\dd\xi)
=
\dd x\otimes\lambda_0(\dd\xi).
\]
\end{theorem}
This theorem proves item~\textup{(i)} of Theorem~\ref{t:intro-main-structure}. The sufficient conditions stated in Theorem~\ref{t:intro-simplified-equidistribution}, as well as the specialization to the magnetic Laplacian, will be obtained below by verifying the geometric non-vanishing condition on
\(G_{\widehat B_0,\Lambda}\).
We first record the elementary escape argument on which the proof of Theorem \ref{t:x-equidistribution} relies.
\begin{lemma}\label{l:lemme-mesure-inv-par-XLambda}
Let \(\Lambda\in\mathcal L_1\), and let \(\mu\) be a finite nonnegative measure 
on
\[
\mathbb T^2\times
\big(E_{\Lambda^\perp\setminus\{0\}}\cap\Omega_{E_1,E_2}\big)
\times\mathbb R.
\]
Assume that \(\mu\) is invariant under the flow generated by \(X_\Lambda\), and
that
\begin{equation}\label{e:hyp-G-neq0}
    \forall (x,\xi) \in\mathbb T^2\times
\big(E_{\Lambda^\perp\setminus\{0\}}\cap\Omega_{E_1,E_2}\big), \qquad  G_{\widehat B_0,\Lambda}(x,\xi)\neq0.
\end{equation}
Then, one has \(\mu=0\).
\end{lemma}

\begin{proof}
We argue by contradiction. Assume that \(\mu\neq0\). Let
\[
s\longmapsto (x(s),\xi(s),\eta(s))
\]
be an integral curve of $X_\Lambda$. Then, one has
\[
\dot x(s)
=
\eta(s)\frac{\mathfrak e_\Lambda}{L_\Lambda},
\qquad
\dot\xi(s)=0,
\qquad
\dot\eta(s)=-G_{\widehat B_0,\Lambda}(x(s),\xi(s)).
\]
Since $\T^2\times\big(E_{\Lambda^\perp\setminus\{0\}}\cap\Omega_{E_1,E_2}\big)$ is compact and \(G_{\widehat B_0,\Lambda}\) is continuous, the assumption \eqref{e:hyp-G-neq0} implies that there exists \(c_\Lambda>0\) such that
\[
\left|G_{\widehat B_0,\Lambda}(x,\xi)\right|\ge c_\Lambda
\]
on this set. Moreover, \(G_{\widehat B_0,\Lambda}\) has a constant sign on each connected component. We treat the restriction of the measure to a component on which \(G_{\widehat B_0,\Lambda}<0\); the argument for the case \(G_{\widehat B_0,\Lambda}>0\) is identical.
On such a component, one has
\[
G_{\widehat B_0,\Lambda}(x,\xi)\le -c_\Lambda.
\]
Therefore, along the flow of \(X_\Lambda\), for every \(s\ge0\),
\[
\eta(s)
=
\eta(0)
-
\int_0^s
G_{\widehat B_0,\Lambda}(x(\sigma),\xi)\,\dd\sigma
\ge
\eta(0)+c_\Lambda s,
\]
which implies that  
\begin{equation}\label{e:eta-tend-vers-infini}
     \eta(s)   \xrightarrow[s\to \infty]{} + \infty.
\end{equation}
One can set $\mu^-:=\mathbbm{1}_{\{G_{\widehat B_0,\Lambda}<0\}}\mu$ that is assumed to be a nonzero measure. Since $\mu$ is a nonzero finite positive measure and the set $\{G_{\widehat B_0,\Lambda}<0\}$ is invariant under the flow, then the measure $\mu^-$ is an invariant finite positive measure. Hence, there exists $R_0>0$ such that $\mu^-\left(\T^2 \times \R^2 \times \{ \lvert \eta \lvert \le R_0 \}\right)\ge \varepsilon >0 $. Moreover, as \eqref{e:eta-tend-vers-infini} holds, there exists $R_1 >0$ such that, on \(\operatorname{supp}\mu^-\),
\begin{center}
    $\displaystyle \phi_{X_\Lambda}^{s_1}\left( \{ \lvert \eta \lvert \le R_0 \}\right) \subset \{ 2R_0 \le \lvert \eta \lvert \le R_1 \}$ for $s_1$ large enough. 
\end{center}
From the invariance of $\mu^-$ under $\phi_{X_\Lambda}^{s}$, one has, 
\begin{center}
    $\displaystyle \mu^-\left(  \phi_{X_\Lambda}^{s_1}\left( \{ \lvert \eta \lvert \le R_0 \}\right) \right) = \mu^-\left( \{ \lvert \eta \lvert \le R_0 \}\right) \ge \varepsilon >0$. 
\end{center}
Hence, one has
\begin{center}
    $\displaystyle \mu^-\left(\T^2 \times \R^2 \times\{ \lvert \eta \lvert \le R_0 \} \sqcup \T^2 \times \R^2 \times\{ 2R_0 \le \lvert \eta \lvert \le R_1 \} \right) \ge 2 \varepsilon. $
\end{center}
In that manner, one can construct a sequence $(s_n, R_n)$ such that, 
\begin{center}
     $\displaystyle \mu^-\left( \T^2 \times \R^2 \times \{\lvert \eta \lvert \le R_0 \} \sqcup \bigsqcup_{k=1}^n \T^2 \times \R^2 \times \{ 2R_{k-1} \le \lvert \eta \lvert \le R_k \} \right) \ge (n+1)\varepsilon$. 
\end{center}
This contradicts the finiteness of $\mu^-$ on $\T^2\times \big( E_{\Lambda^\perp\setminus\{0\}}\cap\Omega_{E_1,E_2}\big) \times\R$ and leads to $\mu^-= 0$. The same argument applied to \(\mu^+ :=\mathbbm{1}_{\{G_{\widehat B_0,\Lambda}>0\}}\mu\) gives \(\mu^+=0\). Since
\(\mu=\mu^-+\mu^+\), we obtain \(\mu=0\). 
\end{proof}
We are now in position to prove Theorem \ref{t:x-equidistribution}.
\begin{proof}[Proof of Theorem \ref{t:x-equidistribution}]
By Lemma~\ref{l:reduction-compact-sec5}, all terms in the decomposition of
\(\nu_t^B\) are already independent of \(x\), except possibly the weighted
compact two-microlocal contributions
\[
h_\Lambda(t)(\pi_{x,\xi})_*\widetilde\nu_{t,\Lambda}^B.
\]
It is therefore enough to prove that, for every \(\Lambda\in\mathcal L_1\),
\[
h_\Lambda(t)\widetilde\nu_{t,\Lambda}^B=0
\qquad\text{for a.e. }t.
\]
Equivalently, for \(h_\Lambda(t)\dd t\)-a.e. \(t\), it is enough to show that
\(\widetilde\nu_{t,\Lambda}^B=0\).

Fix $\Lambda\in\mathcal L_1$. By Lemma~\ref{l:three-regimes-2micro}, for
$h_\Lambda(t)\dd t$-a.e. $t$, the measure
$\widetilde\nu_{t,\Lambda}^B$ is invariant under the flow generated by
$X_\Lambda$. Moreover, by the spectral localization assumption and by the
definition of the compact two-microlocal component, it is supported in
\[
\T^2\times
\big(
E_{\Lambda^\perp\setminus\{0\}}\cap\Omega_{E_1,E_2}
\big)
\times\R.
\]
By the assumption \eqref{e:G-sign-sec5}, the measure \(\widetilde\nu_{t,\Lambda}^B\) satisfies the hypotheses of Lemma~\ref{l:lemme-mesure-inv-par-XLambda} for \(h_\Lambda(t)\dd t\)-a.e.\(t\). Hence, one has
\[
\widetilde\nu_{t,\Lambda}^B=0
\qquad
\text{for }h_\Lambda(t)\dd t\text{-a.e. }t.
\]
Equivalently, one gets
\[
h_\Lambda(t)\widetilde\nu_{t,\Lambda}^B=0
\qquad
\text{for }\dd t\text{-a.e. }t.
\]
By Lemma~\ref{l:reduction-compact-sec5}, the only terms which could carry
nonzero Fourier modes in \(x\) were precisely the weighted compact
two-microlocal contributions. Since these terms vanish, all nonzero Fourier
modes of \(\nu_t^B\) vanish. Hence \(\nu_t^B\) is independent of \(x\). Since the \(\xi\)-marginal of \(\nu_t^B\) is \(\lambda_t\), we obtain
\[
\nu_t^B(\dd x,\dd\xi)
=
\dd x\otimes\lambda_t(\dd\xi),
\]
which proves \eqref{e:phase-space-equi-sec5}. Projecting onto the
configuration variable gives
\[
\nu_t=\dd x
\qquad\text{for a.e. }t.
\]

If \(\tau_h\ll h^{-1}\), Lemma~\ref{l:xi-marginal} gives
\[
\lambda_t=\lambda_0
\qquad\text{for a.e. }t,
\]
and therefore
\[
\nu_t^B(\dd x,\dd \xi)=\dd x\otimes\lambda_0(\dd\xi).
\]
\end{proof}
\begin{corollary}\label{c:x-equidistrib-corollary}
Let \((u_h)_{h\to0^+}\) be a family of solutions to \eqref{e:Schrodinger-PDE}. Assume that \(H\) and \(R_h\) satisfy the standing assumptions of Section~\ref{ss:magnetic-operator}, and that the initial data \((u_h^{(0)})_{h\to0^+}\) satisfy the spectral localization condition \eqref{e:spectral-projection}. Assume that
\[
\tau_h\gg h^{-1/2},
\]
and that, for every $\Lambda\in\mathcal L_1$ and every
\(
(x,\xi)\in \T^2\times \big( E_{\Lambda^\perp\setminus\{0\}} \cap \Omega_{E_1,E_2} \big)
\), one has
\begin{equation}\label{e:G-positive-corollaire-sec5}
G_{\widehat B_0,\Lambda}(x,\xi)\neq0.
\end{equation}
Then every time-dependent configuration-space quantum limit is spatially
equidistributed:
\[
\nu_t=\dd x
\qquad\text{for a.e. }t.
\]
In particular, for every \(a\in \mathcal{C}^\infty(\T^2)\) and every
\(\psi\in \mathcal{C}_c^\infty(\mathbb R)\), one has
\[
\lim_{h\to0^+}
\int_{\R \times \T^2}
\psi(t)a(x)\lvert u_h(t\tau_h)\lvert^2
\dd t \dd x 
=
\left(\int_{\mathbb R}
\psi(t) \dd t \right)
\left(
\int_{\mathbb T^2}
a(x)\dd x\right).
\]
\end{corollary}

\begin{corollary}\label{c:simplified-x-equidistribution}
Let \((u_h)_{h\to0^+}\) be a family of solutions to
\eqref{e:Schrodinger-PDE}. Assume that \(H\) and \(R_h\) satisfy the
standing assumptions of Section~\ref{ss:magnetic-operator}, and that the
initial data \((u_h^{(0)})_{h\to0^+}\) satisfy the spectral localization
condition \eqref{e:spectral-projection}. Assume that
\[
\tau_h\gg h^{-1/2},
\]
and that \(\operatorname{Hess}(H)\) is positive definite on \(\Omega_{E_1,E_2}\). 
Then there exists an explicit constant $C_{H,\widehat B_0,E_1,E_2}>0$, depending only on \(H\), \(\widehat B_0\), \(E_1\), and \(E_2\), such that,
if
\begin{equation}\label{e:petitesse-norm-deriveR}
   \|\partial_x\mathscr R\|_
{L^\infty(\T^2\times\Omega_{E_1,E_2})}
\le
C_{H,\widehat B_0,E_1,E_2}, 
\end{equation}
then, for every \(a\in\mathcal C^\infty(\T^2)\) and every
\(\psi\in\mathcal C_c^\infty(\R)\), one has
\[
\lim_{h\to0^+}
\int_{\R}
\psi(t)
\int_{\T^2}
a(x)|u_h(t\tau_h,x)|^2
\,\dd x\,\dd t
=
\left(
\int_{\R}\psi(t)\,\dd t
\right)
\left(
\int_{\T^2}a(x)\,\dd x
\right).
\]
\end{corollary}
\begin{proof}
Set
\[
c_H:=\min_{\xi\in\Omega_{E_1,E_2}}|\nabla H(\xi)|,
\quad
\kappa_H:=
\min_{\substack{\xi\in\Omega_{E_1,E_2}\\ |v|=1}}
\operatorname{Hess}(H)(\xi)v\cdot v \text{ and }
M_H:=
\max_{\xi\in\Omega_{E_1,E_2}}
\|\operatorname{Hess}(H)(\xi)\|.
\]
By the regularity of the energy window \eqref{e:energy-window} and the positive definiteness of
\(\operatorname{Hess}(H)\), one has
\[
c_H>0,
\qquad
\kappa_H>0,
\qquad
M_H<+\infty.
\]
Fix \(\Lambda\in\mathcal L_1\) and $\xi\in
E_{\Lambda^\perp\setminus\{0\}}
\cap\Omega_{E_1,E_2}$. Hence, one has
\[
\left|
\widehat B_0\,
\operatorname{Hess}(H)(\xi)\frac{\mathfrak e_\Lambda}{L_\Lambda}
\cdot\nabla H(\xi)^\perp
\right|
\ge
|\widehat B_0|c_H\kappa_H.
\]
Moreover, using \eqref{e:petitesse-norm-deriveR}, one has 
\[
\left|
\operatorname{Hess}(H)(\xi)\frac{\mathfrak e_\Lambda}{L_\Lambda}
\cdot
\partial_x\mathcal I_\Lambda(\mathscr R)(x,\xi)
\right|
\le
M_H\|\partial_x\mathscr R\|_{L^\infty}
\le
M_HC_{H,\widehat B_0,E_1,E_2}.
\]
Therefore, defining 
\begin{equation}\label{e:definition-constante-petit-norme-deriveR}
    C_{H,\widehat B_0,E_1,E_2} := \frac{|\widehat B_0|c_H\kappa_H}{2M_H}, 
\end{equation}
leads to 
\[
|G_{\widehat B_0,\Lambda}(x,\xi)|
\ge
\frac{|\widehat B_0|c_H\kappa_H}{2}
>0.
\]
The conclusion follows from Corollary~\ref{c:x-equidistrib-corollary}.
\end{proof}

\subsection{Application to the magnetic Laplacian}
We now explain how the preceding criterion applies to the magnetic Laplacian
introduced in \eqref{e:mag-laplacian}. By the torus-adapted magnetic Weyl
quantization recalled in \eqref{e:quantization-of-magnetic-laplacian}, the magnetic Laplacian can be quantized as
\[
h^2(\mathcal L^{B}+V)
=
\Op\left(|\xi|^2+h\,\xi\cdot c_1(x)+h^2(c_0(x)+V(x)\right),
\]
where, in the notation of \cite[\S3.3]{MorinRiviere2025},
\[
c_1(x)=-2A^{\mathrm{per}}(x),
\qquad
c_0(x)=|A^{\mathrm{per}}(x)|^2 .
\]
Thus the magnetic Laplacian fits into the general definition of $\widehat H_h$; see \eqref{e:magnetic-operators}, 
with
\[
H(\xi)=|\xi|^2,
\qquad
R_h(x,\xi)=\xi\cdot c_1(x)+h (c_0(x)+V(x)).
\]
With the notation \eqref{e:definition-Rh}, the leading subprincipal contribution is therefore
\[
\mathscr{R}(x,\xi):=\xi\cdot c_1(x)
=
-2\,\xi\cdot A^{\mathrm{per}}(x).
\]
Using the fact that $\xi \in E_{\Lambda^\perp\setminus\{0\}}$ in the definition of \(G_{\widehat B_0,\Lambda}\) gives
\[
G_{\widehat B_0,\Lambda}(x,\xi)
=
-4 \left\langle \xi, \frac{\frak{e}_\Lambda^\perp}{L_\Lambda}\right\rangle \left( \widehat B_0+
 \frac{\frak{e}_\Lambda}{L_\Lambda} \cdot \partial_x\mathcal I_\Lambda\left(\frac{\frak{e}_\Lambda^\perp}{L_\Lambda} \cdot A^\mathrm{per}(x)\right) \right).
\]
One can use \eqref{e:B=dA} to compute the right term, which leads to 
\[
\frac{\frak{e}_\Lambda}{L_\Lambda} \cdot \partial_x\mathcal I_\Lambda\left(\frac{\frak{e}_\Lambda^\perp}{L_\Lambda} \cdot A^\mathrm{per}(x)\right) = \operatorname{rot}\mathcal{I}_\Lambda(A^\mathrm{per}) =\mathcal{I}_\Lambda(B-\widehat B_0) .
\]
Since $\xi \in E_{\Lambda^\perp\setminus\{0\}}$, one has 
\begin{equation}\label{e:G-laplacien}
    G_{\widehat B_0,\Lambda}(x,\xi) = -4 \left\langle \xi, \frac{\frak{e}_\Lambda^\perp}{L_\Lambda}\right\rangle \mathcal{I}_\Lambda(B)(x) = \mp 4\lvert \xi \lvert \mathcal{I}_\Lambda(B)(x).
\end{equation}
Considering the magnetic Schrödinger equation given by 
\begin{equation}\label{e:magnetic-Schrodinger}
\left\{
\begin{aligned}
   i h \partial_t u_{h}(t,x) &= h^2(\mathcal{L}^B+V) u_{h}(t,x), \\[0.3em]
   u_{h}(t=0,x) &= u_{h}^{(0)}(x),
\end{aligned}
\right.
\end{equation}
under the normalization
\[
\|u_{h}^{(0)}\|_{L^2(\T^2,L)}=1.
\]
Since \(H(\xi)=|\xi|^2\), the only critical value of \(H\) is \(0\). Hence any
energy window \(0<E_1<E_2\) automatically satisfies the non-criticality
assumption used in Theorem~\ref{t:x-equidistribution}. We therefore obtain the
following corollary.
\begin{corollary}\label{c:corollaire-laplacien-magnetic}
    Let \((u_h)_{h\to 0^+}\) be a family of solutions of \eqref{e:magnetic-Schrodinger} satisfying the spectral localization assumption \eqref{e:spectral-projection}. Assume that
\[
\tau_h\gg h^{-1/2},
\]
and that, for every \((x,\Lambda)\in \T^2 \times \mathcal L_1\), 
\begin{equation}\label{e:champs-positif}
\mathcal{I}_\Lambda(B)(x)\neq 0.
\end{equation}
Then every time-dependent magnetic semiclassical measure is independent of the
configuration variable. More precisely, for a.e. $t\in\R$,
\begin{equation}\label{e:phase-space-equi-laplacien-sec5}
\nu_t^B(\dd x,\dd\xi)
=
\dd x\otimes\lambda_t(\dd\xi),
\qquad
\lambda_t:=(\pi_\xi)_*\nu_t^B.
\end{equation}
Consequently, for every \((a,\psi) \in\mathcal C^\infty(\T^2)\times \mathcal C_c^\infty(\R)\), one has
\begin{equation}\label{e:configuration-equi-laplacien-sec5}
\lim_{h\to0^+}
\int_{\R}
\psi(t)
\int_{\T^2}
a(x)\lvert u_h(t\tau_h,x)\rvert^2
\,\dd x\,\dd t
=
\left(
\int_{\R}\psi(t)\,\dd t
\right)
\left(
\int_{\T^2}a(x)\,\dd x
\right).
\end{equation}
\end{corollary}

\begin{remark}
If the magnetic field satisfies $B>0$, then \eqref{e:champs-positif} holds. The case of the magnetic Laplacian therefore recovers exactly the same geometric control condition on $B$ as in~\cite{MorinRiviere2025}. Similar geometric assumptions on magnetic fields already appeared in a classical kinetic setting in the work of Glass and Han-Kwan~\cite{Glass-Han-Kwan-2012}.
\end{remark}

\begin{remark}[Connection with observability]
The equidistribution result above has a direct interpretation in terms of
observability. Indeed, let \(\omega\subset\T^2\) be a nonempty open set and let
\(\psi\in \mathcal{C}_c^\infty(\R)\) be nonnegative and not identically zero. Under the
assumptions of Theorem~\ref{t:x-equidistribution}, the conclusion
\(\nu_t=\dd x\), combined with the Portmanteau Theorem \cite[Chapter~1]{Billingsley1999} applied to the open set \(\R \times \omega \), implies
\[
\left(\int_{\R}\psi(t)\,\dd t\right)
|\omega|\le \liminf_{h\to0^+}
\int_{\R}\psi(t)
\int_{\omega}
|u_h(t\tau_h,x)|^2\,\dd x\,\dd t.
\]
Moreover, if \(|\partial\omega|=0\), then the preceding \(\liminf\) is a limit and one has equality. In particular, no normalized spectrally localized sequence can become asymptotically invisible on \(\omega\) at the time scale \(\tau_h\gg h^{-1/2}\). By the usual contradiction argument based on semiclassical measures, this yields a semiclassical observability estimate in this regime. Such observability statements are, in turn, closely related to controllability properties of the corresponding Schrödinger equation through the HUM method introduced by Lions~\cite{Lions-1988}. We refer for instance to the works of Burq--Zworski~\cite{Burq-Zworski-2004}, Maci\`a~\cite{Macia-2010} and Anantharaman--Maci\`a~\cite{Anantharaman-Macia-2014} for related semiclassical approaches to observability and control for Schrödinger equations. For electromagnetic Schrödinger operators on \(\T^2\), we also mention the recent work of Le Balc'h, Niu and Sun~\cite{LeBalc'h}, where a geometric condition involving the magnetic field is shown to play a decisive role in observability.
\end{remark}

\section{Beyond the Heisenberg time}\label{s:long-time-propagation}
We now turn to longer time scales and study the remaining dynamics of the limiting measures in the momentum variable. This will conclude the proof of Theorem \ref{t:intro-main-structure}. Throughout this section, we assume that
\[
\tau_h \gg h^{-1/2},
\]
and that the hypotheses of Theorem~\ref{t:x-equidistribution} are satisfied. In particular, the configuration-space component of the limiting measure is equidistributed. Hence there exists a measurable family \((\lambda_t)_{t\in\R}\) of probability measures on \(\Omega_{E_1,E_2}\) such that
\begin{equation}\label{e:lambda-disintegration-sec6}
\nu_t^B(\dd x,\dd \xi)=\dd x\otimes \lambda_t(\dd \xi),
\end{equation}
where \(\lambda_t=(\pi_\xi)_*\nu_t^B\). Let us first recall the behavior below the Heisenberg time scale. By Theorem \ref{t:x-equidistribution}, if $h^{-1/2}\ll \tau_h\ll h^{-1}$, the momentum marginal does not evolve: for almost every \(t\in\R\),
\begin{equation}\label{e:lambda-disintegration-tau-petit-devant-1surh}
\nu_t^B(\dd x,\dd \xi)=\dd x\otimes \lambda_0(\dd \xi).
\end{equation}
Hence, after spatial equidistribution has occurred, one cannot expect to observe a non-trivial evolution in the momentum variable before the Heisenberg time scale \(\tau_h= h^{-1}\).

The purpose of this section is to describe this remaining evolution of \((\lambda_t)_{t\in\R}\) according to the relative size of \(\tau_h\) with respect to \(h^{-1}\). In the very long-time regime, we will further identify the conditional measures of \(\lambda_t\) on the energy levels of \(H\). The first step is the following lemma, which gives the effective dynamics of the momentum marginal. It is useful to observe that the effective momentum vector field
\[
\widetilde X_H
:=
-\widehat B_0\,\nabla H^\perp\cdot\partial_\xi
\]
satisfies
\[
\widetilde X_H H
=
-\widehat B_0\nabla H^\perp\cdot\nabla H
=
0,
\]
so that $\widetilde X_H$ is tangent to the energy levels of $H$. Consequently, the associated flow \(\Upsilon_H^t\) preserves each energy level \(\mathbb S_E =\{H=E\}\).

At this stage, Theorem~\ref{t:x-equidistribution} already yields the product structure \eqref{e:lambda-disintegration-sec6} and proves item~\textup{(i)} of Theorem~\ref{t:intro-main-structure}. The next lemma proves item~\textup{(ii)} and provides the invariance property needed to establish item~\textup{(iii)}.

\begin{lemma}\label{l:lambda-regularity}
Assume that the hypotheses of Theorem~\ref{t:x-equidistribution} are satisfied. Then the measure $\lambda_t$ satisfies the following dynamical properties. 
\begin{enumerate}
\item If $\tau_h= h^{-1}$, then the measure $\lambda_t$ satisfies the following transport equation:
\begin{equation}\label{e:xi-transport}
\displaystyle \partial_t(\lambda_t) = \widehat{B}_0\nabla H^\perp(\xi) \cdot \partial_\xi(\lambda_t).
\end{equation}
More precisely, if \(\Upsilon_H^t\) denotes the flow on \(\R^2\) defined by
\begin{equation}\label{e:def-flot-upsilon}
\dot \xi(t)=-\widehat B_0\,\nabla H(\xi(t))^\perp ,
\end{equation}
then
\begin{equation}\label{e:lambda-transported}
\lambda_t=(\Upsilon_H^{t})_*\lambda_0
\end{equation}
for almost every \(t\in\R\).

\item If $ \tau_h \gg h^{-1}$, then, for almost every \(t\in\R\), the measure \(\lambda_t\) is invariant under the flow \(\Upsilon_H^s\), namely
\begin{equation}\label{e:xi-invariance}
\lambda_t=(\Upsilon_H^{s})_*\lambda_t,
\qquad \forall s\in\R .
\end{equation}
\end{enumerate}
\end{lemma}
\begin{proof}
Let $\psi\in \mathcal{C}_c^\infty(\R)$ and $a=a(\xi)\in \mathcal{C}_c^\infty(\R^2)$. As in the previous sections, we study the commutator appearing in the following identity 
\begin{equation}\label{e:IBP-time-xi}
\displaystyle -\frac{ih}{\tau_h} \int_\R \psi'(t) \langle \Op(a(\xi)) u_h(t\tau_h) , u_h(t\tau_h )\rangle \dd t  = \int_\R \psi(t) \langle \left[ \Op(a) , \Op(H + h R_h)\right] u_h(t\tau_h) , u_h(t\tau_h) \rangle \dd t .
\end{equation}
The computation is simpler than in the previous two-microlocal analysis, because the observable \(a\) depends only on the momentum variable. By the magnetic commutator expansion \eqref{e:commutator}, we obtain
\begin{equation}\label{e:commutator-expansion-xi}
[\Op(a),\Op(H+hR_h)]
=\frac{h}{i}\,\Op(\{a,H+hR_h\})
+\frac{\widehat B_0h^2}{i}\,\Op(\nabla_\xi(H+hR_h)^\perp\cdot\nabla_\xi a)
+h^3\,\Op(r_h),
\end{equation}
where $r_h\in \mathscr{S}^0$ is bounded uniformly in $h$ (and depends on finitely many seminorms of
$a,H,R_h$). Since $a=a(\xi)$, we have $\{a,H\}=0$ and $\{a,hR_h\}=-h\,\nabla_\xi a\cdot\partial_xR_h$,
so \eqref{e:commutator-expansion-xi} becomes
\begin{equation}\label{e:commutator-expansion-xi-simplified}
[\Op(a),\Op(H+hR_h)]
=-\frac{h^2}{i}\,\Op(\nabla_\xi a\cdot\partial_xR_h)
+\frac{\widehat B_0h^2}{i}\,\Op(\nabla_\xi H(\xi)^\perp\cdot\nabla_\xi a)
+h^3\,\Op(\widetilde r_h),
\end{equation}
with $\widetilde r_h\in \mathscr{S}^0$ uniformly bounded. Plugging \eqref{e:commutator-expansion-xi-simplified} into \eqref{e:IBP-time-xi} and dividing by $-ih^2$ gives
\begin{align}\label{e:master-identity-xi}
\frac{1}{h\tau_h}\int_\R \psi'(t)\,
\big\langle \Op(a)\,u_h(t \tau_h),u_h(t \tau_h)\big\rangle\,\dd t 
&= -  \int_\R \psi(t)\,
\big\langle \Op(\nabla_\xi a\cdot\partial_xR_h)\,u_h(t \tau_h),u_h(t \tau_h)\big\rangle\,\dd t 
\nonumber\\
&\quad +\widehat B_0  \int_\R \psi(t)\,
\big\langle \Op(\nabla_\xi H^\perp\cdot\nabla_\xi a)\,u_h(t \tau_h),u_h(t \tau_h)\big\rangle\,\dd t 
\nonumber\\
&\quad + h\,\mathcal O(1),
\end{align}
where $\mathcal O(1)$ is uniform in $h$ (for fixed $a,\psi$) by the Calderón--Vaillancourt Theorem \ref{t:Calderon-Vaillancourt}. Recall that from the definition \eqref{e:definition-Rh} of $R_h$, one has
\[
R_h(x,\xi) = \mathscr{R}(x,\xi) + \mathcal{O}_{\mathscr{S}^{m}}(h). 
\]
Since \(a\) is compactly supported, the symbol \(\nabla_\xi a\cdot\partial_xR_h\) is uniformly bounded in \(\mathscr{S}^0\). Therefore, by the Calderón--Vaillancourt theorem,
\begin{align}\label{e:integrale-avec-Rh}
    \int_\R \psi(t)\big\langle \Op(\nabla_\xi a\cdot\partial_xR_h)\,u_h(t \tau_h),u_h(t \tau_h)\big\rangle\,\dd t  = &\int_\R \psi(t)\big\langle \Op(\nabla_\xi a\cdot\partial_x\mathscr{R})\,u_h(t \tau_h),u_h(t \tau_h)\big\rangle\,\dd t \nonumber\\
    &+ h\mathcal{O}(1).
\end{align}
Finally, combining \eqref{e:integrale-avec-Rh} and \eqref{e:master-identity-xi} leads to 
\begin{align}\label{e:master-identity-xi-final}
\frac{1}{h\tau_h}\int_\R \psi'(t)\,
\big\langle \Op(a)\,u_h(t \tau_h),u_h(t \tau_h)\big\rangle\,\dd t 
&= -  \int_\R \psi(t)\,
\big\langle \Op(\nabla_\xi a\cdot\partial_x\mathscr{R})\,u_h(t \tau_h),u_h(t \tau_h)\big\rangle\,\dd t 
\nonumber\\
&\quad +\widehat B_0  \int_\R \psi(t)\,
\big\langle \Op(\nabla_\xi H^\perp\cdot\nabla_\xi a)\,u_h(t \tau_h),u_h(t \tau_h)\big\rangle\,\dd t 
\nonumber\\
&\quad + h\,\mathcal O(1).
\end{align}
Recall that we place ourselves in the regime $\tau_h \gg h^{-1/2}$ and that from Theorem \ref{t:x-equidistribution} one knows exactly the $x$-marginal of the measure $\nu_t^B$ that is the Lebesgue measure. Moreover, one has 
\[
\int_\R \psi'(t)
\big\langle \Op(a)u_h(t\tau_h),u_h(t\tau_h)\big\rangle\,\dd t \xrightarrow[h\to 0^+]{} \int_\R \psi'(t)\int_{\R^2}a(\xi)\,\lambda_t(\dd\xi)\,\dd t.
\]

Therefore, if \(\tau_h= h^{-1}\), so that \(h\tau_h\to1\), the whole left-hand side of \eqref{e:master-identity-xi-final} converges to this limit. If \(\tau_h\gg h^{-1}\), then \(h\tau_h\to+\infty\), and the whole left-hand side converges to \(0\). Moreover, since $\nu_t^B$ is Lebesgue in $x$, the term involving $\partial_x\mathscr{R}$ vanishes in the limit, because $\mathscr{R}$ is periodic in $x$ and therefore $\int_{\T^2}\partial_x\mathscr{R}(x,\xi)\,\dd x=0$. Hence, the right-hand side of \eqref{e:master-identity-xi-final} converges to 
\[
\widehat B_0 \int_\R \psi(t)\left[\int_{\T^*\T^2} \nabla_\xi H^\perp(\xi) \cdot \nabla_\xi a(\xi) \dd x \lambda_t(\dd \xi)\right] \dd t  = \widehat B_0 \int_\R \psi(t)\left[\int_{\R^2} \nabla_\xi H^\perp(\xi) \cdot \nabla_\xi a(\xi) \lambda_t(\dd \xi)\right] \dd t. 
\]
In the regime \(\tau_h= h^{-1}\), passing to the limit in \eqref{e:master-identity-xi-final} gives
\begin{equation}\label{e:weak-xi-transport}
\int_\R \psi'(t) \int_{\R^2}a(\xi)\,\lambda_t(\dd\xi)\dd t = \widehat B_0 \int_\R \psi(t) \int_{\R^2} \nabla_\xi H(\xi)^\perp\cdot\nabla_\xi a(\xi) \,\lambda_t(\dd\xi)\dd t .
\end{equation}
This is exactly the weak formulation of \eqref{e:xi-transport}.

In the regime \(\tau_h\gg h^{-1}\), the left-hand side of \eqref{e:master-identity-xi-final} converges to zero. Hence, for almost every
\(t\in\R\),
\[
\int_{\R^2}
\nabla_\xi H(\xi)^\perp\cdot\nabla_\xi a(\xi)
\,\lambda_t(\dd\xi)
=0,
\qquad
\forall a\in \mathcal{C}_c^\infty(\R^2).
\]
This is the infinitesimal invariance of \(\lambda_t\) under
\(-\widehat B_0\nabla H^\perp\cdot\partial_\xi\). Equivalently, \(\lambda_t\) is invariant under the associated flow, which gives \eqref{e:xi-invariance}.
\end{proof}

We now combine Lemma~\ref{l:lambda-regularity} with the energy disintegration obtained in Lemma~\ref{l:xi-marginal}. Recall that there exists a probability measure $\sigma^B$ on $[E_1,E_2]$, depending only on the initial data, such that for almost every $t\in\R$ one can disintegrate
\[
\nu_t^B(\dd x,\dd\xi)
=
\nu_{t,\xi,E}^B(\dd x)\otimes \lambda_{t,E}(\dd\xi)\otimes \sigma^B(\dd E),
\]
where $\lambda_{t,E}$ is supported on the energy level
\[
\mathbb S_E:=\{\xi\in\R^2:\ H(\xi)=E\}.
\]
Recall from Lemma \ref{l:xi-marginal} that the marginal $\sigma^B$ depends only on the initial data for every time scale. Since we are working under the assumptions of Theorem~\ref{t:x-equidistribution}, we know that the configuration-space variable is already equidistributed. In other words, for $\sigma^B$-almost every $E$, for $\lambda_{t,E}$-almost every $\xi\in\mathbb S_E$, and for almost every $t\in\R$, one has
\[
\nu_{t,\xi,E}^B(\dd x)=\dd x.
\]
Consequently, one has 
\begin{equation}\label{e:nu-energy-disintegration-after-x-equi}
\nu_t^B(\dd x,\dd\xi)
=
\dd x\otimes \lambda_{t,E}(\dd\xi)\otimes \sigma^B(\dd E).
\end{equation}

We now consider the regime $\tau_h\gg h^{-1}$. By Lemma~\ref{l:lambda-regularity}, the momentum marginal $\lambda_t$ is invariant under the flow generated by
\[
\widetilde X_H(\xi)
:=
-\widehat B_0\nabla H(\xi)^\perp\cdot\partial_\xi,
\]
whose associated flow $\Upsilon_H^t$ preserves each energy level $\SE$. Therefore, by uniqueness of the disintegration with respect to the energy variable, the conditional measures $\lambda_{t,E}$ satisfy, for $\sigma^B$-almost every $E$ and almost every $t\in\R$,
\begin{equation}\label{e:lambda-tE-invariance}
(\Upsilon_E^s)_*\lambda_{t,E}=\lambda_{t,E},
\qquad \forall s\in\R,
\end{equation}
where $\Upsilon_E^s$ denotes the restriction to $\mathbb S_E$ of the flow $\Upsilon_H^s$ generated by $\widetilde X_H$.

Recall that \eqref{e:flux} implies $\widehat B_0\neq0$. Since the levels $\mathbb S_E$ are regular, $\nabla H$ does not vanish on $\mathbb S_E$, and hence $\widetilde X_H$ is a nonvanishing tangent vector field on $\mathbb S_E$. Since $\mathbb S_E$ is diffeomorphic to a circle, the flow of $\widetilde X_H$ on $\mathbb S_E$ is periodic. Let $L_E>0$ denote its minimal period. Choosing a base point $\xi_E\in\mathbb S_E$, we define
\begin{equation}\label{e:def-psi-E}
\psi_E:\R/L_E\Z\longrightarrow \mathbb S_E,
\qquad
\psi_E(\theta):=\Upsilon_E^\theta(\xi_E).
\end{equation}
Then $\psi_E$ is a smooth diffeomorphism and conjugates the translation flow on $\R/L_E\Z$ with the flow $\Upsilon_E^s$ on $\mathbb S_E$.
\begin{lemma}\label{l:flow-coordinate-energy-level}
For every $E\in[E_1,E_2]$, the map $\psi_E$ satisfies
\[
(\psi_E)_*(\partial_\theta)=\widetilde X_H
\qquad\text{on }\mathbb S_E.
\]
Consequently, the flow $\Upsilon_E^s$ is conjugated by $\psi_E$ to the translation flow $\theta\mapsto \theta+s$ on $\R/L_E\Z$.
\end{lemma}

\begin{proof}
Let $\xi\in\mathbb S_E$. By definition of $\psi_E$, there exists $\theta_0\in\R/L_E\Z$ such that
\[
\xi=\psi_E(\theta_0)=\Upsilon_E^{\theta_0}(\xi_E).
\]
Then, one has
\[
\big[(\psi_E)_*\partial_\theta\big](\xi)
=
\left.\frac{\dd}{\dd\theta}\right|_{\theta=0}
\psi_E(\theta_0+\theta)
=
\left.\frac{\dd}{\dd\theta}\right|_{\theta=0}
\Upsilon_E^{\theta_0+\theta}(\xi_E)
=
\widetilde X_H(\Upsilon_E^{\theta_0}(\xi_E))
=
\widetilde X_H(\xi).
\]
This proves the first statement. The conjugation of the flows follows immediately from the definition of $\psi_E$.
\end{proof}
\begin{theorem}\label{t:energy-level-equidistribution}
Assume that the hypotheses of Theorem~\ref{t:x-equidistribution} are satisfied. Assume that $\tau_h\gg h^{-1}$. Then, on the regular energy window $[E_1,E_2]$, every time-dependent semiclassical measure satisfies, for almost every $t\in\R$,
\begin{equation}\label{e:final-energy-level-equi}
\nu_t^B(\dd x,\dd\xi)
=
\dd x\otimes
\int_{E_1}^{E_2}
(\psi_E)_*\left(\frac{\dd\theta}{L_E}\right)
\,\sigma^B(\dd E),
\end{equation}
where \(\sigma^B\), introduced in Lemma~\ref{l:xi-marginal}, depends only on the initial data. Equivalently, for every $\psi\in \mathcal{C}_c^\infty(\R)$ and every $a\in \mathcal{C}^\infty(\T^2\times\Omega_{E_1,E_2})$, one has
\begin{equation}\label{e:final-energy-level-equi-test}
\lim_{h\to0^+} \int_\R \psi(t) \left\langle \Op(a)u_h(t\tau_h),u_h(t\tau_h) \right\rangle \dd t = \int_\R \psi(t)\dd t \int_{E_1}^{E_2} \int_0^{L_E} \int_{\T^2} a(x,\psi_E(\theta)) \dd x\frac{\dd\theta}{L_E} \sigma^B(\dd E).
\end{equation}
In particular, in this regime, the limiting measure $\nu_t^B$ is independent of $t$ on the regular energy window.
\end{theorem}
Together with Theorem~\ref{t:x-equidistribution}, Lemma~\ref{l:xi-marginal}, and Lemma~\ref{l:lambda-regularity}, this theorem completes the proof of
Theorem~\ref{t:intro-main-structure}. Indeed, item~\textup{(i)} follows from the time-independence of the momentum marginal below the Heisenberg scale, item~\textup{(ii)} follows from the transport statement in Lemma~\ref{l:lambda-regularity}, and the present theorem proves item~\textup{(iii)}.
\begin{proof}
By \eqref{e:lambda-tE-invariance}, for \(\sigma^B\)-almost every \(E\in[E_1,E_2]\) and almost every \(t\in\R\), the probability measure \(\lambda_{t,E}\) is invariant under the flow \(\Upsilon_E^s\). Fix such an energy $E$. By Lemma~\ref{l:flow-coordinate-energy-level}, the map $\psi_E:\R/L_E\Z\to\mathbb S_E$ conjugates the flow $\Upsilon_E^s$ to the translation flow on the circle $\R/L_E\Z$. Hence the pullback measure
\[
\mu_{t,E}:=(\psi_E^{-1})_*\lambda_{t,E}
\]
is a probability measure on $\R/L_E\Z$ invariant under all translations. The only probability measure with this property is the normalized Lebesgue measure. Thus, one gets
\[
\mu_{t,E}=\frac{\dd\theta}{L_E},
\]
or equivalently,
\[
\lambda_{t,E}
=
(\psi_E)_*\left(\frac{\dd\theta}{L_E}\right).
\]
Substituting this identity into the energy disintegration gives
\[
\nu_t^B(\dd x,\dd\xi)
=
\dd x\otimes
\int_{E_1}^{E_2}
(\psi_E)_*\left(\frac{\dd\theta}{L_E}\right)
\sigma^B(\dd E),
\]
which is \eqref{e:final-energy-level-equi}. The formulation \eqref{e:final-energy-level-equi-test} follows by testing this identity against $a$ and integrating in time against $\psi$.
\end{proof}

\begin{remark}
The measure $(\psi_E)_*(\dd\theta/L_E)$ is the normalized invariant measure of the flow generated by $\widetilde X_H$ on $\mathbb S_E$. If $\dd\ell_E$ denotes the Euclidean arclength measure on $\mathbb S_E$, then this measure can also be written as
\[
(\psi_E)_*\left(\frac{\dd\theta}{L_E}\right)
=
\frac{\displaystyle \frac{\dd\ell_E}{|\nabla H(\xi)|}}
{\displaystyle \int_{\mathbb S_E}\frac{\dd\ell_E}{|\nabla H(\xi)|}}.
\]
Thus the equidistribution is uniform in the flow parameter, or equivalently with respect to the Liouville measure induced on the energy curve. It is not, in general, the normalized arclength measure unless $|\nabla H|$ is constant on
$\mathbb S_E$.
\end{remark}
We conclude the section, and the article, by spelling out the consequence of the previous theorem in the model case of the magnetic Laplacian. In that case the principal Hamiltonian is $H(\xi)=|\xi|^2$, and the regular energy levels are the Euclidean circles
\[
\mathbb S_E=\{\xi\in\R^2:\ |\xi|^2=E\}.
\]
The effective vector field of Lemma~\ref{l:lambda-regularity} is then
\[
\widetilde X_H(\xi)
=
-\widehat B_0\nabla H(\xi)^\perp\cdot\partial_\xi
=
-2\widehat B_0\xi^\perp\cdot\partial_\xi,
\]
so its flow is the rotation flow on each circle $\mathbb S_E$. Hence the invariant probability measure on $\mathbb S_E$ is the normalized angular
measure.
\begin{corollary}[Magnetic Laplacian]\label{c:magnetic-laplacian-final}
Assume that the hypotheses of Corollary~\ref{c:corollaire-laplacien-magnetic} are satisfied. If \(\tau_h\gg h^{-1}\), then for every \(\psi\in \mathcal{C}_c^\infty(\R)\) and every
\(a\in \mathcal{C}^\infty(\T^2\times\Omega_{E_1,E_2})\), one has
\begin{align}\label{e:magnetic-laplacian-final-test}
\lim_{h\to0^+}
\int_\R
&\psi(t)
\left\langle
\Op(a)u_h(t\tau_h),u_h(t\tau_h)
\right\rangle
\dd t \nonumber\\
&=
\int_\R \psi(t)\,\dd t
\int_{E_1}^{E_2}
\int_{\T^2}
\int_0^{2\pi}
a\big(x,\sqrt E(\cos\theta,\sin\theta)\big)
\,\frac{\dd\theta}{2\pi}\,\dd x\,
\sigma^B(\dd E),
\end{align}
where \(\sigma^B\) is the probability measure defined by 
\[
\forall f\in\mathcal C_c^\infty(\R), \qquad \lim_{h\to0^+}
\left\langle
f\bigl(h^2(\mathcal L^B+V)\bigr)u_h^{(0)},
u_h^{(0)}
\right\rangle = \int_{E_1}^{E_2} f(E)\sigma^B(\dd E).
\]
\end{corollary}
In particular, in the very long-time regime \(\tau_h\gg h^{-1}\), the limiting measure is independent of \(t\), Lebesgue in the position variable, and uniformly distributed in the angular momentum variable on each energy circle.
Under the hypotheses of Corollary~\ref{c:corollaire-laplacien-magnetic}, one has \(\mathcal I_\Lambda(B)\neq0\) for every \(\Lambda\in\mathcal L_1\). Therefore Theorem~\ref{t:energy-level-equidistribution} applies. This corollary thus completes the description of the successive long-time evolution of the limiting measures associated with admissible initial data for the magnetic Laplacian on the flat two-torus: first, the measure becomes Lebesgue in the position variable; then, beyond the Heisenberg time scale, the remaining momentum component becomes uniformly distributed on each regular energy circle.

\appendix
\section{Appendix}\label{s:appendix}

\subsection{Functional calculus}\label{ss:functional-calculus}
In the first part of the appendix, we review the elliptic, spectral and functional-calculus
properties of the magnetic pseudodifferential operators
\[
\widehat H_h:=\Op(H+hR_h)
\]
that are used throughout the paper. In the usual nonmagnetic semiclassical setting, these facts are standard; we refer for instance to Dimassi--Sj\"ostrand~\cite[Chapters~7--8]{DimassiSjostrand1999} and Zworski~\cite[\S4.7, \S14.3 and Appendix~C]{Zworski2012}. The point of this appendix is to record that the same arguments apply to the torus-adapted magnetic Weyl quantization introduced in Section~\ref{s:Magnetic field and quantization}. In the proof, the standard adjoint formula, composition formula and Calder\'on--Vaillancourt theorem need to be replaced by their magnetic counterparts recalled in Morin--Rivi\`ere~\cite[Proposition~3.4 and Theorems~3.5, 3.7]{MorinRiviere2025} but the overall arguments are the same. For comparison, related elliptic and spectral properties for magnetic pseudodifferential operators in $\R^n$ are established in Iftimie--M\u{a}ntoiu--Purice~\cite{Iftimie-Mantoiu-Purice-2007}. We shall use the Helffer--Sj\"ostrand formula in the form presented in Dimassi--Sj\"ostrand~\cite[Chapter~8]{DimassiSjostrand1999} and Zworski~\cite[\S14.3.2]{Zworski2012}. The original reference is Helffer--Sj\"ostrand~\cite{HelfferSjostrand1989}.

\medskip

We first fix the Sobolev conventions used below. For \(s\geq 0\), we define
\[
\|u\|_{\mathcal H^s(\T^2;L)}
:=
\|\Opsansh(\langle\xi\rangle^s)u\|_{L^2}
+
\|u\|_{L^2},
\]
and
\[
\|u\|_{\mathcal H_h^s(\T^2;L)}
:=
\|\Op(\langle\xi\rangle^s)u\|_{L^2}
+
\|u\|_{L^2}.
\]
The corresponding spaces are the spaces of distributions for which these norms are finite and are denoted by \(\mathcal H^s(\T^2;L)\) and \(\mathcal H_h^s(\T^2;L)\). Equivalently, one may use any elliptic symbol of order \(s\). For each fixed \(h>0\), the spaces \(\mathcal H_h^s(\T^2;L)\) and \(\mathcal H^s(\T^2;L)\) coincide as topological vector spaces, although the equivalence between the corresponding norms is not uniform as \(h\to0^+\).

Recall that, if \(a_h\in\mathscr S^\ell\) is uniformly bounded in \(\mathscr S^\ell\), then
\[
\Op(a_h):\mathcal H_h^s(\T^2;L)\to
\mathcal H_h^{s-\ell}(\T^2;L)
\]
is uniformly bounded. This follows directly from the magnetic composition formula and the magnetic Calder\'on--Vaillancourt Theorem \ref{t:Calderon-Vaillancourt}. For even nonnegative integers \(s=2k\), this definition is equivalent to the one obtained from powers of the magnetic Laplacian used in Section \ref{s:Magnetic field and quantization}.
\medskip

Throughout the appendix, we assume that $m>0$ and
\begin{equation}\label{e:appendix-symbols}
H\in \mathscr{S}^m(\R^2),
\qquad
R_h\in \mathscr{S}^{m}(\T^2\times\R^2),
\end{equation}
where the seminorms of $(R_h)_{0<h\le h_0}$ are uniformly bounded. We also assume that $H$ and $R_h$ are real-valued and that $H$ is elliptic of order $m$, namely there exist constants $C>0$ and $M_0>0$ such that
\begin{equation}\label{e:H-elliptic}
|H(\xi)|
\ge
C\langle\xi\rangle^m,
\qquad
|\xi|\ge M_0.
\end{equation}

\begin{theorem}\label{t:selfadjoint-domain}
Let $m>0$. Let $H\in \mathscr{S}^m(\R^2)$ be real-valued and elliptic of order $m$,
and let
\[
R_h\in \mathscr{S}^{m}(\T^2\times\R^2)
\]
be real-valued with seminorms uniformly bounded for
$h\in(0,h_0]$. Set
\[
a_h=H+hR_h,
\qquad
\widehat H_h:=\Op(a_h).
\]
Then, after possibly reducing $h_0>0$, the following properties hold for every
$0<h\le h_0$:
\begin{enumerate}
\item The operator $\widehat H_h$ is essentially self-adjoint on
$\mathcal{C}^\infty(\T^2;L)$.
\item Its unique self-adjoint extension has domain
\[
\operatorname{Dom}(\widehat H_h)=\mathcal H^m(\T^2;L).
\]
\item For each fixed $h>0$, the graph norm
\[
\|u\|_{L^2}+\|\widehat H_hu\|_{L^2}
\]
is equivalent to the standard $\mathcal H^m$-norm.
\item For every $z\in\C\setminus\R$, the resolvent
\[
(\widehat H_h-z)^{-1}:L^2(\T^2;L)\longrightarrow \mathcal H^m(\T^2;L)
\]
is well defined.
\item For every compact set $K\subset\C\setminus\R$, there exist
$h_K>0$ and $C_K>0$ such that
\[
\|(\widehat H_h-z)^{-1}\|_{L^2\to \mathcal H_h^m}
\le
C_K,
\qquad
0<h\le h_K,\quad z\in K.
\]
\item Since $\widehat H_h$ is self-adjoint, one also has the standard resolvent bound
\[
\|(\widehat H_h-z)^{-1}\|_{L^2\to L^2}
\le
\frac{1}{|\Im z|},
\qquad
z\in\C\setminus\R.
\]
\end{enumerate}
\end{theorem}

\begin{proof}
The proof follows the standard elliptic argument for self-adjoint pseudodifferential operators on compact manifolds. Namely, the ellipticity of \(H\), together with the uniform boundedness of \(R_h\) in \(\mathscr S^m\), implies that \(a_h=H+hR_h\) is elliptic of order \(m\), uniformly for
\(0<h\le h_0\), after reducing \(h_0\) if necessary. The usual elliptic parametrix construction gives the graph norm equivalence, the closedness of the realization on \(\mathcal H^m\), and the uniform
resolvent estimate away from the real axis. The adjoint formula gives symmetry for real-valued symbols, and the range criterion for closed symmetric operators then gives self-adjointness.

These are the standard arguments in the semiclassical calculus; see Dimassi--Sj\"ostrand~\cite[Chapters~7--8]{DimassiSjostrand1999} and Zworski~\cite[\S4.7, \S14.3 and Appendix~C]{Zworski2012}. In the present magnetic setting, the same proof applies with the magnetic composition formula, adjoint formula and Calder\'on--Vaillancourt Theorem recalled in Section~\ref{s:Magnetic field and quantization}.
\end{proof}

\begin{corollary}\label{c:discrete-spectrum}
Under the same hypotheses as in Theorem~\ref{t:selfadjoint-domain}, the
operator \(\widehat H_h=\Op(a_h)\) has compact resolvent on \(L^2(\T^2;L)\). Hence its
spectrum is real and purely discrete. More precisely, there exist a sequence
of real eigenvalues \((\lambda_j(h))_{j\ge1}\), counted with multiplicity and
satisfying \(|\lambda_j(h)|\to\infty\), and an orthonormal basis of
eigenfunctions \((u_j(h))_{j\ge1}\) of \(L^2(\T^2;L)\).
\end{corollary}

\begin{proof}
By Theorem~\ref{t:selfadjoint-domain}, the resolvent
\[
(\widehat H_h-i)^{-1}:L^2(\T^2;L)\longrightarrow \mathcal H^m(\T^2;L)
\]
is bounded. Since \(m>0\), the embedding
\[
\mathcal H^m(\T^2;L)\hookrightarrow L^2(\T^2;L)
\]
is compact by the Rellich--Kondrachov theorem. Hence \((\widehat H_h-i)^{-1}\) is a
compact operator on \(L^2(\T^2;L)\).

The conclusion follows from the spectral theorem for self-adjoint operators
with compact resolvent; see for instance Reed--Simon
\cite[Chapter~VIII, \S3]{ReedSimonI}.
\end{proof}

By the spectral theorem, for every \(u\in \operatorname{Dom}(\widehat H_h)\), one has
\begin{equation}\label{e:decomposition-operator}
\widehat H_hu
=
\sum_{j=1}^\infty
\lambda_j(h)\,
\langle u,u_j(h)\rangle_{L^2}\,
u_j(h),
\end{equation}
with convergence in \(L^2(\T^2;L)\). More generally, for every
\(f\in L^\infty(\R)\), we define the bounded operator
\(f(\widehat H_h):L^2(\T^2;L)\to L^2(\T^2;L)\) by
\begin{equation}\label{e:f-of-Op}
f(\widehat H_h)u
:=
\sum_{j=1}^\infty
f(\lambda_j(h))\,
\langle u,u_j(h)\rangle_{L^2}\,
u_j(h),
\qquad
u\in L^2(\T^2;L).
\end{equation}

%-------------------------------------------------------------
\subsubsection{The Helffer--Sj\"ostrand formula}

We now recall the Helffer--Sj\"ostrand formula in the form needed in the article. This
part is purely functional-analytic and does not use any specific property of
the magnetic quantization beyond the self-adjointness of \(\widehat H_h\).

Let \(f\in \mathcal{C}_c^\infty(\R)\). An almost-analytic extension of \(f\) is a
function \(\widetilde f\in \mathcal{C}_c^\infty(\C)\) such that
\[
\widetilde f|_{\R}=f
\]
and, for every \(N\in\N\),
\[
\partial_{\bar z}\widetilde f(z)
=
\mathcal O_N(|\Im z|^N)
\qquad
\text{as } \Im z\to0,
\]
where
\[
\partial_{\bar z}
:=
\frac12(\partial_x+i\partial_y),
\qquad
z=x+iy.
\]
Such extensions are standard; see for instance Zworski~\cite[\S14.3.2]{Zworski2012}.

\begin{proposition}[Helffer--Sj\"ostrand formula]
\label{p:helffer-sjostrand}
Let \(f\in \mathcal{C}_c^\infty(\R)\), and let \(\widetilde f\) be an almost-analytic
extension of \(f\). Then
\begin{equation}\label{e:helffer-sjostrand}
f(\widehat H_h)
=
\frac{1}{i\pi}
\int_{\C}
\partial_{\bar z}\widetilde f(z)\,(\widehat H_h-z)^{-1}\,L(\dd z),
\end{equation}
where \(L(\dd z)\) denotes the Lebesgue measure on \(\C\). The integral
converges in the operator norm on \(\mathcal L(L^2(\T^2;L))\).
\end{proposition}

\begin{proof}
This is the standard Helffer--Sj\"ostrand formula for self-adjoint operators.
We refer to Helffer--Sj\"ostrand~\cite{HelfferSjostrand1989},
Dimassi--Sj\"ostrand~\cite[Chapter~8]{DimassiSjostrand1999} or
Zworski~\cite[\S14.3.2]{Zworski2012}. The convergence follows from the
resolvent estimate
\[
\|(\widehat H_h-z)^{-1}\|_{L^2\to L^2}
\le
|\Im z|^{-1}
\]
and from the almost-analytic decay of
\(\partial_{\bar z}\widetilde f(z)\) near the real axis.
\end{proof}

\begin{corollary}\label{c:functional-calculus}
Let \(\widehat H_h:=\Op(a_h)\), with
\[
a_h(x,\xi)=H(\xi)+hR_h(x,\xi),
\]
under the assumptions of Theorem~\ref{t:selfadjoint-domain}. Let
\(f\in \mathcal{C}_c^\infty(\R)\). Then the following holds
\begin{equation}\label{e:functional-calculus-a_h}
f(\widehat H_h)
=
\Op(f\circ a_h)
+
\mathcal O_{L^2\to L^2}(h).
\end{equation}
Moreover, one has 
\begin{equation}\label{e:functional-calculus-H}
f(\widehat H_h)
=
f(\Op(H))
+
\mathcal O_{L^2\to L^2}(h)
\quad \text{and} \quad 
\Op(f\circ a_h)
=
f(\Op(H))
+
\mathcal O_{L^2\to L^2}(h).
\end{equation}
\end{corollary}

\begin{proof}
The first estimate is the standard semiclassical functional calculus for
elliptic self-adjoint pseudodifferential operators. In the nonmagnetic case,
it follows from the Helffer--Sj\"ostrand formula and the
parameter-dependent symbolic construction of the resolvent; see
Dimassi--Sj\"ostrand~\cite[Chapter~8]{DimassiSjostrand1999} and
Zworski~\cite[\S14.3.2, Theorem~14.9]{Zworski2012}. The same proof applies
here with the magnetic Weyl product \(\star_h\), using the magnetic
composition formula and the magnetic Calder\'on--Vaillancourt Theorem recalled
in Section~\ref{s:Magnetic field and quantization}. It gives
\[
f(\widehat H_h)
=
\Op(f\circ a_h)
+
h\,\Op(r_h),
\]
where \((r_h)_{0<h\le h_0}\) is bounded in \(\mathscr{S}^0(\T^2\times\R^2)\). Hence,
by the magnetic Calder\'on--Vaillancourt theorem,
\[
f(\widehat H_h)
=
\Op(f\circ a_h)
+
\mathcal O_{L^2\to L^2}(h).
\]
This proves \eqref{e:functional-calculus-a_h}.

It remains to compare \(f\circ a_h\) with \(f\circ H\). Since
\(a_h=H+hR_h\), Taylor's formula gives
\[
f(a_h(x,\xi))-f(H(\xi))
=
hR_h(x,\xi)
\int_0^1
f'\big(H(\xi)+shR_h(x,\xi)\big)\,\dd s.
\]
Set
\[
q_h(x,\xi)
:=
R_h(x,\xi)
\int_0^1
f'\big(H(\xi)+shR_h(x,\xi)\big)\,\dd s.
\]
Since \(H+s h R_h\) is elliptic of order \(m\), uniformly for \(s\in[0,1]\) and \(0<h\le h_0\), the condition \(H(\xi)+s hR_h(x,\xi)\in \operatorname{supp} f'\) forces \(\xi\) to remain in a fixed compact set. Hence all derivatives of \(f'(H+s hR_h)\) are uniformly bounded, and \(q_h\) is bounded in \(\mathscr S^0\) and
\[
f\circ a_h
=
f\circ H+hq_h.
\]
By the magnetic Calder\'on--Vaillancourt theorem,
\begin{equation}\label{e:compare-f-ah-f-H}
\Op(f\circ a_h)
=
\Op(f\circ H)
+
\mathcal O_{L^2\to L^2}(h).
\end{equation}

Applying the first part of the proof to the \(h\)-independent symbol \(H(\xi)\)
also gives
\[
f(\Op(H))
=
\Op(f\circ H)
+
\mathcal O_{L^2\to L^2}(h).
\]
Combining this estimate with
\eqref{e:functional-calculus-a_h} and \eqref{e:compare-f-ah-f-H} proves
\eqref{e:functional-calculus-H}.
\end{proof}

%-------------------------------------------------------------
\subsection{Disintegration of Radon measures}
\label{ss:disintegration}

We shall also use the following standard disintegration theorem. We recall it
for completeness in the form needed in the paper. We refer to
Baccelli--B{\l}aszczyszyn--Karray~\cite[Appendix~14.D, Theorem~14.D.10]{BaccelliBlaszczyszynKarray2020}
for a proof, and also to Kallenberg~\cite[Chapter~5]{Kallenberg2002} for a
classical account of conditioning and disintegration.

\begin{theorem}[Disintegration of measures]
\label{t:desintegration}
Let \(X\) and \(Y\) be locally compact Polish spaces and let \(W\) be a
positive Radon measure on \(X\times Y\). Assume that the marginal
\[
\omega:=(\pi_Y)_*W
\]
is \(\sigma\)-finite. Then there exists an \(\omega\)-measurable family
\[
(\nu_y)_{y\in Y}
\]
of probability measures on \(X\) such that, for every nonnegative Borel
function \(F:X\times Y\to[0,+\infty]\),
\begin{equation}\label{e:disintegration-formula}
\int_{X\times Y}F(x,y)\,W(\dd x,\dd y)
=
\int_Y
\left(
\int_X F(x,y)\,\nu_y(\dd x)
\right)
\omega(\dd y).
\end{equation}
Equivalently,
\begin{equation}\label{e:disintegration-measure}
W(\dd x,\dd y)=\nu_y(\dd x)\,\omega(\dd y).
\end{equation}
Moreover, the family \((\nu_y)_{y\in Y}\) is unique for
\(\omega\)-almost every \(y\).
\end{theorem}

\begin{remark}
The measurability of the family \((\nu_y)_{y\in Y}\) means that, for every
bounded Borel function \(\varphi:X\to\R\), the map
\[
y\longmapsto \int_X \varphi(x)\,\nu_y(\dd x)
\]
is Borel measurable. In the paper, the theorem is applied only to
finite-dimensional manifolds or Borel subsets thereof, so the above
assumptions are automatically satisfied.
\end{remark}

\end{document}